\documentclass[11pt]{article}
\usepackage[T1]{fontenc}
\usepackage{amsfonts}
\usepackage{amsmath}
\usepackage{amssymb}
\usepackage{amsthm}
\usepackage{bbm}
\usepackage{bm}
\usepackage{mathrsfs}
\usepackage{verbatim}
\usepackage{setspace}
\usepackage{color}
\usepackage{tabularx}
\usepackage{booktabs}
\usepackage[labelfont=bf,format=plain,justification=raggedright,singlelinecheck=false]{caption}
\usepackage{algorithm}
\usepackage[noend]{algpseudocode}

\usepackage{geometry}
\usepackage{pdfsync}
\usepackage{comment}
\usepackage{enumitem}
\usepackage{graphicx}
\usepackage{subfigure}
\usepackage{tikz}
\usetikzlibrary{patterns}

\usepackage[pdfborder={0 0 0}, bookmarks=false]{hyperref}
\hypersetup{
	urlcolor = black,
	pdfauthor = {},
	pdfkeywords = {},
	pdftitle = {},
	pdfsubject = {},
	pdfpagemode = UseNone
}

\theoremstyle{plain}
\newtheorem{theorem}{Theorem}[section]
\newtheorem{proposition}[theorem]{Proposition}
\newtheorem{lemma}[theorem]{Lemma}
\newtheorem{corollary}[theorem]{Corollary}

\theoremstyle{definition}
\newtheorem{definition}[theorem]{Definition}
\newtheorem{remark}[theorem]{Remark}

\newtheorem{example}[theorem]{Example}

\newtheorem{assumption}[theorem]{Assumption}

\theoremstyle{remark}

\newcommand\TV{{\rm TV}} 
\newcommand\W{\mathcal{W}} 

\renewenvironment{thebibliography}[1]{%
	\begin{oldthebibliography}{#1}%
		\setlength{\baselineskip}{1em}
		\linespread{.2}
		\small
		\setlength{\parskip}{0.25ex}%
		\setlength{\itemsep}{.20em}%
	}%
	{%
	\end{oldthebibliography}%
}

\newcommand{\R}{\mathbb{R}}

\newcommand{\X}{\mathcal{X}}

\makeatletter
\newcommand\avsuminner[2]{%
	{\sbox0{$\m@th#1\sum$}%
		\vphantom{\usebox0}%
		\ooalign{%
			\hidewidth
			\smash{\vrule height\dimexpr\ht0+1pt\relax depth\dimexpr\dp0+1pt\relax}%
			\hidewidth\cr
			$\m@th#1\sum$\cr
		}%
	}%
}
\makeatother
\newcommand{\1}{\mathbf{1}}

\newcommand{\mykill}[1]{}

\newcommand{\brak}[1]{\left(#1\right)}    %
\newcommand{\eins}{\1}%

\newcommand{\pa}{\text{\rm pa}}
\newcommand{\anc}{\text{\rm pre}}

\newcommand{\E}[1]{{\mathbb E}\left[#1\right]}

\newcommand{\abs}[1]{\left|#1\right|}     %
\numberwithin{equation}{section}

\def\XXint#1#2#3{{\setbox0=\hbox{$#1{#2#3}{\int}$ }
		\vcenter{\hbox{$#2#3$ }}\kern-.6\wd0}}

\begin{document}
	\title{\vspace{-2em}
        Fast Wasserstein rates for estimating probability distributions of probabilistic graphical models
        
	}
	\date{\today}

	\author{
		Daniel Bartl%
		\thanks{Department of Mathematics \& Departments of Statistics and Data Science, National University of Singapore, bartld@nus.edu.sg.}
			\and
		Stephan Eckstein%
		\thanks{Department of Mathematics, University of T\"{u}bingen, Germany, stephan.eckstein@uni-tuebingen.de.}
	}
	\maketitle \vspace{-1.2em}
	
	\begin{abstract}
		 Using i.i.d.~data to estimate a high-dimensional distribution in Wasserstein distance is a fundamental instance of the curse of dimensionality. We explore how structural knowledge about the data-generating process which gives rise to the distribution can be used to overcome this curse. More precisely, we work with the set of distributions of probabilistic graphical models for a given directed acyclic graph. It turns out that this knowledge is only helpful if it can be quantified, which we formalize via smoothness conditions on the transition kernels in the disintegration corresponding to the graph. In this case, we prove that the rate of estimation is governed by the local structure of the graph, more precisely by dimensions corresponding to single nodes together with their parent nodes. The precise rate depends on the exact notion of smoothness assumed for the kernels, where either weak (Wasserstein-Lipschitz) or strong (bidirectional Total-Variation-Lipschitz) conditions lead to different results. We prove sharpness under the strong condition and show that this condition covers, as a special case, distributions having a positive Lipschitz density.
	\end{abstract}
	
	\vspace{.3em}
	
	{\small

	}
	\vspace{.6em}

	\section{Introduction}\label{sec:intro}
	Overcoming the curse of dimensionality in high-dimensional learning settings usually requires inductive biases, i.e., some a priori assumptions on the kind of structures one tries to learn. One of the basic learning settings of this kind is non-parametric estimation of probability measures, which aims at learning the distribution of high-dimensional random variables without parametric assumptions (see, e.g., \cite{dudley1969speed,fournier2015rate,tsybakov}). Most approaches towards overcoming the curse of dimensionality in this setting have focused on imposing biases towards smoothness, often explicitly by working with distributions having smooth Lebesgue densities (see, e.g., \cite{niles2022minimax, tsybakov}) or also implicitly through the kind of distance used to measure the difference of the estimate from the truth (see, e.g., \cite{kloeckner2020empirical,singh2018nonparametric}). In this paper, we provide complementary results by focusing on biases related to the relational structure between the different variables of the distributions (cf.~\cite{battaglia2018relational}). More precisely, we focus on distributions of probabilistic graphical models (see, e.g., \cite{koller2009probabilistic,peters2017elements}) corresponding to a given graph such that the kernels occurring in the disintegration according to the graph are continuous in a suitable sense. With this setting, we aim to accomplish two things: First, establish conditions for large random systems which guarantee that the rate of estimation only depends on local parts of the system. And second, introduce smoothness criteria based on stochastic kernels instead of Lebesgue densities to cover settings with partly discrete variables as well.
	
	\subsection{Setting and summary of the main results}
	\subsubsection{General learning setting}

    Let $\X=[0,1]^d$, denote by $\mathcal{P}(\X)$ the set of probability measures on $\X$ and set $\W$  to be the first order Wasserstein distance on $\mathcal{P}(\X)$, defined by
    \[ \W(\mu,\nu) = \inf_\pi \int_{\X\times\X} \|x-y\| \,\pi(dx,dy), \]
    where the infimum is taken over all couplings $\pi$, i.e.\ measures $\pi$ with first marginal $\mu$ and second marginal $\nu$. Throughout, we use $\|\cdot\| = \|\cdot\|_\infty$, which is of course only relevant up to constants.
    We refer e.g.\ to \cite{graf2000foundations,villani2008optimal} for background on Wasserstein distances.

    We are interested in estimating a probability measure $\mu \in \mathcal{P}(\X)$ via $n$ i.i.d.~samples  $X^1,\dots,X^n$ selected according to $\mu$, that is, find an estimator $E_n : \X^n \rightarrow \mathcal{P}(\X)$ such that 
    \[\int \W(\mu, E_n) \,d\mu^{\otimes n} = \int \W(\mu, E_n(x^1, \dots, x^n)) \,\mu(dx^1) \ldots \mu(dx^n)
	\]
	is small simultaneously for many different distributions $\mu$ in a set $\mathcal{Q} \subseteq \mathcal{P}(\X)$. 
    Hence, we wish to solve
	\begin{equation}\label{eq:basiclearninggoal}
		V_Q(n) := \inf_{E_n} \sup_{\mu \in \mathcal{Q}} \int \W(\mu, E_n) \,d\mu^{\otimes n}.
	\end{equation}
	In the case where one does not impose any additional prior knowledge and thus works with $\mathcal{Q} = \mathcal{P}(\X)$ for $d\geq 3$, it is well known that $V_Q(n) \lesssim n^{-1/d}$ (see \cite{dudley1969speed} and also \cite{fournier2015rate}).
    Notably, these rates are attained by the empirical measure $E_n(x^1, \dots, x^n) = \frac{1}{n} \sum_{i=1}^n \delta_{x^i}$. 
    In this case, no estimator can do better, and so $V_Q(n) \gtrsim n^{-1/d}$ holds as well (see, e.g., \cite{chewi2024statistical}). 
    With additional structural assumptions, that is, when $\mathcal{Q} \subsetneq \mathcal{P}(\X)$, the empirical measure is usually suboptimal and other estimators must be used to obtain faster rates (cf.~\cite{niles2022minimax, tsybakov}).

	\subsubsection{Probabilistic graphical models}
	Throughout this paper, we always assume that a directed acyclic graph (DAG) $G$ with nodes $\{1, \dots, K\}$ is given (and for most of this work, known to the statistician) and is topologically sorted, which means that there are no edges from $i$ to $j$ for $j < i$. 
    The space $\X=[0,1]^d$ is partitioned into $\X=\X_1\times \dots\times \X_K$ where $\X_k=[0,1]^{d_k}$ and $\sum_{k=1}^K d_k=d$.
    For $x \in \X$, denote by $x_k \in \X_k$ the projection onto the $k$-th coordinate, and by $x_I$ the projection onto a subset of variables $I \subseteq \{1, \dots, K\}$. Further, denote by $\pa(k)$ the set of parent nodes of a node $k$. 
    Probabilistic graphical models for the graph $G$ are defined as
	\[
	\mathcal{P}_G := \left\{\mu \in \mathcal{P}(\X) \mid \mu(dx_1, \dots, dx_K) = \prod_{k=1}^K \mu(dx_k \mid x_{\pa(k)})\right\}.
	\]
	That is, when integrating the $k$-th variable in the disintegration of $\mu \in \mathcal{P}_G$, one only needs to condition on the parent variables of $k$ according to $G$. One simple example of a relevant graph is $1\rightarrow 2 \rightarrow \ldots \rightarrow K$, in which case $\mathcal{P}_G$ corresponds to distributions of Markov chains. 
  A slightly richer example is the abstract rooted tree depicted in
Figure~\ref{fig:rooted-tree}.  Such graphs can be viewed as simplified
tree-shaped dependence models for fully observed variables, as in
Chow--Liu-type approximations of multivariate distributions
\cite{chow1968}; related Bayesian-network models also appear e.g.\ in clinical
screening and diagnosis applications
\cite{friedman1997,ferreira2019,JensenNielsen2007}.
\begin{figure}[h]
\centering
\begin{tikzpicture}[
    >=stealth,
    var/.style={circle, draw, minimum size=7mm, inner sep=0pt}
]
    \node[var] (x1) at (0,0) {$X_1$};

    \node[var] (x2) at (-2,-1.2) {$X_2$};
    \node[var] (x3) at ( 2,-1.2) {$X_3$};

    \node[var] (x4) at (-3,-2.5) {$X_4$};
    \node[var] (x5) at (-1,-2.5) {$X_5$};
    \node[var] (x6) at ( 1,-2.5) {$X_6$};
    \node[var] (x7) at ( 3,-2.5) {$X_7$};

    \draw[->] (x1) -- (x2);
    \draw[->] (x1) -- (x3);
    \draw[->] (x2) -- (x4);
    \draw[->] (x2) -- (x5);
    \draw[->] (x3) -- (x6);
    \draw[->] (x3) -- (x7);
\end{tikzpicture}
\caption{A rooted tree DAG, motivated by tree-shaped dependence models for
fully observed variables.}
\label{fig:rooted-tree}
\end{figure}
Whenever $\pa(k)$ is empty, the conditional distribution is just understood as the marginal distribution. For instance since $\pa(1)=\emptyset$, $\mu(dx_1\mid x_{\pa(1)})$ is just the first marginal of $\mu$.
    We mention that $\mu \in \mathcal{P}_G(\X)$ can analogously be defined using conditional independences (see, e.g., \cite[Remark 3.2]{cheridito2025optimal}). 

    Probabilistic graphical models are also known as Bayesian networks and naturally related to Bayesian inference (cf.~\cite{heckerman1998tutorial}). They are more generally used to combine structural assumptions about data generating processes with probabilistic modelling tools. 
    This is important for expressing and inferring causal probabilistic relations, for instance in fields such as environmental modelling (cf.~\cite{marcot2019advances}), biology (cf.~\cite{laubach2021biologist}), or climate research (cf.~\cite{ebert2012causal}). In this regard, statistically optimal methods for causality inspired tools are considered a timely challenge (see, e.g., \cite[Section 3.5]{cinelli2025challenges}).
     Probabilistic graphical models are further used to bridge the gap between causality and machine learning (cf.~\cite{scholkopf2022causality}) and thus naturally occur in the study of modern machine learning architectures (see, e.g., \cite{nichani2024transformers,zevcevic2021relating}).
    We refer to \cite{koller2009probabilistic,pearl2009causality,peters2017elements} for more background on probabilistic graphical models.

	\subsubsection{Lower bounds without continuous kernels}
	The first natural idea is to explore the learning problem \eqref{eq:basiclearninggoal} with $Q = \mathcal{P}_G$. We find that this is, however, not a fruitful approach. Aside from trivial cases in which the graph $G$ has disconnected components and thus one can estimate those components separately, it is not obvious at all how the prior knowledge of $\mu \in \mathcal{P}_G$ is beneficial compared to $\mu \in \mathcal{P}(\X)$. To explain these difficulties, it might be helpful to think of $\mathcal{P}_G$ as an analogue of the set of all distributions having a Lebesgue density, but without any quantitative control on the smoothness of this density. With this point of view, it is natural that the prior knowledge of $\mu \in \mathcal{P}_G$ is statistically not helpful, similarly to how knowledge of the existence of a Lebesgue density alone is not helpful.
	
	And indeed, we establish in Section \ref{sec:lowerbounds} that for many graph structures, the set $\mathcal{P}_G$ is dense in $\mathcal{P}(\X)$ with respect to the weak convergence, and the learning problem \eqref{eq:basiclearninggoal} using $Q = \mathcal{P}_G$ still has  the rate $n^{-1/d}$. This involves all graphs which have only one root node, as for instance the Markovian graph $1 \rightarrow 2 \rightarrow \ldots \rightarrow K$, or any kind of tree (as e.g.\ in Figure \ref{fig:rooted-tree}), see Theorem \ref{thm:lowercond} and Corollary \ref{cor:graphcondind}. The reason is that those graph structures do not impose any kind of unconditional, but only conditional independences, which we show in Theorem \ref{thm:lowercond} to be, as a purely qualitative assumption, statistically useless. To complement this result, we also explore another graph structure which can be regarded as an extreme case in terms of imposing several unconditional independences, namely the graph only including the nodes $k \rightarrow K$ for $k \in \{1, \dots, K-1\}$. In this graph all nodes $\{1, \dots, K-1\}$ are independent, and we establish in Proposition \ref{lem:particulargraphlower} that the rate is again $n^{-1/d}$ in this case. 
	
	While we do not establish the lower bound $n^{-1/d}$ for all graphs having only one connected component, we believe the covered cases provide evidence that, for statistical purposes, one should \emph{quantify the compatibility} of a probability measure $\mu$ with a graph $G$ instead of merely working with $\mathcal{P}_G$.

	\subsubsection{Fast (and sharp) rates under continuous kernels}
	To quantify how well a probability measure is compatible with a graph $G$, we introduce Lipschitz continuity conditions on the stochastic kernels occurring in the definition of $\mathcal{P}_G$. 
    More precisely, we shall consider two different conditions, one where Lipschitz continuity of the kernels is formulated via the Wasserstein distance, and one via the total variation distance.
    
    In both cases, the construction for the estimators we use requires certain conditions on the graphical structures. To state this assumption, recall that a subset $J$ of the nodes of the graph $G$ is called fully connected if there is an edge $k \rightarrow \ell$ for all $k, \ell \in J$ with $k<\ell$.
    \begin{assumption}
    \label{ass:graph_struc}
        The graph $G$ contains no colliders, that is, for any $k \in \{1, \dots, K\}$, the set $\pa(k)$ is fully connected.
    \end{assumption}

    The rooted tree in Figure~\ref{fig:rooted-tree} satisfies
Assumption~\ref{ass:graph_struc}: every non-root vertex has exactly one
parent, and hence each parent set is trivially fully connected.  Thus Assumption \ref{ass:graph_struc} allows branching, but excludes unshielded colliders in which two
unrelated parents point into the same child.  
 If a DAG contains such a collider, one can add edges to the graph to make the parents connected (say, in Figure \ref{fig:rooted-tree}, if $X_6$ had parents $X_3$ and $X_2$, we could add an edge from $X_2$ to $X_3$ to make the graph compatible with Assumption \ref{ass:graph_struc}). 
This operation preserves compatibility of the underlying measure with the graphical structure. 
However, we will see below that more edges translate to a possibly worse rate of convergence of our statistical estimators, which is of course undesirable. In other words, Assumption \ref{ass:graph_struc} can be circumvented, albeit at the cost of a possibly worse rate of convergence. We explore this in more detail in Appendix \ref{app:graphass}.

    \vspace{0.5em}
    \noindent
    {\bf Kernels which are Wasserstein-Lipschitz.}
    The kernels corresponding to the disintegration of the graph are given by the maps 
	\[
	\X_{\pa(k)} \ni x_{\pa(k)} \mapsto \mu(dx_k \mid x_{\pa(k)}) \in \mathcal{P}(\X_k).
	\]
	The most natural approach to impose continuity for these maps is to use the Wasserstein distance on $\mathcal{P}(\X_k)$, leading to the following assumption.

    \begin{assumption}[with constant $L$]
    \label{ass:W.Lip}
      The measure $\mu\in\mathcal{P}_G(\X)$ is such that
        \begin{align*}
    			\W(\mu(dx_k \mid x_{\pa(k)}), \mu(dx_k \mid \tilde{x}_{\pa(k)})) 
                &\leq L \|x_{\pa(k)} - \tilde{x}_{\pa(k)}\|,
        \end{align*}
        for all $2\leq k \leq K$ and for all $x_{\pa(k)},\tilde{x}_{\pa(k)}\in\X_{\pa(k)}$.
    \end{assumption}

  This assumption is quite natural; we provide a simple example illustrating its scope in Section \ref{sec:W.vs.TV}.    
  To formulate the main result in the setting of Wasserstein-Lipschitz kernels, set  $d_{\pa(k)} := \sum_{\ell \in \pa(k)} d_\ell$ and define the local dimension $d_{\rm loc}$ as
    \begin{align}
    \label{eq:def.d.loc}
    d_{\rm loc}: = \max_{k=1, \dots, K} \left( \max\{2,d_k\} + d_{\pa(k)} \right).
    \end{align}

    The following showcases that for graphical models with Lipschitz kernels, the overall rate of estimation no longer depends on the overall dimension $d$, but on the local dimension $d_{\rm loc}$ instead.
    For instance, in the context of the example in Figure \ref{fig:rooted-tree} with $d_k=2$ for all $k$ where $d=14$ and $d_{\rm loc}=4$, this means that the rate improves from $n^{-1/14}$ to $n^{-1/4}$ (up to a logarithmic factor).

    \begin{theorem}
    \label{thm:main.W.informal}
        Assume $G$ satisfies Assumption \ref{ass:graph_struc}, fix $L > 0$ and denote by $\mathcal{Q}$ the set of measures $\mu \in \mathcal{P}_G$ which satisfy Assumption \ref{ass:W.Lip} with constant $L$. Then, there exists a constant $C$ depending only on $G$, $L$ and $d_{\rm loc}$ such that
        \[
    	   V_{\mathcal{Q}}(n) \leq C \, \max\{\log(n),1\} \, n^{-1/d_{\rm loc}}.
        \]
    \begin{proof}
        The result follows from Theorem \ref{thm:estWLip}.
    \end{proof}
    \end{theorem}
    
    Three comments are in order: First, the estimator used to obtain the given upper bound is simple and tractable,  see Definition \ref{def:mu.A}.
    Most importantly, it is still a discrete distribution and the computation involves no optimization, merely recombining samples in a suitable way, see Algorithm \ref{alg:hatmuA}. Second, the $\log(n)$-factor is actually only necessary if $d_{\rm loc}$ is attained at $d_k=2$, and the constant $C$ is explicitly tractable (arising mainly from Lemma \ref{lem:Wdecomp}).
    Finally, although the estimator in Theorem \ref{thm:main.W.informal} requires prior knowledge of the graph, this assumption can be removed. More precisely, using a cross-validation-type procedure between the estimators constructed in Theorem \ref{thm:main.TV.informal} for different graphs, one can construct a graph-adaptive estimator which achieves the same error rate even when the underlying graph \(G\) is unknown; see Theorem \ref{thm:estWLip.adaptoive}. This, however, comes at the cost of reduced computational tractability.

    While Theorem \ref{thm:main.W.informal} gives a simple way to exploit the graphical structure and leads to rates depending only on the local dimension, it is open whether the given continuity condition is used optimally by our estimator---that is, whether the given rates are sharp. Indeed, even in the simple case with two nodes and the graph $1 \rightarrow 2$, with $d_1=d_2=3$, we could not formally establish a matching lower bound on the rate. An experiment towards the question of sharpness is presented in Appendix \ref{app:lower.muA}.

    \vspace{0.5em}
    \noindent
    {\bf Kernels which are Total Variation-Lipschitz.}	
	To move towards faster rates which are sharp, we work under a stronger (yet, as we shall explain, natural) continuity assumption on the stochastic kernels.
    Denote by $\TV$ the total variation distance, that is, $\TV(\nu,\tilde{\nu}) = \sup_f (\int f \, d\nu - \int f\,d\tilde \nu)$ where the supremum is taken over all measurable functions $f$ satisfying $|f|\leq 1/2$.
    In the setting of this paper, we always have that the Wasserstein distance is upper bounded by the total variation distance, which leads to the following strengthening of Assumption \ref{ass:W.Lip}:
    To formulate it, we write $\anc(k) := \{1, \ldots, k-1\} \setminus \pa(k)$.

    \begin{assumption}[with constant $L$]
    \label{ass:TV.Lip}
       The measure $\mu\in\mathcal{P}_G(\X)$ is such that
        \begin{align*}
    			\TV(\mu(dx_k \mid x_{\pa(k)}), \mu(dx_k \mid \tilde{x}_{\pa(k)})) 
                &\leq L \|x_{\pa(k)} - \tilde{x}_{\pa(k)}\|,\\
                \TV(\mu(dx_{\anc(k)} \mid x_{\pa(k)}), \mu(dx_{\anc(k)} \mid \tilde{x}_{\pa(k)})) 
                &\leq L \|x_{\pa(k)} - \tilde{x}_{\pa(k)}\|,
        \end{align*}
        for all $2\leq k \leq K$, $x_{\pa(k)},\tilde{x}_{\pa(k)}\in\X_{\pa(k)}$ and, for all $x_k, \tilde{x}_k \in \X_k$,
        \begin{align*}
    			\TV(\mu(dx_{\pa(k)} \mid x_k), \mu(dx_{\pa(k)} \mid \tilde{x}_k)) 
                &\leq L \|x_k - \tilde{x}_k\|.
    		\end{align*}
    \end{assumption}

Intuitively, the first condition in  Assumption \ref{ass:TV.Lip} states that small changes in the cause (the parents $\pa(k)$) lead to small changes in the effect’s (the node $k$) distribution, while the other conditions mean that the distribution of the cause remains stable under varying observed effects. 
Though the second may seem less intuitive at first, it can be viewed as a form of stable Bayesian updating---a natural assumption. For instance for a Markov graph $1 \rightarrow 2 \rightarrow \ldots \rightarrow K$, one quickly checks that Assumption \ref{ass:TV.Lip} reduces to both the conditional distributions $x_{k-1}\mapsto \mu(dx_k \mid x_{k-1})$ and $x_{k} \mapsto \mu(dx_{k-1} \mid x_k)$ being Lipschitz.
We show in Lemma~\ref{lem:char.TV} that Assumption~\ref{ass:TV.Lip} admits an equivalent formulation in terms of the existence of a density, relative to a suitable product reference measure, whose sections are Lipschitz continuous in an integrated sense.

The following is the main result of this paper, showcasing the estimation under the strengthened Lipschitz condition. To this end, set $d_{\rm max}:=\max_{1\leq k\leq K} \max\{2,d_k\}$.
\begin{theorem}
    \label{thm:main.TV.informal}
    Assume $G$ satisfies Assumption~\ref{ass:graph_struc}, fix $L > 0$ and denote by $\mathcal{Q}$ the set of measures $\mu \in \mathcal{P}_G$ which satisfy Assumption~\ref{ass:TV.Lip} with constant $L$. Then, there exists a constant $C$ depending only on $G$, $L$ and $d_{\rm loc}$ such that
       \[
	   V_{\mathcal{Q}}(n) \leq C \cdot \max\{\log(n),1\} \left(n^{-2/(2+d_{\rm loc})} + n^{-1/d_{\rm max}} \right).
     \]
\begin{proof}
    The result follows from Theorem \ref{thm:estTV}.
\end{proof}
\end{theorem}

 For concreteness, note that in the context of the example in Figure~\ref{fig:rooted-tree} with $d_k=2$ for all $k$, the rate obtained in the theorem is $n^{-1/3}$.
We also emphasize again, as in Theorem~\ref{thm:main.W.informal}, that the $\log(n)$-factor is only needed in cases where $d_k=2$ leads to the dominant terms in $d_{\rm loc}$, and that the estimator achieving the given rate is highly tractable and discrete, see Definition~\ref{def:mu.ba} and Algorithm \ref{alg:hatmubA}.
 Moreover, a graph-adaptive estimator achieves the same rate when \(G\) is unknown, at the cost of reduced computational tractability; see Theorem~\ref{thm:main.TV.adaptive}.
    More importantly and in contrast to Theorem \ref{thm:main.W.informal}, the established rate in Theorem \ref{thm:main.TV.informal} is actually sharp! (At least up to the $\log(n)$ term.) 
    
\begin{proposition}\label{lem:sharp}
    In the setting of Theorem \ref{thm:main.TV.informal}: Suppose further that $d_{\rm loc}$ is attained for some $k$ satisfying $d_k \geq 2$  and that $L\geq 2$. Then, there exists an absolute constant $C > 0$ such that
    \[
	 V_{\mathcal{Q}}(n)
    \geq C  \left( n^{-2/(2+d_{\rm loc})} +  n^{-1/ d_{\rm max}} \right).
	\]
\end{proposition}

The proof of the proposition is given at the end of Section \ref{sec:tvlip}.
We emphasize that the term $n^{-1/d_{\rm max}}$ is clearly necessary, as no restrictions on the marginal distributions for each node are imposed. 
On the other hand, the sharpness of the term $n^{-2/(2+d_{\rm loc})}$ builds on lower bounds for density estimation under Lipschitz conditions. In this context, we emphasize that Theorem \ref{thm:main.TV.informal} is novel even for the graph $1\rightarrow 2$; that is, even without the focus on the graphical structure, but merely focusing on the smoothness aspect, the given result provides new conditions for sharp rates, see Remark \ref{rem:density.estim} and the preceding results for more details.

	\subsubsection{Structure of the paper} The remainder of the paper is structured as follows: We start by reviewing additional related literature. 
    In Section \ref{sec:wlip}, we work in the setting when (forward)-kernels are Lipschitz with respect to the Wasserstein distance.
    Section \ref{sec:wlip} also serves as a warm-up for Section \ref{sec:tvlip} which contains our main results, namely about sharp rates in the setting when (forward and backward)-kernels are Lipschitz with respect to the total variation distance.
    Section \ref{sec:lowerbounds} establishes the lower bounds for $\mathcal{P}_G$ without continuity assumptions on the kernels,  and the Appendix contains additional proofs and two numerical examples.

	\subsection{Related Literature}
	Upper bounds for empirical measures for the $p$-th order Wasserstein distance are a classical topic in probability theory and statistics and are established for instance in \cite{dudley1969speed,fournier2015rate}. A general approach to establish lower bounds is given in \cite{tsybakov} and using smoothness of densities to improve Wasserstein estimation rates is established in \cite{niles2022minimax}.
	
	Another line of work focuses on using weaker notions of distances which do not exhibit the curse of dimensionality, for instance through integral probability metrics under smooth test functions (e.g., \cite{kloeckner2020empirical,singh2018nonparametric}; IPMs notably include the Sinkhorn divergence, cf.~\cite{genevay2019sample}), low-dimensional projections (e.g., \cite{bartl2025structure,manole2022minimax,nietert2022statistical,rodriguez2025improved}) or smoothed versions of Wasserstein distances (e.g., \cite{goldfeld2020convergence}). We refer to \cite[Sections 2.7 and 2.8]{chewi2024statistical} for a detailed overview.
	
	The recent works \cite{bos2024supervised,chakraborty2026generalization,fan2025optimal,gyorfi2022tree,kumar2024likelihood,kwon2026nonparametric,vandermeulen2024breaking} also focus on improved estimation rates of distributions under structural assumptions. They explore estimating smooth densities under additional decomposition assumptions on the densities, and thus the goal of reducing to a local notion of complexity is the same as in this paper.
    The considered smoothness assumptions based on Lebesgue densities are however very different (see Remark \ref{rem:density.estim}).
    On top of the motivations from the current paper, this literature focuses on estimators related to generative networks (cf.~\cite{bos2024supervised,fan2025optimal,kumar2024likelihood,kwon2026nonparametric}) and often specifically in relation to diffusion models, which are hence closely related to practical pipelines using neural networks. In this regard, the adaptivity of the estimators to the graph structure can automatically be achieved by optimizing over suitably sparse neural networks (which is a hard, but practically well understood problem), which builds on \cite{10.1214/19-AOS1875}.

	In \cite{backhoff2022estimating}, the authors study statistical estimation under a stronger notion of Wasserstein distance focusing on differences between stochastic processes by comparing kernels forward in time. A similar goal of learning conditional distributions is pursued in \cite{acciaio2024convergence,benezet2024learning}. 
    The techniques in Section \ref{sec:wlip} (excluding the ones used to prove the graph-adaptive estimator in Theorem \ref{thm:estWLip.adaptoive}) build on the ones used in \cite{backhoff2022estimating}, in particular a similar result to Theorem \ref{thm:main.W.informal} in the particular case where the graph arises from a Markov-chain is given by \cite[Theorem 6.1]{backhoff2022estimating}. 
    It is important to note, however, that the proofs of our main results in Section \ref{sec:tvlip} do not follow from \cite{backhoff2022estimating} and require genuine new ideas and techniques, cf.~Remark \ref{rem:sec23}.

We also mention a related but distinct line of work on faster Wasserstein convergence of empirical measures under lower intrinsic-dimensional assumptions on the target distribution; see, for instance, \cite{weed2019sharp,ledoux2019optimal,hundrieser2024empirical}. 
Although the motivation is similar---namely to replace ambient-dimensional rates by rates governed by a smaller notion of complexity---the mechanism in our paper is different, as it relies on graphical structure and regularity of conditional kernels.
Currently we don't know to what extent these techniques can be combined, but believe it would be interesting, and leave it for future research.

	A string of literature with a different, but related, goal is the one focusing on establishing whether a given probability measure $\mu$ satisfies certain conditional independences (see, e.g., \cite{azadkia2021simple,neykov2021minimax,shah2020hardness}). Similarly to our paper, it turns out that this task is generally impossible (see \cite{shah2020hardness}), but becomes possible under a priori smoothness conditions on the stochastic kernels involved (see \cite{azadkia2021simple,neykov2021minimax}).
	Beyond that, the recent works \cite{faller2024self,gresele2022causal,sani2023bounding} explore causal inference tasks using the technique of combining distributions of partly overlapping sets of variables of a graphical model. 
    In this regard, the estimators used in Sections \ref{sec:wlip} and \ref{sec:tvlip} are slightly related as they are also based on gluing together estimates from different parts of the graph. Also related is the task of simultaneously estimating a distribution with a tree structure and the corresponding tree, which in discrete settings can be accomplished by the Chow-Liu algorithm (see \cite{chow1968}).

    \subsection{Notation}
    \begin{itemize}\setlength\itemsep{-0.375mm}
        \item $\X = \X_1 \times \ldots \times \X_K$, where $\X_k = [0, 1]^{d_k}$ and $d = \sum_{k=1}^K d_k$
        \item $\|\cdot\|$ will always be the $\infty$-norm (on any $\mathbb{R}^l$ for $l \in \mathbb{N}$)
        \item $G$ is a directed, acyclic graph with nodes $\{1, \ldots, K\}$ that is topologically sorted, that is, all edges $i \rightarrow j$ satisfy $i < j$
        \item $\mathcal{P}(\X)$ is the set of Borel probability measures on $\X$ and $\mathcal{P}_G(\X) \subseteq \mathcal{P}(\X)$ the subset of probability measures $\mu$ with disintegration $\mu(dx) = \prod_{k=1}^{K} \mu(dx_k \mid x_{\pa(k)})$, where $\mu(dx_k \mid x_{\pa(k)})$ are the regular conditional distributions of the $k$-th variable given its parent variables
        \item $\W$ is the Wasserstein distance and $\TV$ the total variation distance
        \item For $I \subseteq \{1, \ldots, K\}$ and $x \in \X$, set $\X_I = \prod_{i \in I} \X_i$ and $x_I = (x_k)_{k \in I} \in \X_I$
        \item $\mathcal{A}$ always denotes a partition of $\X$ into cubes of side length $\delta_{\mathcal{A}}$ (to be specified in what follows), and $\mathcal{A}_I$ is the implied partition of $\X_I$, and similarly $\mathcal{A}_k$ the one of $\X_k$. We usually denote by $c$ the cells in $\mathcal{A}$, so $\stackrel{.}{\cup}_{c \in \mathcal{A}} c = \X$
        \item For $B\subset \X_{I}$ and $\nu\in\mathcal{P}(\X)$ we write $\nu(B):=\nu(\pi_I^{-1}(B))$ where $\pi_I\colon\X\to\X_I$ is the canonical projection
        \item For $c\in\mathcal{A}_I$ we set $\nu|_c(\cdot):=\nu_{I}(\cdot\mid c)\in\mathcal{P}(\X_I)$ if $\nu(c)>0$ and $\nu|_c(\cdot)=\delta_{m}$ else, where $m$ is the mid-point of $c$
        \item Denote by $\pa(k) \subseteq \{1, \dots, K\}$ the parent nodes of $k$ according to $G$, and we often write $x_{\pa(k)}, \X_{\pa(k)}, \mathcal{A}_{\pa(k)}$ etc.~for $I=\pa(k)$ as above
        \item For a map $\varphi$ and a probability measure $\mu$, we denote by $\varphi(\mu)$ the pushforward $\varphi(\mu)(A) = \mu(\varphi^{-1}(A))$
        \item We write $1:k$ for the set $\{1, \dots, k\}$ and correspondingly $x_{1:k}, \X_{1:k}$, etc.
        \item $\eins_A$ is the indicator function of a set $A$, so $\eins_A(x) = 1$ if $x \in A$ and $\eins_A(x) = 0$, else
        \item $\hat{\mu} = \frac{1}{n}\sum_{i=1}^n \delta_{X^i}$ is the empirical measure of $\mu$. 
        We suppress the dependence on $n$ for notational clarity, as lower indices denote placements in the graph 
    \end{itemize}

	\section{Upper bounds for graphical models with Wasserstein-Lipschitz kernels}\label{sec:wlip}

This section establishes faster rates for graphical models with transition kernels that are Lipschitz in Wasserstein distance.
In addition, we introduce  some of the notation and tools needed to analyse our estimators, laying the groundwork for the more involved analysis in Section~\ref{sec:tvlip}.

Let $\eta \in \mathbb{N}\cup\{0\}$ be some parameter that is specified in what follows and set $\mathcal{A}$ to be a partition of $\X$ into hypercubes of side length $\delta_{\mathcal{A}} = 2^{-\eta}$,  hence $|\mathcal{A}| = 2^{\eta d}$ and  $|\mathcal{A}_S| = 2^{\eta d_S}$ for $S\subseteq\{1,\dots,K\}$. 
For $x \in \X$, denote by $c(x) \subseteq \X$ the unique cell of the hypercube containing $x$; similarly $c_S(x_S)$ is the unique cell in $\mathcal{A}_S$ containing $x_S\in \X_S$. For the following, recall the convention that if $\pa(k)$ is empty, then $\mu(dx_k \mid x_{\pa(k)})$ is understood as the $k$-th marginal distribution. 

We start by defining a ``smoothing''-operation $\mathcal{P}(\X)\ni\nu\mapsto \nu^\mathcal{A} $  and our estimator.

\begin{algorithm}[t]
\caption{Construction of $\hat\mu^\mathcal{A}$ and sampling from it}
\label{alg:hatmuA}
\begin{algorithmic}[1]
\Statex \textbf{Input:} DAG $G$ on $\{1,\dots,K\}$ (topologically sorted); samples $X^1,\dots,X^n\in\X$; resolution $\eta\in\mathbb{N}$.
\Statex \textbf{Output:} kernels $(\hat{\mathcal{K}}_k)_{k=1}^K$ and a procedure \textsc{Sample}() returning a draw from $\hat\mu^\mathcal{A}$.
\Statex For $S\subseteq\{1,\dots,K\}$, let $\mathrm{ind}_S(x_S):=\bigl(\lfloor 2^\eta x_\ell\rfloor\wedge(2^\eta-1)\bigr)_{\ell\in S}$ index the cell of $\mathcal{A}_S$ containing $x_S$.
\For{$k=1,\dots,K$}
   \If{$\pa(k)=\emptyset$} $\hat{\mathcal{K}}_k \gets \tfrac{1}{n}\sum_{i=1}^n \delta_{X^i_k}$  \Comment{$\frac{1}{|S|}\sum_{x \in S}\delta_x$ can be stored as the list $S$.}
   \Else
      \State initialise $\hat{\mathcal{K}}_k, U_k$ as empty dictionaries with default $()$
      \For{$i=1,\dots,n$} append $X^i_k$ to $U_k[\mathrm{ind}_{\pa(k)}(X^i_{\pa(k)})]$
      \EndFor
      \ForAll{$j$ with $U_k[j]\neq()$} $\hat{\mathcal{K}}_k[j] \gets \tfrac{1}{|U_k[j]|}\sum_{x\in U_k[j]}\delta_x$ \Comment{list structure, allows duplicates}
      \EndFor
   \EndIf
\EndFor
\Procedure{Sample}{\,}
   \For{$k=1,\dots,K$}
      \If{$\pa(k)=\emptyset$} draw $X_k\sim\hat{\mathcal{K}}_k$
      \Else\ draw $X_k\sim\hat{\mathcal{K}}_k[\mathrm{ind}_{\pa(k)}(X_{\pa(k)})]$ \Comment{$\hat{\mathcal{K}}_k[\mathrm{ind}_{\pa(k)}(X_{\pa(k)})] \neq ()$ by Ass.~\ref{ass:graph_struc} \& Lem.~\ref{lem:muAwelldefined}}
      \EndIf
   \EndFor
   \State \Return $(X_1,\dots,X_K)$
\EndProcedure
\end{algorithmic}
\end{algorithm}

	\begin{definition}
    \label{def:mu.A}
		For every $\nu \in \mathcal{P}(\X)$ and $k=1,\dots,K$, set
        \[
		\nu^{\mathcal{A}}(dx_k \mid x_{\pa(k)}) :=
        \begin{cases}\int \nu(dx_k \mid \tilde{x}_{\pa(k)}) \,\nu_{|c_{\pa(k)}(x_{\pa(k)})}(d\tilde{x}_{\pa(k)}) &\text{if }\pa(k) \neq \emptyset, \\
        \nu(dx_k) &\text{else},
        \end{cases}
		\]
                and define $\nu^\mathcal{A}(dx) := \prod_{k=1}^K \nu^\mathcal{A}(dx_k \mid x_{\pa(k)})$.
                Moreover, define 
        $\hat{\mu}^\mathcal{A} := \left( \hat{\mu}\right)^{\mathcal{A}},$
    that is, $\hat{\mu}^\mathcal{A}$ is obtained by applying the operation $\nu\mapsto \nu^\mathcal{A}$ to the empirical measure $\nu=\hat{\mu}$.
	\end{definition}
    
    Since the kernels $\nu(dx_k \mid \tilde{x}_{\pa(k)})$ are only $\nu$-almost surely unique, it is a priori not obvious that $\nu^\mathcal{A}$ is well-defined; this is shown in  Lemma \ref{lem:muAwelldefined}.    
    Moreover, by definition, $\nu^\mathcal{A}$ is always compatible with the graph, that is $\nu^\mathcal{A}\in\mathcal{P}_G(\X)$---even if $\nu\notin\mathcal{P}_G(\X)$.
    In particular, $\hat\mu^\mathcal{A}\in \mathcal{P}_G(\X)$.
    An algorithmic description to obtain $\hat\mu^{\mathcal{A}}$ is provided in Algorithm \ref{alg:hatmuA}. We note that $\hat\mu^{\mathcal{A}}$ remains a discrete measure, but typically has more than $n$ atoms. For the simple graph $1\rightarrow 2$, we refer to Figure~\ref{fig:mu.bA} (middle) in Section \ref{sec:tvlip} for an illustration of $\hat\mu^{\mathcal{A}}$.

The following is the main result of this section for statistical estimation using Wasserstein-Lipschitz kernels. 
Recall $d_{\rm loc} = \max_{1\leq k \leq K} (\max\{2, d_k\} + d_{\pa(k)})$.
    
	\begin{theorem}\label{thm:estWLip} 
    Suppose that  Assumptions \ref{ass:graph_struc} and \ref{ass:W.Lip} are satisfied and set $\eta=\lfloor \frac{1}{d_{\rm loc}} \log_2(n)\rfloor$.
    Then, the estimator $\hat\mu^\mathcal{A}$ (with the current choice of $\eta$) satisfies that
		\begin{align*}
		\mathbb{E}\left[ \W(\mu, \hat{\mu}^{\mathcal{A}}) \right] 
        &\leq C \cdot l_n \cdot n^{-1/d_{\rm loc}},
		\end{align*}
        where $C$ is a constant that depends only on $G, L, d_{\rm loc}$ and $l_n = 
        \max\{1,\log(n)\}$ if there is a node $k$ with $d_k=2$ and $d_{\rm loc} = d_k + d_{\pa(k)}$ and $l_n=1$ otherwise.
    \end{theorem}

In fact, the proof of Theorem \ref{thm:estWLip} gives the constants more explicitly as
\[C \cdot l_n = \sum_{k=1}^K M_{L, k} \begin{cases}
2L + 8 & \text{if } d_k\neq 2,\\
2L + 16 d_{\rm loc} & \text{if } d_k= 2 \text{ and } d_{\rm loc} > d_{\pa(k)} + d_k,\\
2L + 8\max\{\log(n),1\} & \text{if } d_k= 2 \text{ and } d_{\rm loc} = d_{\pa(k)} + d_k,
\end{cases}
\]
where $M_{L, k}  = 1 + \sum_{\ell=1}^K a_{k, \ell}\,L^\ell$ and $a_{k, \ell}$ is the number of paths of length $\ell$ going away from node $k$ in the direction of edges of $G$ (see Figure \ref{fig:constants} in Appendix \ref{sec:proofs.W} for an exemplification of $a_{k,\ell}$).
    
We emphasize that in Theorem \ref{thm:estWLip}, it remains an open question whether the derived rates are sharp. In fact, even for the simplest directed graph with two nodes, $1 \rightarrow 2$, the rate provided by Theorem \ref{thm:estWLip} coincides with the classical $n^{-1/d}$ rate, which means in this case the assumption of Wasserstein-Lipschitz kernels was not helpful.
In this context, we provide in Appendix \ref{app:lower.muA} a  numerical experiment suggesting that the rate $n^{-1/d}$ indeed cannot be improved in general.
In Section \ref{sec:tvlip} we will construct a refinement of the estimator $\mu^\mathcal{A}$ that achieves optimal rates (under a stronger version of Assumption \ref{ass:W.Lip}).

The construction in Theorem \ref{thm:estWLip} requires knowledge of the graph $G$. 
Ignoring computational aspects, one can show that a  cross-validation-type aggregation over candidate graphs (see, e.g., \cite[Section 6]{devroye2001combinatorial} for related methods) achieves the same rate uniformly over $G$. 
To formulate the result, set
\[ \mathcal{G}:=\{ G : G \text{ is a directed acyclic graph satisfying Assumption \ref{ass:graph_struc}}\}\]
and let $d_{\rm loc}(G)$ denote the local dimension associated with the graph $G$.

\begin{theorem}
\label{thm:estWLip.adaptoive}
There exists an estimator $\hat\mu^{ \mathcal{A},ad}$, depending only on the data $X^1,\dots,X^n$ and on $(d_k)_{k=1}^K$, such that for every $G\in\mathcal{G}$ and every $\mu \in \mathcal{P}_G$ satisfying Assumption \ref{ass:W.Lip} with constant $L$,
\begin{align*}
    \E{\W(\mu, \hat\mu^{ \mathcal{A},ad})} \;\leq\; C \cdot \max\{1,\log(n)\} \cdot n^{-1/d_{\rm loc}(G)},
\end{align*}
where $C$ is a constant depending only on $L$, $K$, and $(d_k)_{k=1}^K$.
\end{theorem}


\subsection{Main ingredients in the proof of Theorem \ref{thm:estWLip}}

The idea behind the proof of Theorem~\ref{thm:estWLip} is the following. 
The estimator $\hat\mu^{\mathcal A}$ replaces the empirical conditional kernels by averages over small parent cells. 
If the true kernels are Wasserstein-Lipschitz, this recombination introduces only a bias of order $\delta_{\mathcal A}$. 
The remaining stochastic error is then local: for each node $k$, one has to estimate $d_k$-dimensional conditional distributions on roughly $\delta_{\mathcal A}^{-d_{\pa(k)}}$ parent cells. 
This, at least heuristically, explains the appearance of the local dimension $    d_{\pa(k)}+\max\{2,d_k\}.$

The details of the proof follow arguments similar to those in \cite{backhoff2022estimating}. 
There are, however, two points specific to the present graphical setting. 
First, the constants $M_{L,k}$ require slightly more careful bookkeeping, since errors propagate along all directed paths starting at node $k$. 
Second, the operation $\nu\mapsto\nu^{\mathcal A}$ has to be shown to be well-defined, i.e.\ independent of the choice of versions of the kernels $\nu(dx_k\mid \tilde x_{\pa(k)})$ on $\nu$-null sets. 
We give this well-definedness argument in the main text, only sketch the proof of Theorem~\ref{thm:estWLip}, and defer the details to Appendix \ref{sec:proofs.W}; leaving the main text for  the substantially different proof of our main result, Theorem~\ref{thm:estTV}, which is given in Section~\ref{sec:tvlip}.
We also note that the proof of Theorem \ref{thm:estWLip.adaptoive} (presented in Appendix \ref{sec:proofs.W}) does not follow from \cite{backhoff2022estimating}.


The following Lemma shows the importance of Assumption \ref{ass:graph_struc}: It states that under the given graphical assumption, our estimator really only performs a \emph{local} smoothing, but the global distribution of the mass (that is, the mass given to each cell), remains unchanged.

      \begin{lemma}
      \label{lem:muAwelldefined}
        If Assumption \ref{ass:graph_struc} is satisfied, $\nu^{\mathcal{A}}$ does not depend on the particular choice of the kernels of $\nu$.
        Moreover, for every fully connected set\footnote{Recall that this means there is an edge between any two nodes in the set.}  $I \subseteq \{1, \dots, K\}$ and for all $c_I \in \mathcal{A}_{I}$, we have that $\nu^{\mathcal{A}}(c_I)=\nu(c_I)$.
        \begin{proof}
         We start with a supplementary observation: For every $k$ and $c_{\pa(k)} \in \mathcal{A}_{\pa(k)}$ with $\nu(c_{\pa(k)}) > 0$, and any Borel set $B\subset \X_k$,  regardless of the particular choice of the kernels,
            \begin{align}
            \label{eq:nu.A.kernels.identity}
            \begin{split}
                \nu^{\mathcal{A}}(B \mid c_{\pa(k)}) &
            = \int_{c_{\pa(k)}} \nu(B \mid \tilde{x}_{\pa(k)}) \frac{\nu(d\tilde{x}_{\pa(k)})}{ \nu(c_{\pa(k)}) }\\ &= \frac{\nu(c_{\pa(k)} \times B)}{\nu(c_{\pa(k)})} = \nu(B \mid c_{\pa(k)}).
              \end{split}
            \end{align}
        
            We now prove the second claim via induction. 
            Specifically, we show that for each $k = 1, \dots, K$, and for every fully connected subset $I \subseteq \{1, \dots, k\}$, it holds that $\nu(c_I) = \nu^{\mathcal{A}}(c_I)$.   
            For the base case $k = 1$, this is immediate since $\nu^\mathcal{A}_1 = \nu_1$. 
            For the induction step from $k - 1$ to $k$, let $I \subseteq \{1, \dots, k\}$ be a fully connected set; we may assume that $k\in I$ because otherwise there is nothing to show.
             Suppose first that $I= \pa(k) \cup \{k\}$ (which is fully connected thanks to Assumption \ref{ass:graph_struc}) and let $c_I\in\mathcal{A}_I$. We may assume $\nu(c_{\pa(k)}) > 0$, since otherwise, by induction hypothesis, $\nu^{\mathcal{A}}(c_{\pa(k)}) = \nu(c_{\pa(k)}) = 0$ holds and thus $\nu^{\mathcal{A}}(c_I) = 0 = \nu(c_I)$.
            By the induction hypothesis and \eqref{eq:nu.A.kernels.identity},
    \begin{align}
    \label{eq:nu.A.cells}
    \begin{split}
    \nu^{\mathcal{A}}(c_{I}) 
    =  \nu^{\mathcal{A}}( c_{\pa(k)}  \times c_k) 
    &=  \nu^{\mathcal{A}}(c_{\pa(k)})  \nu^{\mathcal{A}}(c_k \mid c_{\pa(k)})\\
    &= \nu(c_{\pa(k)}) \nu(c_k \mid c_{\pa(k)}) 
    = \nu(c_{I}).
    \end{split}
    \end{align}
    For a general fully connected set $I$, note that, by Assumption \ref{ass:graph_struc}, $I\subset J:=\pa(k) \cup \{k\}$ as otherwise an edge to $k$ must be missing, and let $c_I\in\mathcal{A}_I$.
    For every $c_{J \setminus I}\in\mathcal{A}_{J\setminus I}$, \eqref{eq:nu.A.cells} implies that  $\nu^{\mathcal{A}}(c_I\times c_{J\setminus I})=\nu(c_I\times c_{J\setminus I})$.
    Thus, since by definition $\nu(c_I) = \nu(c_I\times \X_{J\setminus I})$ and similarly for $\nu^\mathcal{A}$, the claim follows by taking the union over all $c_{J \setminus I}\in\mathcal{A}_{J\setminus I}$.

     It remains to show that $\nu^{\mathcal{A}}$ does not depend on the particular choice of the kernels of $\nu$.
     To that end, for every $c_{\pa(k)}$ for which $\nu(c_{\pa(k)})>0$,  and for $x_{\pa(k)}\in c_{\pa(k)}$, $\nu^\mathcal{A}(dx_k\mid x_{\pa(k)})$ does not depend on the particular choice of disintegration $\nu(dx_k\mid \tilde{x}_{\pa(k)})$ because all disintegrations are $\nu$-almost surely equal. 
     Next, if  $c_{\pa(k)}$ satisfies $\nu(c_{\pa(k)})=0$, then by the first part we must always have $\nu^\mathcal{A}(c_{\pa(k)})=0$, hence $\nu^\mathcal{A}(dx_k\mid x_{\pa(k)})$ is irrelevant for $x_{\pa(k)}\in c_{\pa(k)}$.
     This completes the proof.
        \end{proof}
    \end{lemma}

The following lemma controls the global Wasserstein distance by local distances between transition kernels and is central to our rate estimates.

    \begin{lemma}\label{lem:Wdecomp}
		Let $\mu, \nu \in \mathcal{P}_G(\X)$ and assume that $\mu$ satisfies Assumption \ref{ass:W.Lip}.
        Then
		\[
		\W(\mu, \nu) \leq \int \sum_{k=1}^K M_{L, k} \W(\mu(\cdot \mid y_{\pa(k)}), \nu(\cdot \mid y_{\pa(k)})) \,\nu(dy),
		\]
		where $M_{L, k} = 1 + \sum_{\ell=1}^K a_{k, \ell}\,L^\ell$ and $a_{k, \ell}$ is the number of paths of length $\ell$ going away from node $k$ in the direction of edges of $G$.
    \end{lemma}

   \begin{proof}[Sketch of proof]
   We sketch the argument in the case of the graph $1\to2$; the general proof is given in Appendix~\ref{sec:proofs.W}. 
    For each $(x_1,y_1)$, choose an optimal coupling between 
$\mu(dx_2\mid x_1)$ and $\nu(dx_2\mid y_1)$, and concatenate it with $\pi \in \Pi(\mu_1, \nu_1)$. 
This gives a coupling of $\mu$ and $\nu$, hence
\[
\mathcal W(\mu,\nu)
\leq
\inf_{\pi\in\Pi(\mu_1,\nu_1)}
\int
\Big(
\|x_1-y_1\|
+
\mathcal W(\mu(\cdot\mid x_1),\nu(\cdot\mid y_1))
\Big)\,\pi(dx_1,dy_1).
\]
By the Wasserstein-Lipschitz assumption on the kernel of $\mu$,
\[
\mathcal W(\mu(\cdot\mid x_1),\nu(\cdot\mid y_1))
\leq
L\|x_1-y_1\|
+
\mathcal W(\mu(\cdot\mid y_1),\nu(\cdot\mid y_1)).
\]
Therefore,
\[
\mathcal W(\mu,\nu)
\leq
(1+L)\mathcal W(\mu_1,\nu_1)
+
\int
\mathcal W(\mu(\cdot\mid y_1),\nu(\cdot\mid y_1))\,\nu_1(dy_1),
\]
as required (since here $M_{L,1}=1+L$ and $M_{L,2}=1$).
\end{proof}

We also use the following rather intuitive observation, proved in Appendix~\ref{sec:proofs.W}.

    \begin{lemma}
        \label{lem:conditional.iid}
        Let $k\leq K$, let $m\leq n$, and let $c_{\pa(k)}\in \mathcal{A}_{\pa(k)}$ with $\mu^{\mathcal{A}}(c_{\pa(k)})>0$.
        Then, conditionally on the event $n\cdot \hat{\mu}^{\mathcal{A}}(c_{\pa(k)})=m$, the random probability measure $\hat{\mu}^{\mathcal{A}}(\cdot \mid c_{\pa(k)})$ has the same distribution as the empirical measure of $\mu^{\mathcal{A}}(\cdot \mid c_{\pa(k)})$ with sample size $m$.
    \end{lemma}

The final ingredient in the proof of Theorem \ref{thm:estWLip}  is a standard result in probability theory: 
if $\nu\in\mathcal{P}([0,1]^{r})$ and $\hat\nu$ denotes its empirical measure with sample size $m$, then 
    \begin{align}
        \label{eq:classical.convergence.wasserstein.in.proof}
        \mathbb{E}[\W(\nu,\hat\nu)] 
        \leq 
            8 l_m(r) m^{-1/\max\{r,2\}} , 
            \qquad l_m(r) = 
        \begin{cases}
           1 & \text{if } r\neq 2,\\
            \max\{\log(n),1\}  & \text{if } r= 2,
        \end{cases}
    \end{align}
    we refer to \cite{fournier2023convergence} for a version that quantifies the multiplicative constants explicitly.


    \begin{proof}[Sketch of proof of Theorem \ref{thm:estWLip}]
    We focus on the graph $1\to 2$ and $d_1,d_2\neq 2$ for simplicity, the full proof is given in Appendix \ref{sec:proofs.W}.
Using Lemma~\ref{lem:Wdecomp}, that $\mu^{\mathcal A}(\cdot\mid y_1)$ is the average of the Lipschitz kernel
$\mu(\cdot\mid x_1)$ over $x_1\in c_1(y_1)$,  followed by the fact that $\mu^\mathcal{A}(\cdot\mid x_1)$ and  $\hat\mu^\mathcal{A}(\cdot\mid x_1)$ are constant over $x_1\in c_1\in\mathcal{A}_1$,
\begin{align*}
\W(\mu,\hat\mu^{\mathcal A})
&\leq
(1+L)\W(\mu_1,\hat\mu_1)
+
\int
\W\bigl(\mu(\cdot\mid y_1),\hat\mu^{\mathcal A}(\cdot\mid y_1)\bigr)
\,\hat\mu_1(dy_1)\\
& \leq (1+L)\W(\mu_1,\hat\mu_1)
+
L\delta_{\mathcal A}
+
\sum_{c_1\in\mathcal A_1}
\hat\mu(c_1)
\W\bigl(\mu^{\mathcal A}(\cdot\mid c_1),
\hat\mu^{\mathcal A}(\cdot\mid c_1)\bigr)\\
&
=:({\rm I}) + L \delta_\mathcal{A} +  ({\rm II})
\end{align*}
The wanted estimate on $\mathbb{E}[({\rm I})]$ follows from \eqref{eq:classical.convergence.wasserstein.in.proof}. As for the term $({\rm II}),$ by Lemma \ref{lem:conditional.iid} and \eqref{eq:classical.convergence.wasserstein.in.proof}, 
\[
\mathbb{E}\left[
\W\bigl(\mu^{\mathcal A}(\cdot\mid c_1),
\hat\mu^{\mathcal A}(\cdot\mid c_1)\bigr)
\,\middle|\,
n\hat\mu(c_1)
\right]
\leq 8(n\hat\mu(c_1))^{-1/\bar d_2}.
\]
Using the tower property and Jensen's inequality,
\[
\mathbb{E}[({\rm II})]
\leq  8  \mathbb{E}\left[ \sum_{c_1\in\mathcal{A}_1}  \hat\mu(c_1) \cdot  (n\hat\mu(c_1))^{-1/\bar d_2} \right]
\leq 8 
\left(
\frac{n}{|\mathcal A_1|}
\right)^{-1/\bar d_2}.
\]
Next note that  $\delta_\mathcal{A} \sim n^{-1/d_{\rm loc}}$ and $|\mathcal{A}_1|=2^{\eta d_1} \leq n^{d_1/d_{\rm loc}}$.
Using this, one can calculate that $(n/|\mathcal{A}_1|)^{-1/\bar d_2} \lesssim n^{-1/d_{\rm loc}}$.
\end{proof}

	\section{Sharp rates for graphical models with TV-Lipschitz kernels}\label{sec:tvlip}

    This section contains the main results of this paper, namely we introduce an estimator $\hat{\mu}^{b\mathcal{A}}$ that achieves optimal error rates for distributions of graphical models.

    We start the section by defining the estimator in Definition \ref{def:mu.ba} and giving the main results, Theorem \ref{thm:estTV} and Theorem \ref{thm:main.TV.adaptive}. Before proceeding with the proof, we compare the Wasserstein-Lipschitz-assumption \ref{ass:W.Lip} with the Total-Variation-Lipschitz-assumption \ref{ass:TV.Lip} in Section \ref{sec:W.vs.TV}. The proofs of the main results are in Sections \ref{subsec:propertiesmuba}--\ref{sec:tv.adaptive}.

To define the estimator,
recall that $\nu_{k}$ is the $k$-th marginal of $\nu$ and that for $c_k\subset \X_k$, $\nu_{k|c_k}(A_k) = \frac{\nu_k(A_k \cap c_k)}{\nu_k(c_k)}$ for Borel sets $A_k \subseteq \X_k$ if $\nu_k(c_k) > 0$ (if $\nu_k(c_k)=0$, we arbitrarily set $\nu_{k|c_k}=\delta_m$ where $m$ is the midpoint of the cell $c_k$).
Further,  $\mathcal{A}$ is the partition of $\X$ into cubes of side-length $\delta_\mathcal{A}=2^{-\eta}$.

Similarly to Definition \ref{def:mu.A}, define a smoothing-operation $\mathcal{P}(\X)\ni \nu\mapsto\nu^{b\mathcal{A}}$:

\begin{definition}
\label{def:mu.ba}
    For every $\nu\in\mathcal{P}(\X)$ and $k=1,\dots,K$, set
    \[
	\nu^{b\mathcal{A}}(dx_k \mid x_{\pa(k)}) 
    := \sum_{c_k \in \mathcal{A}_k} \nu\left(c_k \mid c_{\pa(k)}(x_{\pa(k)}) \right) \cdot \nu_{k | c_k}(dx_k)
	\]
    and define $\nu^{b\mathcal{A}}(dx) = \prod_{k=1}^K \nu^{b\mathcal{A}}(dx_k \mid x_{\pa(k)})$.
    Moreover, define $
        \hat{\mu}^{b\mathcal{A}} := \left( \hat{\mu}\right)^{b\mathcal{A}}$,  that is, $\hat{\mu}^{b\mathcal{A}}$ is obtained by applying the operation $\nu\mapsto \nu^{b\mathcal{A}}$ to the empirical measure $\nu=\hat{\mu}$.
\end{definition}

Under Assumption \ref{ass:graph_struc} and similarly to Lemma \ref{lem:muAwelldefined}, $\nu^{b\mathcal{A}}$ is indeed well-defined (i.e., it is the same for all representatives of the disintegration of $\nu$), which follows from Lemma \ref{lem:mubAwelldefined} below. An algorithmic description for $\hat\mu^{b\mathcal{A}}$ is given in Algorithm \ref{alg:hatmubA} and a visualization in the case of the simple graph $1\rightarrow 2$ is given in Figure \ref{fig:mu.bA}.

\begin{algorithm}[t]
\caption{Construction of $\hat\mu^{b\mathcal{A}}$ and sampling from it}
\label{alg:hatmubA}
\begin{algorithmic}[1]
\Statex \textbf{Input:} DAG $G$ on $\{1,\dots,K\}$ (topologically sorted); samples $X^1,\dots,X^n\in\X$; resolution $\eta\in\mathbb{N}$.
\Statex \textbf{Output:} cell-level kernels $(\hat{\mathcal{K}}_k)_{k=1}^K$, points for cell-marginals $(V_k)_{k=1}^K$, and a procedure \textsc{Sample}() returning a draw from $\hat\mu^{b\mathcal{A}}$.
\Statex For $S\subseteq\{1,\dots,K\}$, let $\mathrm{ind}_S(x_S):=\bigl(\lfloor 2^\eta x_\ell\rfloor\wedge(2^\eta-1)\bigr)_{\ell\in S}$ index the cell of $\mathcal{A}_S$ containing $x_S$.
\For{$k=1,\dots,K$}
   \State initialise $V_k$ as an empty dictionary with default $()$
   \For{$i=1,\dots,n$} append $X^i_k$ to $V_k[\mathrm{ind}_k(X^i_k)]$
   \EndFor
   \If{$\pa(k)=\emptyset$} $\hat{\mathcal{K}}_k \gets \tfrac{1}{n}\sum_{i=1}^n \delta_{\mathrm{ind}_k(X^i_k)}$
   \Else
      \State initialise $\hat{\mathcal{K}}_k$ as an empty dictionary with default $()$
      \For{$i=1,\dots,n$} append $\mathrm{ind}_k(X^i_k)$ to $\hat{\mathcal{K}}_k[\mathrm{ind}_{\pa(k)}(X^i_{\pa(k)})]$
      \EndFor
      \ForAll{$j$ with $\hat{\mathcal{K}}_k[j]\neq()$} $\hat{\mathcal{K}}_k[j] \gets \tfrac{1}{|\hat{\mathcal{K}}_k[j]|}\sum_{j_k\in\hat{\mathcal{K}}_k[j]}\delta_{j_k}$ \Comment{list structure, allows duplicates}
      \EndFor
   \EndIf
\EndFor
\Procedure{Sample}{\,}
   \For{$k=1,\dots,K$}
      \If{$\pa(k)=\emptyset$} draw $j_k\sim\hat{\mathcal{K}}_k$
      \Else\ draw $j_k\sim\hat{\mathcal{K}}_k[\mathrm{ind}_{\pa(k)}(X_{\pa(k)})]$ \Comment{$\hat{\mathcal{K}}_k[\mathrm{ind}_{\pa(k)}(X_{\pa(k)})] \neq ()$ by Ass.~\ref{ass:graph_struc} \& Lem.~\ref{lem:mubAwelldefined}}
      \EndIf
      \State draw $X_k$ uniformly from $V_k[j_k]$ \Comment{i.e., $X_k\sim\hat\mu_k|_{c_k}$ for $c_k$ with index $j_k$}
   \EndFor
   \State \Return $(X_1,\dots,X_K)$
\EndProcedure
\end{algorithmic}
\end{algorithm}

At this point, it perhaps makes sense to intuitively clarify the concept behind $\nu^{b\mathcal{A}}$ for the case $K=2$ and the graph  $1 \rightarrow 2$, in which case
\begin{align}
\label{eq:mu.as.local.prod}
\nu^{b\mathcal{A}} 
    = \sum_{c_1\in\mathcal{A}_1, c_2\in\mathcal{A}_2} \nu(c_1\times c_2) \left( \nu_{1|c_1}\otimes \nu_{2|c_2}\right). 
\end{align}
Thus, $\nu^{b\mathcal{A}}$ is locally a product measure, which on an intuitive level introduces additional smoothness (compared to $\nu^\mathcal{A}$).

\begin{figure}
    \centering
    \includegraphics[width=0.9\linewidth]{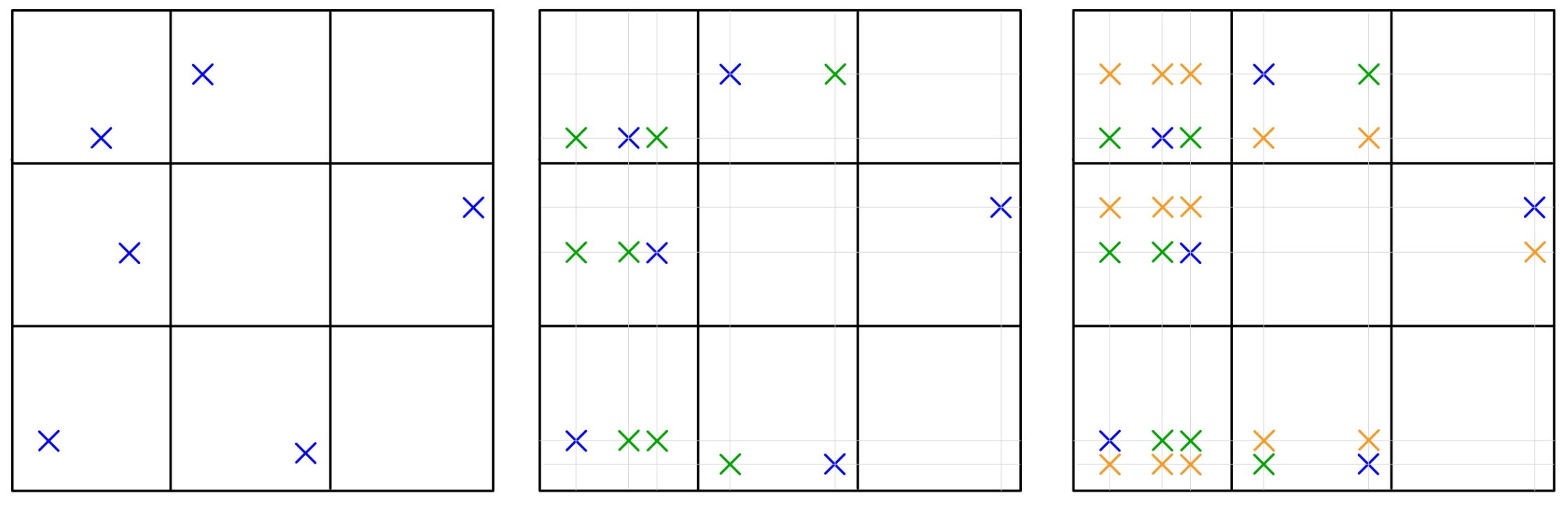}
    \caption{
 Visualization of the estimators $\hat\mu$ (left), $\hat\mu^\mathcal{A}$ (middle), and $\hat \mu^{b\mathcal{A}}$ (right) for the simple graph $1 \rightarrow 2$ with $\X_1 = \X_2 = [0, 1]$ with a partition of each interval into three subsets ($\X_1$ is on the horizontal axis) and sample size $n=6$.
    Blue crosses are the initial data points, green crosses are the new data points which are added in the construction of $\hat\mu^\mathcal{A}$ by making the kernels constant in the direction from first to second coordinate, and orange are the new points which are added in the construction of $\hat\mu^{b\mathcal{A}}$ by further making the kernels constant in the direction from second to first coordinate. Note that on the right, we have product measures locally on each cube.}
    \label{fig:mu.bA}
\end{figure}

We are now ready to state the main result of this paper: For this, we  recall $\bar{d}_k := \max\{2, d_k\}$, $d_{\rm max} = \max_{k=1, \dots, K} \bar{d}_k$, and $d_{\rm loc} = \max_{k=1, \dots, K} (\bar{d}_k + d_{\pa(k)})$.

\begin{theorem}\label{thm:estTV}
	Let $\mu \in \mathcal{P}_G(\X)$ and suppose that Assumptions \ref{ass:graph_struc} and \ref{ass:TV.Lip} hold. Set  $\eta=\lfloor \frac{\log_2(n)}{2+d_{\rm loc}}\rfloor$.
    Then,
	\[
	\mathbb{E}\left[\W(\mu, \hat{\mu}^{b\mathcal{A}})\right] 
    \leq C\cdot l_n \cdot  \left( n^{-2/(2+d_{\rm loc})} +  n^{-1/ d_{\rm max}} \right) ,
	\]
    where $C$ is a constant only depending on $L$, $G$, and $(d_k)_{k=1}^K$, and where
    \[
    l_n = \begin{cases}
        \max\{1, \log(n)\} & \text{if $n^{-2/(2+d_{\rm loc})} \geq n^{-1/ d_{\rm max}}$ and $d_{\rm loc}$ is attained at $d_k=2$,}\\
        \max\{1, \log(n)\} & \text{if $n^{-2/(2+d_{\rm loc})} \leq n^{-1/ d_{\rm max}}$ and $\max_{k=1, \dots, K} d_k = 2$,} \\
        1 & \text{otherwise.}
    \end{cases}
    \]
\end{theorem}

Note that if $d_k\geq 3$ for all $k$, then $l_n=1$.

Finally, just as in Section \ref{sec:wlip}, one may define an estimator that does not require knowledge of $G$ (the proof is given in Section \ref{sec:tv.adaptive}):

\begin{theorem}
\label{thm:main.TV.adaptive}
There exists an estimator $\hat\mu^{b\mathcal{A},ad}$, depending only on the data and on $(d_k)_{k=1}^K$, such that for every $G\in\mathcal{G}$ and every $\mu \in \mathcal{P}_G$ satisfying Assumption \ref{ass:TV.Lip} with constant $L$,
\begin{align*}
    \E{\W(\mu, \hat\mu^{b\mathcal{A},ad})} \;\leq\; C \cdot \max\{1,\log(n)\} \cdot \brak{n^{-2/(2+d_{\rm loc}(G))} + n^{-1/d_{\rm max}}},
\end{align*}
where $C$ is a constant depending only on $L$, $K$, and $(d_k)_{k=1}^K$.
\end{theorem}

\subsection{Comments on Assumption \ref{ass:W.Lip} and Assumption \ref{ass:TV.Lip}}
\label{sec:W.vs.TV}

We briefly illustrate the difference between Assumptions~\ref{ass:W.Lip} and~\ref{ass:TV.Lip}. 
The Wasserstein-Lipschitz condition is insensitive to singular changes of support, as long as these supports move in a Lipschitz way. 
By contrast, total variation is sensitive to singularity: two Dirac masses, or more generally two measures supported on disjoint lower-dimensional sets, have total variation distance one even if their supports are arbitrarily close.

\begin{example}
Assume that there is some Polish space $\mathcal{Z}$ and for every $k=1,\dots, K$,  there is a function $f_k\colon \X_{\pa(k)}\times \mathcal{Z}\to \X_k$ such that
\[
    X_k := f_k(X_{\pa(k)},Z_k) 
\]
where $Z_k$ is a $\mathcal{Z}$-valued random vector that is independent of $X_{\pa(k)}$.
If $f_k$ is $L$-Lipschitz in its first argument (that is $\|f_k(x_{\pa(k)},z)-f_k(\tilde x_{\pa(k)},z)\|\leq L \|x_{\pa(k)} - \tilde x_{\pa(k)} \|$ for every $z\in\mathcal{Z}$), then 
\begin{align*}
    \W\bigl(\mu(dx_k\mid x_{\pa(k)}),\mu(dx_k\mid \tilde x_{\pa(k)})\bigr)
    &\leq 
    \E{  \|f_k(x_{\pa(k)}, Z_k)-f_k(\tilde x_{\pa(k)}, Z_k)\|}\\
    &\leq
    L\|x_{\pa(k)}-\tilde x_{\pa(k)}\|.
\end{align*}
Thus the kernel is Wasserstein-Lipschitz, i.e.\ Assumption \ref{ass:W.Lip} is satisfied.

However, the kernels need not be Lipschitz in total variation.  For example, take $K=2$ and 
\(d_1=d_2=1\), \(f(x,z)=\frac{1}{2}x+z\), and
$Z$ uniform on $\{0,\frac{1}{2}\}$.  
Then, for
\(x_1\neq \tilde x_1\) sufficiently close, the two conditional laws have 
disjoint finite supports, and hence
\[
    \TV\bigl(\mu(dx_2\mid x_1),\mu(dx_2\mid \tilde x_1)\bigr)=1.
\]
Thus Assumption \ref{ass:TV.Lip} is not satisfied.
\end{example}

The example suggests that Assumption \ref{ass:W.Lip} is not that restrictive, and satisfied by various (both discrete and continuous) models.
At the same time, Assumption \ref{ass:TV.Lip} seems more delicate, and the reader might wonder about connections between this assumption and classical assumptions in density estimation, namely for $\mu$ to have a  strictly positive  density w.r.t.\ the Lebesgue measure that is Lipschitz continuous (or has higher order derivatives).
Before we explain those connections, let us present a characterization of Assumption \ref{ass:TV.Lip}; its proof is deferred to Section \ref{sec:remaining.proofs.TV}. 

\begin{lemma}
\label{lem:char.TV}
Let \(\mu\in\mathcal P(\X)\), and let
\(I,J\subseteq \{1,\dots,K\}\) be disjoint and non-empty and denote by
\(\mu_I,\mu_J,\mu_{I,J}\) the corresponding marginals of $\mu$.
The following are equivalent:
\begin{enumerate}
    \item[(i)] There exists a version of the conditional kernel
    \[
        x_J\mapsto \mu(dx_I\mid x_J)
    \]
    which is \(L\)-Lipschitz with respect to total variation.

    \item[(ii)] \(\mu_{I,J}\ll \mu_I\otimes\mu_J\), and the
    Radon--Nikodym derivative
    $
        f_{I,J}:=\frac{d\mu_{I,J}}{d(\mu_I\otimes\mu_J)}
    $
    admits a version such that
    \[
        x_J\mapsto f_{I,J}(\cdot,x_J)\in L^1(\mu_I)
    \]
    is \(2L\)-Lipschitz.
\end{enumerate}
\end{lemma}

\begin{corollary}
\label{cor:block.char.TV.graph}
Let \(\mu\in\mathcal P_G(\X)\). Suppose that for every \(k=2,\ldots,K\),
the conditions in Lemma~\ref{lem:char.TV} hold with constant \(L\)
for the ordered pairs
\begin{align}
    \label{eq:pairs.i.j}
    (I,J)=(\{k\},\pa(k)),\qquad
    (I,J)=(\anc(k),\pa(k)),\qquad
    (I,J)=(\pa(k),\{k\}),
\end{align}
whenever both \(I\) and \(J\) are non-empty. 
Then \(\mu\) satisfies Assumption~\ref{ass:TV.Lip} with constant $L$.

In particular, for $K=2$ and the simple graph $1\to 2$,  Assumption~\ref{ass:TV.Lip} is equivalent to: \(\mu\ll \mu_1\otimes\mu_2\), and the Radon--Nikodym
    derivative
    $
    f:=\frac{d\mu}{d(\mu_1\otimes\mu_2)}
    $
    admits  versions such that
    \[
         x_1\mapsto f(x_1,\cdot)\in L^1(\mu_2),
        \qquad
        x_2\mapsto f(\cdot,x_2)\in L^1(\mu_1)
    \]
     are \(2L\)-Lipschitz.
\end{corollary}

In particular, the transition kernels are required to be Lipschitz only in an integrated (rather than pointwise) sense, and the marginals are completely arbitrary (and in particular can have atoms or arbitrarily irregular densities).
A  consequence of Lemma \ref{lem:char.TV} is the following result (its proof is presented in Section \ref{sec:remaining.proofs.TV}):

\begin{corollary}
\label{cor:leb-density-tvlip}
Let \(I,J\subseteq \{1,\dots,K\}\) be disjoint and non-empty.
Assume that \(\mu_{I,J}\) has a Lebesgue density \(f_{I,J} =\frac{d\mu_{I,J}}{dx_{I,J}}\) and set $f_I =\frac{d\mu_{I}}{dx_{I}}$ and $f_J =\frac{d\mu_{J}}{dx_{J}}$.
If $f_J\geq a$ and $ \X_J\ni x_J\mapsto f_{I,J}(\cdot, x_J) \in L^1(x_I)$  is $M$-Lipschitz, then
$    x_J\mapsto \mu(dx_I\mid x_J)
$
is \(M/a\)-Lipschitz with respect to total variation.

Consequently, if the above condition holds for all non-empty pairs $I,J$ as in \eqref{eq:pairs.i.j}, then Assumption~\ref{ass:TV.Lip} holds with constant \(M/a\).
\end{corollary}

\begin{remark}
\label{rem:density.estim}
    Corollary \ref{cor:leb-density-tvlip} shows that Assumption \ref{ass:TV.Lip} is  weaker than the classical conditions used in density estimation (that the Lebesgue density is strictly positive and suitably smooth).
For instance,  let $K=2$, set $G$ to be the simple graph $1\to 2$, and put $d_1=d_2=2$.
Let $\mu$ have the Lebesgue density $f(x_1,x_2) = \alpha \eins_{x_1, x_2 \in [0, 1], |x_1-x_2| \leq 0.1}$. 
Then the conditions in Corollary \ref{cor:leb-density-tvlip} are satisfied, but $f$ is clearly neither strictly positive everywhere, nor pointwise Lipschitz continuous.

Moreover, it should be stressed that Assumption \ref{ass:TV.Lip} is intrinsically different to the assumptions imposed in the literature on factorizable density estimation (see, e.g., \cite{fan2025optimal,kwon2026nonparametric}).
This can most easily be seen in the trivial case \(K=1\), in which neither Assumption \ref{ass:W.Lip} nor Assumption \ref{ass:TV.Lip} impose any conditions on $\mu$.
Here our estimator recovers the standard non-parametric Wasserstein rate (e.g.\ \(n^{-1/d_1}\) for $d_1\geq 3$)  without imposing any regularity assumption on the measure.
In contrast, this setting falls outside the scope of density estimation theory.
\end{remark}

\subsection{Properties of $\mu^{b\mathcal{A}}$}\label{subsec:propertiesmuba}

\begin{remark}
\label{rem:sec23}
Before we proceed with the proof of Theorem \ref{thm:estTV}, let us explain the general strategy and compare the differences to the proof of Theorem \ref{thm:estWLip} (in particular, the strategy in Section \ref{sec:tvlip} deviates completely from prior work like \cite{backhoff2022estimating}).

In the approach of the proof for Theorem \ref{thm:estWLip}, the \emph{bias} introduced by the estimator $\hat\mu^{\mathcal{A}}$ enters on the level of conditional distributions, and is of order $\delta_{\mathcal{A}}$. In contrast, for the improved rate in Theorem \ref{thm:estTV}, we require a bias of order $\delta_{\mathcal{A}}^2$, which necessitates the assumption of Lipschitz continuity with respect to total variation instead of Wasserstein distance. Further, even under this stronger continuity assumption, a bias of order $\delta_{\mathcal{A}}^2$ cannot be achieved on the fine-grained scale of conditional distributions, hence requiring a bias analysis on the level of joint distributions. The key reason this does not work on the level of conditional distributions is that while our estimators only perturb mass locally on each cell (cf.~Lemma~\ref{lem:mubAwelldefined}), this is \emph{not} true for ``conditional cells''. This property of local perturbation is crucial for the rough approach ``$\W(\mu, \mu^{b\mathcal{A}})\lesssim \TV^2(\mu, \mu^{b\mathcal{A}})$'' leading to bias of order $\delta_{\mathcal{A}}^2$ in Lemma~\ref{lem:TVboundbias}, which doesn't work with the simple bias strategy on the level of conditional distributions. While controlling the bias in this fashion would still work for $\hat\mu^{\mathcal{A}}$, the strategy to bound the variance must be adjusted without this fine-grained bias control, and the new strategy no longer works using $\hat\mu^{\mathcal{A}}$, hence requiring the new estimator $\hat\mu^{b\mathcal{A}}$. The way the variance is controlled for this estimator is to project to the grid of cell midpoints, and this transition leads to the error of order $n^{-1/d_{\rm max}}$ in Theorem \ref{thm:estTV} and necessitates the local-product-structure given by $\hat\mu^{b\mathcal{A}}$. This transition requires a control of kernels forward and backward along the graph, which explains for instance why kernels of the form $x_{\pa(k)} \mapsto \mu(dx_{\anc(k)} \mid x_{\pa(k)})$ are important in Assumption \ref{ass:TV.Lip}. Once the transition to discrete measures is achieved, the remaining variance of the discrete measures can again be controlled along the topological order of the graph, similarly to Section 2; see Section~\ref{subsec:fullydiscrete} and Lemmas~\ref{lem:W.mu.M.hat.mu.M} and~\ref{lem:E.sup.fmu.minus.fhatmu}.
\end{remark}

\begin{lemma}\label{lem:mubAwelldefined}
    Let $\mu\in\mathcal{P}(\X)$, $2\leq k \leq K$ and $c_{\pa(k)}\in\mathcal{A}_{\pa(k)}$. 
    Then, for all $c_k\in \mathcal{A}_k$,
    \[ \mu^{b\mathcal{A}}(c_k \mid c_{\pa(k)}) 
    = \mu(c_k \mid c_{\pa(k)}).\]
    In particular, under Assumption \ref{ass:graph_struc}, for any fully connected part of the graph $I \subseteq \{1, \dots, K\}$ and any cell $c_I \in \mathcal{A}_I$, we have $\mu^{b\mathcal{A}}(c_I) = \mu(c_I)$. 
    \begin{proof}
    By definition of $\mu^{b\mathcal{A}}$,
	\begin{align*}
	    \mu^{b\mathcal{A}}(c_k \mid c_{\pa(k)}) 
        =\sum_{\tilde{c}_k \in \mathcal{A}_k} \mu( \tilde{c}_k \mid c_{\pa(k)}) \cdot \mu_{k | \tilde{c}_k}(c_k)
        =\mu( c_k \mid c_{\pa(k)}) 
	\end{align*}
	since $\mu_{k | \tilde{c}_k}(c_k)$ is equal to 1 if $c_k=\tilde{c}_k$ and zero otherwise.

    The second part of the claim works inductively by showing the claim for sets $I \subseteq \{1, \dots, k\}$ for increasing $k$. For $k=1$ the statement is clearly true. Regarding the induction step from $k-1$ to $k$, we only need to show the claim for all fully connected parts $I \subseteq \{1, \dots, k\}$ with $k \in I$. For $J = \pa(k) \cup \{k\}$, by the assumption that $\pa(k)$ is fully connected, this follows by the above since
    \begin{align*}
    \mu^{b\mathcal{A}}(c_J) &= \mu^{b\mathcal{A}}(c_k \times c_{\pa(k)})\\ &= \mu^{b\mathcal{A}}(c_{\pa(k)}) \mu^{b\mathcal{A}}(c_k \mid c_{\pa(k)}) = \mu(c_{\pa(k)}) \mu(c_k \mid c_{\pa(k)}) = \mu(c_J).
    \end{align*}
    For any other $I \subseteq \{1, \dots, k\}$ which is fully connected and with $k \in I$, we clearly have $I \subseteq J$ (otherwise an edge to $k$ must be missing), and hence
    \[
    \mu^{b\mathcal{A}}(c_I) = \sum_{c_{J \setminus I}} \mu^{b\mathcal{A}}(c_I \times c_{J \setminus I}) = \sum_{c_{J \setminus I}} \mu(c_I \times c_{J \setminus I}) = \mu(c_I).
   \qedhere \]
    \end{proof}
\end{lemma}

There is a subtle difference between $\mu^{\mathcal{A}}$ and $\mu^{b\mathcal{A}}$ (cf.~Lemma \ref{lem:muAwelldefined} compared to Lemma \ref{lem:mubAwelldefined}). For the former, we had $\mu^{\mathcal{A}}(dx_k \mid c_{\pa(k)}) = \mu(dx_k \mid c_{\pa(k)})$, while for the latter the equality holds only when restricted to cells in $\X_k$.
This has several consequences---for instance, Lemma \ref{lem:conditional.iid} no longer applies in this section and we need to suitably work around it, which is one of the objectives in Subsection \ref{subsec:fullydiscrete} below.

    The following clarifies the relation between $\mu^{\mathcal{A}}$ and $\mu^{b\mathcal{A}}$, which shows that the latter arises from a twofold (i.e., \emph{bi}directional) application of the $\nu\mapsto \nu^{\mathcal{A}}$ operation.
    \begin{lemma}
    \label{lem:mu.ba.euals.swap}
    Let $K=2$ and define 
    \[ \mathcal{S}\colon\X_1\times\X_2\to\X_2\times \X_1, \quad (x_1,x_2)\mapsto (x_2,x_1).\]
    Then, for every $\mu\in\mathcal{P}(\X)$,
        \[
       \mu^{b\mathcal{A}}
       =\mathcal{S}\Big( \big( \mathcal{S}(\mu^{\mathcal{A}})\big)^{\mathcal{A}}
        \Big) .
        \]
        \begin{proof} Note that for $K=2$, there are only two relevant graph structures, $1\rightarrow 2$ and the graph without edges. For the graph without edges, the statement is clearly satisfied since $\mu^{b\mathcal{A}} = \mu^{\mathcal{A}} = \mu_1 \otimes \mu_2$. We can thus restrict to the case $1 \rightarrow 2$.
        
            For notational simplicity, write $\mu= \nu\otimes R = \mathcal{S}(\theta \otimes V)$; in particular $\nu=\mu_1$ and $\theta=\mu_2$.

            \vspace{0.5em}
            \noindent
            \emph{Step 1:} We first claim that
            \begin{align}
                \label{eq:swap.swap.first.step}
                \mathcal{S}((\nu \otimes R)^{\mathcal{A}}) 
            = \theta \otimes V^{r\mathcal{A}},
            \end{align}  
            where 
            \begin{align*} 
            V^{r\mathcal{A}}(x_2, dx_1) 
            &= \int H(\tilde{x}_1, dx_1) \, V(x_2, d\tilde{x}_1), \qquad 
            H(\tilde{x}_1, dx_1) 
            = \sum_{\tilde c_1 \in \mathcal{A}_1} \eins_{\tilde c_1}(\tilde{x}_1)\nu_{|\tilde c_1}(dx_1).
            \end{align*}
            
            To show \eqref{eq:swap.swap.first.step}, it suffices to test it for Borel sets of the form $A \times B$ which satisfy $A \subseteq c_1$ and $B \subseteq c_2$ for some fixed $c_1 \in \mathcal{A}_1$ and $c_2 \in \mathcal{A}_2$. 
            Denote by $R^\mathcal{A}$ the kernel of $\mu^\mathcal{A}$, that is,  $\mu^\mathcal{A}=\nu\otimes R^\mathcal{A}$.
            Then,  for $x_1\in A$, we have $R^{\mathcal{A}}(x_1, B) = \int R(\tilde{x}_1, B) \nu_{|c_1}(d\tilde{x}_1)$ and thus
            \begin{equation}\label{eq:firstswap}
                (\nu \otimes R)^{\mathcal{A}}(A \times B) 
            = \nu(A) \int R(\tilde{x}_1, B)\,\nu_{|c_1}(d\tilde{x}_1).
            \end{equation}
            Moreover, since $H(\tilde x_1, A) = \eins_{c_1}(\tilde x_1) \nu_{|c_1}(A)$  (in particular it is zero for $\tilde x_1\notin c_1$) and $\nu_{|c_1}(A) = \nu(A)/\nu(c_1)$, 
            \begin{align*}
            \theta \otimes V^{r\mathcal{A}}(B \times A) &= \int_B \int H(\tilde x_1, A) \, V(x_2, d\tilde x_1) \,\theta(dx_2)\\
            &= \frac{\nu(A)}{\nu(c_1)} \int_B V(x_2, c_1) \,\theta(dx_2)\\
            &= \frac{\nu(A)}{\nu(c_1)} \int_{c_1} R(x_1, B) \,\nu(dx_1) = \nu(A) \int R(x_1, B) \,\nu_{|c_1}(dx_1).
            \end{align*}
            This readily shows \eqref{eq:swap.swap.first.step}.
            
              \vspace{0.5em}
            \noindent
            \emph{Step 2:} Consider $A$ and $B$ as above. 
            Note that by the definition of $\mathcal{S}$ and  \eqref{eq:swap.swap.first.step},
            \[ \mathcal{S}\Big( \big(\mathcal{S}((\nu \otimes R)^{\mathcal{A}}))^{\mathcal{A}} \Big)(A \times B) 
            = (\theta \otimes V^{r\mathcal{A}})^{\mathcal{A}}(B \times A).\]
            Moreover, by repeating the steps in \eqref{eq:firstswap} and using that $V^{r\mathcal{A}}(x_2,A)=\nu|_{c_1}(A) V(x_2,c_1)$, 
             \begin{align*}
            (\theta \otimes V^{r\mathcal{A}})^{\mathcal{A}}(B \times A)
            &=\theta(B) \int V^{r\mathcal{A}}(x_2, A) \,\theta_{|c_2}(dx_2) \\
            &= \theta(B) \nu_{|c_1}(A) \int V(x_2, c_1) \,\theta_{|c_2}(dx_2) \\
            &= \theta_{|c_2}(B) \nu_{|c_1}(A) (\theta \otimes V)(c_2 \times c_1).
            \end{align*}
            Finally, by  \eqref{eq:mu.as.local.prod}, the last term is equal to $\mu^{b\mathcal{A}}(A\times B)$, which is exactly what we needed to show.
        \end{proof}
    \end{lemma}

The next ingredient to the proof of Theorem \ref{thm:estTV} is to show that the equality $\nu(c_I) = \nu^{b\mathcal{A}}(c_I)$ from Lemma \ref{lem:mubAwelldefined} ``almost'' also holds for arbitrary cells $c \in \mathcal{A}$ under Assumption \ref{ass:TV.Lip}; we only require a small adjustment of order $\delta_\mathcal{A}^2$ (where $\delta_\mathcal{A}$ is, as always, the size of the cells in $\mathcal{A}$).

Before we proceed to state the result, let us spell out an observation that is important in what follows.
Firstly, if $\nu\in\mathcal{P}_G(\X)$, then by definition $\nu(dx_k \mid x_{1:k-1}) = \nu(dx_k \mid x_{\pa(k)})$.
This relation is no longer true for sets, e.g., it is no longer true that $\nu(c_k \mid c_{1:k-1}) $ is equal to $\nu(c_k \mid c_{\pa(k)})$ for cells $c\in\mathcal{A}$.
However, if  $\nu(c_k | x_{\pa(k)})$ is constant for $x_{\pa(k)}\in c_{\pa(k)}$, then it is true.
In particular, we have that 
\begin{align}
    \label{eq:mu.ba.markov.cells}
    \mu^{b\mathcal{A}}(c_k \mid c_{1:k-1}) = \mu^{b\mathcal{A}}(c_k \mid c_{\pa(k)})
\end{align}
for all cells $c\in\mathcal{A}$ and $k=1,\dots, K$.

In the formulation of the next result we will use signed measures and denote by $\|\nu\|_{\TV}=\TV(\nu,0)=\sup_{f: |f|\leq 1/2} \int f\, d\nu$ the total variation norm of a signed measure $\nu$.

  \begin{lemma}
  \label{lem:mu.bA.mu.almost.same.cell}
        Let Assumption \ref{ass:graph_struc} hold and assume that $\mu$ satisfies Assumption \ref{ass:TV.Lip}.
        Then there exist a constant $C$ that only depends on $G$ and $L$ and a signed measure $\tilde\mu$ which satisfies $\|\tilde \mu \|_{\TV} \leq C\delta_\mathcal{A}^2$, such that 
        \[
        \mu(c) = \mu^{b \mathcal{A}}(c) + \tilde\mu(c) ~~ \text{ for all } c \in \mathcal{A}.
        \]
        \begin{proof}
           We inductively show the corresponding statement for $\mu_{1:k}$ and $\mu^{b\mathcal{A}}_{1:k}$.
            The start $k=1$ is trivial since $\mu_1=\mu^{b\mathcal{A}}_1$; hence we may choose $\tilde \mu_1=0$.
            The proof for the induction from $k-1$ to $k$ requires some preparations, spelled out in the next step.

            \vspace{0.5em}
            \noindent
            \emph{Step 0:} We split the nodes 
            \[\{1,\dots, k\}=\anc(k) \cup \pa(k)\cup \{k\} \]
            into $k$, its parents, and the rest.
            Moreover, we disintegrate $\mu_{1:k}$ via $\pa(k)$,\footnote{Notably, the case $\pa(k)=\emptyset$ is trivial, as then $\mu_{1:k} = \mu_{1:k-1} \otimes \mu_k$ and $\mu_{1:k}^{b\mathcal{A}} = \mu_{1:k-1}^{b\mathcal{A}} \otimes \mu_k$.} thus
            \begin{align*}
            \mu_{1:k}(dx_{1:k})
            = \mu_{\pa(k)}(dx_{\pa(k)}) \, \mu(dx_k\mid x_{\pa(k)}) \,\mu(dx_{\anc(k)} \mid x_{\pa(k)}).
            \end{align*}
            Note that this disintegration holds true since $\mu \in \mathcal{P}_G$, which means the variable $k$ and the variables in $\anc(k)$ are conditionally independent given the $\pa(k)$ variables.

            To simplify notation, we shall assume that $k=3$, $\pa(k)=2$, $\anc(k)=1$, and write  $R_{2\to 3}(x_2,dx_3) = \mu(dx_3 \mid x_2)$ and similarly for $R_{2\to 1}$; thus
            \[ \mu(dx_{1:3}) = \mu_2(dx_2) R_{2\to 3}(x_2, dx_3) R_{2\to 1}(x_2, dx_1).\]
            This can be done without loss of generality and the proof in the original case follows simply by exchanging notation.

        \vspace{0.5em}
            \noindent
            \emph{Step 1:}
            Define the averaged version of $R_{2\to 3}$ via
            \begin{align*}
                \bar R_{2\to 3}(x_2, dx_3) 
                &:=  \int R_{2\to 3}(\tilde x_2, dx_3)\, \mu_2|_{c_2(x_2)}(d\tilde x_2) 
            \end{align*}
            for $x_2\in\X_2$.
            Thus $x_2\mapsto \bar R_{2\to 3}(x_2, dx_3) $ is constant as long as $x_2$ belongs to a fixed cell $c_2$, and we often write $\bar R_{2\to 3}(c_2, dx_3)$ in that case.
           Using the convexity of $\TV$ followed by Assumption \ref{ass:TV.Lip}, we have that 
            \begin{align}
            \label{eq:TV.convex}
            \begin{split}
            &\TV(\bar R_{2\to 3}(x_2,\cdot),  R_{2\to 3}(x_2, \cdot)) \\
            &\leq \int \TV( R_{2\to 3}(\tilde x_2,\cdot),  R_{2\to 3}(x_2, \cdot) )\, \mu_2|_{c_2(x_2)}(d\tilde x_2)
            \leq L \delta_\mathcal{A}
            \end{split}
            \end{align}
            
                           and hence
            \[  R_{2\to 3}(x_2, dx_3)  
            = \bar R_{2\to 3}(x_2, dx_3)  + \delta_\mathcal{A} D_{2\to 3}(x_2,dx_3) \]
            for some kernel $D_{2\to 3}$ that satisfies $\|D_{2\to 3}(x_2,dx_3) \|_{\TV} \leq L$.
            Moreover,  by definition we find $\int_{c_2} D_{2\to 3}(x_2,dx_3)\,\mu_2(dx_2) = 0 $ for every $c_2\in\mathcal{A}_2$.
        
           \vspace{0.5em}
            \noindent
            \emph{Step 2:} 
            Here we analyse the error made by replacing $R_{2\to 3}$ by $\bar R_{2\to 3}$.
            Fix $c_{1:3}\in\mathcal{A}_{1:3}$
            Using the decomposition $ R_{2\to 3}= \bar R_{2\to 3} + \delta_\mathcal{A} D_{2\to 3}$ and that $\bar R_{2\to 3}(x_2,\cdot)$ is constant for $x_2\in c_2$,   
             \begin{align}
             \label{eq:mu.bA.cells.almost.step.1}
             \begin{split}
                \mu(c_{1:3}) 
                &= \int_{c_{1:2}} \left( \bar R_{2\to 3}(x_2, c_3)  + \delta_\mathcal{A} D_{2\to 3}(x_2,c_3) \right)\, \mu_{1:2}(dx_{1:2})  \\
                &= \mu_{1:2}(c_{1:2}) \bar R_{2\to 3}(c_2, c_3)  +\delta_\mathcal{A} \int_{c_{1:2}}D_{2\to 3}(x_2, c_3) \, \mu_{1:2}(dx_{1:2}) \\
                &=:({\rm I}) + ({\rm II}).
                \end{split}
            \end{align}
            By the induction hypothesis,
            \begin{align}
              \label{eq:mu.bA.cells.almost.step.2}
             \begin{split}
                ({\rm I}) &=\left( \mu_{1:2}^{b\mathcal{A}}(c_{1:2}) + \tilde \mu_{1:2}(c_{1:2}) \right) \bar R_{2\to 3}(c_2, c_3)  \\
            &= \mu_{1:3}^{b\mathcal{A}}(c_{1:3}) + \tilde \mu_{1:2}(c_{1:2})\bar R_{2\to 3}(c_2, c_3) .
            \end{split}
            \end{align}
            (For the second equality note that while $ \mu_{1:2}^{b\mathcal{A}}(c_{1:2})\bar R_{2\to 3}(c_2, c_3) = \mu_{1:3}^{b\mathcal{A}}(c_{1:3}) $ does not hold for arbitrary sets $c_{1:3}\subset \X_{1:3}$, it does hold for cells.)

            \vspace{0.5em}
            \noindent
            \emph{Step 3:}
            We proceed to control the term  $({\rm II})$. 
            Analogously to Step 1, define $\bar R_{2\to 1}(x_2,dx_1) =\int R_{2\to 1}(\tilde x_2, dx_1) \,\mu_{2}|_{c_2(x_2)}(d\tilde x_2)$ and $D_{2\to 1}$ via  $ R_{2\to 1}  = \bar R_{2\to 1}+ \delta_\mathcal{A} D_{2\to 1}$.
            Using this notation, 
              \begin{align}  
            \label{eq:mu.bA.cells.almost.step.3}
             \begin{split}
               ({\rm II}) 
               &= \delta_\mathcal{A} \int_{c_{2}} D_{2\to 3}(x_2, c_3) \bar R_{2\to 1}(x_2, c_1)  \, \mu_2(dx_2)  \\
               &\quad + \delta_\mathcal{A}^2  \int_{c_{2}} D_{2\to 3}(x_2, c_3)  D_{2\to 1}(x_2, c_1) \, \mu_2(dx_2) \\
               &=: ({\rm III}) + ({\rm IV}).
            \end{split}
            \end{align}
             By definition, $|({\rm IV})|\leq L^2 \delta_\mathcal{A}^2$.
            Moreover, 
            since $\bar R_{2\to 1}(x_2, c_1)$ is constant for $x_2\in c_2$ and $\int_{c_{2}} D_{2\to 3}(x_2, c_3)\,\mu_2(dx_2)=0$ by the definition of $D_{2\to 3}$, it follows that 
               \begin{align}  
            \label{eq:mu.bA.cells.almost.step.4}
             \begin{split}
             ({\rm III})
             &= \delta_\mathcal{A}\int_{c_{2}} D_{2\to 3}(x_2, c_3)  \,\mu_2(dx_2) \,\bar R_{2\to 1}(c_2, c_1)  
             =0.
        \end{split}
            \end{align}            
            
             \vspace{0.5em}
            \noindent
            \emph{Step 4:}
            Define the (signed) measure 
            \[ \tilde{\mu}_{1:3}:= \tilde \mu_{1:2}\otimes \bar R_{2\to 3} + \delta_\mathcal{A}^2 \cdot  \mu_2 \otimes D_{2\to 3}\otimes D_{2\to 1}. \]
            One readily checks that $\|\tilde{\mu}_{1:3}\|_{\TV} \leq \|\tilde \mu_{1:2}\|_{\TV} + (\delta_\mathcal{A} L)^2 $.
            By \eqref{eq:mu.bA.cells.almost.step.1}--\eqref{eq:mu.bA.cells.almost.step.4} we have that $\mu^{b\mathcal{A}}(c_{1:3}) = \mu(c_{1:3}) + \tilde \mu(c_{1:3})$, which completes the induction step and thus the proof.            
        \end{proof}
    \end{lemma}

To eventually bound $\W(\mu, \mu^{b\mathcal{A}})$, we first show how to control $\TV(\mu, \mu^{b\mathcal{A}})$.
	\begin{lemma}\label{lem:TVboundtwo}
		Let $K=2$ and assume that $\mu$ satisfies Assumption \ref{ass:TV.Lip}.
        Then,
		\begin{align}
        \label{eq:lemma.tv.estiamte}
		\TV(\mu, \mu^{b\mathcal{A}})
        &\leq 2L\delta_{\mathcal{A}},\\
         \label{eq:lemma.tv.integrated.estiamte}
			\int \TV(\mu(dx_2 \mid x_1), \mu^{b\mathcal{A}}(dx_2 \mid x_1)) \,\mu(dx_1) 
            &\leq 2L \delta_{\mathcal{A}}.
		\end{align}
        	\begin{proof}
            Write
            \[ \TV(\mu,\mu^{b\mathcal{A}})
            \leq\TV(\mu,\mu^{\mathcal{A}})   + \TV(\mu^\mathcal{A},\mu^{b\mathcal{A}}).\]
            To estimate $\TV(\mu,\mu^{\mathcal{A}})$, since $\mu^{\mathcal{A}}(dx_2\mid x_1)$ is the average of $\mu(dx_2\mid \tilde x_1)$ over $\tilde{x}_1$ that satisfy $\|x_1-\tilde x_1\|\leq \delta_\mathcal{A}$, it follows as in  \eqref{eq:TV.convex} from convexity of the total variation distance  and Assumption \ref{ass:TV.Lip} that 
            \[\TV(\mu(dx_2 \mid x_1), \mu^{\mathcal{A}}(dx_2 \mid x_1)) 
            \leq  L  \delta_{\mathcal{A}}.\]
            And since $\mu$ and $\mu^\mathcal{A}$ have the same marginals, 
            \begin{align}
                \label{eq:TV.integral.proof}
                \TV(\mu,\mu^\mathcal{A}) = \int \TV(\mu(dx_2 \mid x_1), \mu^{\mathcal{A}}(dx_2 \mid x_1))   \,\mu_1(dx_1)
            \leq  L  \delta_{\mathcal{A}}.
            \end{align}

            To estimate $\TV(\mu^\mathcal{A},\mu^{b\mathcal{A}})$, we recall from Lemma \ref{lem:mu.ba.euals.swap} that 
            \[\mu^{b\mathcal{A}} = \mathcal{S}\Big( \big(\mathcal{S}(\mu^\mathcal{A})\big)^\mathcal{A} \Big) \]
            and (recalling the proof of Lemma~\ref{lem:mu.ba.euals.swap}) that
            \[ \mu^\mathcal{A} = \mu_2(dx_2) V^{r\mathcal{A}}(x_2,dx_1)\]
            where $V^{r\mathcal{A}}(x_2,dx_1)= \int H(\tilde x_1, dx_1) \, \mu(d\tilde x_1 \mid x_2)$.
            The exact form of $H$ is not important, only that $H$ is a stochastic kernel (i.e., it has total variation norm 1 for each $\tilde x_1$), hence $x_2\mapsto V^{r\mathcal{A}}(x_2,dx_1)$ is $L$-Lipschitz in total variation by the data processing inequality.
            Since the operator $\mathcal{S}^{-1}$ does not change the total variation distance,
            \[\TV \left(  \mu^{\mathcal{A}}, \mu^{b\mathcal{A}} \right)
            = \TV\left( \mu^{\mathcal{A}},  \mathcal{S}(\mathcal{S}(\mu^{\mathcal{A}})^\mathcal{A}) \right)
            = \TV \left( \mathcal{S}(\mu^{\mathcal{A}}),  \mathcal{S}(\mu^{\mathcal{A}})^\mathcal{A} \right).\]
            Thus, the same arguments as used when estimating $\TV(\mu,\mu^{\mathcal{A}})$ show that $\TV(\mu^{\mathcal{A}}, \mu^{b\mathcal{A}})\leq L \delta_\mathcal{A}$.
            This completes the proof of \eqref{eq:lemma.tv.estiamte}.
            Finally \eqref{eq:lemma.tv.integrated.estiamte} follows because $\mu$ and $\mu^{b\mathcal{A}}$ have the same marginals.
		\end{proof}
	\end{lemma}

In the proof of Theorem \ref{thm:estTV} we shall make use of the Kantorovich–Rubinstein duality, that is, for any $\nu,\tilde{\nu}\in\mathcal{P}(\X)$,
        \begin{align}
            \label{eq:W.dual.rep}
            \W(\nu,\tilde{\nu}) = \sup_{f\in\mathcal{F}} \int f \,d(\nu-\tilde{\nu}),
        \end{align} 
        where $\mathcal{F}$ is the set of all functions from $(\X,\|\cdot\|)$ to $\mathbb{R}$ that are $1$-Lipschitz (and satisfy $f(0)=0$).
        See, e.g., \cite[Theorem 5.10 \& 5.16]{villani2008optimal} for a proof of this fact.
        
The following gives the main result on controlling the bias $\W(\mu, \mu^{b\mathcal{A}})$ for Theorem \ref{thm:estTV}.
\begin{lemma}\label{lem:TVboundbias}
    	Suppose that Assumption \ref{ass:graph_struc} holds and that $\mu$ satisfies Assumption \ref{ass:TV.Lip}.
    Then, 
	\begin{align*}
	\TV(\mu, \mu^{b\mathcal{A}}) &\leq (K-1)2L\delta_{\mathcal{A}},\\
	\W(\mu, \mu^{b\mathcal{A}}) &\leq ((K-1)2L+2C)  \delta_{\mathcal{A}}^2,
	\end{align*}
    where $C$ is the constant from Lemma \ref{lem:mu.bA.mu.almost.same.cell}.
	\begin{proof}
        The first inequality follows by using the optimal transport definition of the total variation distance\footnote{That is, the total variation between two probability measures is equal to the optimal transport distance with discrete cost function $c(x,y) = 1_{x\neq y}$: $\TV(\nu,\tilde\nu)=\inf_\pi \int c(x,y),\pi(dx,dy)$ with the infimum being over all couplings $\pi$ between $\nu$ and $\tilde\nu$.}  and iterating  \eqref{eq:lemma.tv.integrated.estiamte}.
        Indeed, we claim that for every $M=1, \dots, K$,
        \[
		\TV\left(\mu_{1:M} ,\mu^{b\mathcal{A}}_{1:M}\right) \leq (M-1) 2L \delta_{\mathcal{A}}.
		\]
        
        The proof of this claim is via induction, noting that the case $M=2$ is already covered by Lemma \ref{lem:TVboundtwo}.
        For the induction step from $M-1$ to $M$, let $\pi$ be an optimal coupling for the OT-representation of $\TV(\mu_{1:M-1},\mu^{b\mathcal{A}}_{1:M-1})$.
        Thus 
        \[\pi(\{x_{1:M-1},y_{1:M-1} : x_{1:M-1}\neq y_{1:M-1}\}) \leq (M-2)2L\delta_\mathcal{A}.\]
        Similarly to the proof of Lemma \ref{lem:Wdecomp}, let $\pi^{x_{\pa(M)},y_{\pa(M)}}$ be a measurable family of couplings between $\mu(dx_{M}|x_{\pa(M)})$ and $\mu^{b\mathcal{A}}(dy_{M}|y_{\pa(M)})$ which are optimal for their $\TV$-distance; and let $\Gamma$ be the concatenation with respect to $\pi$.
        Thus
        \begin{align*}
          &  \Gamma(\{x_{1:M}, y_{1:M} : x_{1:M} \neq y_{1:M}\}) \\
            &\leq (M-2)2L\delta_\mathcal{A} +\int  \pi^{ x_{\pa(M)},x_{\pa(M)}}(\{x_M,y_M : x_M\neq y_M\}) \, \mu(dx_{1:M-1})\\
            &= (M-2)2L\delta_\mathcal{A} +\int  \TV\left(  \mu(dx_{M} | x_{\pa(M)}), \mu^{b\mathcal{A}}(dx_{M}|x_{\pa(M)} ) \right) \, \mu(dx_{1:M-1})\\
            &\leq (M-2)2L\delta_\mathcal{A} + 2L\delta_\mathcal{A},
        \end{align*}
        where we used the induction hypothesis in the first inequality, and \eqref{eq:lemma.tv.integrated.estiamte} in the second inequality.
        This completes the proof of the claim, and thus of the first statement in the lemma.

       For the proof of the second statement we rely on the Kantorovich-Rubinstein duality, see \eqref{eq:W.dual.rep}.
       Let $f$ be $1$-Lipschitz with $f(0)=0$.
       Denoting by $m_c$ the midpoint of a cell $c\in\mathcal{A}$, we have that
\begin{align*}
\int f\,d(\mu-\mu^{b\mathcal{A}})
  &= \sum_{c\in\mathcal A}    \int_c f(x)-f(m_c) \,d(\mu-\mu^{b\mathcal{A}}) + \sum_{c\in\mathcal{A}} 
        f(m_c)\,(\mu(c)-\mu^{b\mathcal{A}}(c)) \\
        &=:({\rm I})+({\rm II}).
\end{align*}
Since $|f(x)-f(m_c)|\leq \delta_\mathcal{A}/2$ for every $c\in\mathcal{A}$ and every $x\in c$, by the first part of the lemma,
\[
({\rm I}) \leq  \sum_{c\in\mathcal A} \frac{\delta_{\mathcal A}}{2}\,|\mu(c)-\mu^{b\mathcal{A}}(c)|
   \leq  \delta_{\mathcal A}\,\|\mu- \mu^{b\mathcal{A}}\|_{\mathrm{TV}}
    \leq (K-1)2L
  \delta_{\mathcal A}^2.
\]
Moreover, using the notation of Lemma \ref{lem:mu.bA.mu.almost.same.cell}, and that $|f(m_c)|\le 1$ for every $c\in\mathcal{A}$,
\[
({\rm II})=  \sum_{c\in\mathcal{A}} f(m_c)\,\tilde \mu(c) 
\leq 2 \|\tilde \mu \|_{\mathrm{TV}} 
\leq 2 C \delta_\mathcal{A}^2.
\]
As $f$ was arbitrary, this shows that $\W(\mu,\mu^{b\mathcal{A}}) \leq ((K-1)2L+2C) \delta_{\mathcal A}^2$, as claimed.
	\end{proof}
\end{lemma}

\subsection{Projection to fully discrete measures}\label{subsec:fullydiscrete}

The next preliminary results needed for the proof of Theorem \ref{thm:estTV} require projecting and working on the midpoints $\mathcal{M}$ of the cells $\mathcal{A}$.

\begin{definition}
    For every $\nu\in\mathcal{P}(\X)$, define $\nu^{\mathcal{M}}\in\mathcal{P}(\mathcal{M})$  by 
\begin{align*} 
\nu^{\mathcal{M}}(\{x\}):=\nu^{b\mathcal{A}}(c(x)), 
\qquad x\in\mathcal{M}.
\end{align*}
    In particular, this defines $\hat\mu^\mathcal{M}$ using $\nu=\hat\mu$.
    \end{definition}
In particular, $\nu^{\mathcal{M}}(c)=\nu^{b\mathcal{A}}(c)$ for every $c\in\mathcal{A}$.
Note that $\nu^\mathcal{M}_1(\{x_1\})=\nu^{b\mathcal{A}}(c_1(x_1))$  for every $x_1\in\mathcal{M}_1$.
Moreover, by \eqref{eq:mu.ba.markov.cells}, we have that $\nu^{\mathcal{M}}\in\mathcal{P}_G(\X)$  if $\nu\in\mathcal{P}_G(\X)$ (hence in particular $\hat\mu^\mathcal{M}\in\mathcal{P}_G(\X)$), and by Lemma \ref{lem:mubAwelldefined}, that for every $k=2,\dots,K$ and $x_{\pa(k)}\in\mathcal{M}_{\pa(k)}$,
\begin{align} 
\label{eq:mu.ba.corollary.from.lemma}
\nu^{\mathcal{M}}(\{x_k\} \mid x_{\pa(k)})
&=\nu^{b\mathcal{A}}(c_k(x_k) \mid x_{\pa(k)})  
=\nu(c_k(x_k) \mid c_{\pa(k)}(x_{\pa(k)}))  .
\end{align}
We have the following:

\begin{lemma}
\label{lem:mu.M.kernels.Lipschitz}
 Suppose that $\mu$ satisfies Assumption \ref{ass:TV.Lip}. Then
   the kernels of $\mu^\mathcal{M}$ are $(2L+1)$-Lipschitz w.r.t.\ $\mathcal{W}$; that is, for every $2\leq k \leq K$ and $x_{\pa(k)},\tilde x_{\pa(k)}\in\mathcal{M}_{\pa(k)}$, 
   \[ \W\left( \mu^\mathcal{M}(dx_k \mid x_{\pa(k)}) ,   \mu^\mathcal{M}(dx_k \mid \tilde x_{\pa(k)}) \right)
   \leq (2L+1) \|  x_{\pa(k)} -  \tilde x_{\pa(k)} \|.\]
\end{lemma}
\begin{proof}
    We first claim that, for every $x_{\pa(k)}, \tilde x_{\pa(k)}\in\mathcal{M}_{\pa(k)}$,
    \begin{align}
    \label{eq:W.A.lip}
    \begin{split}
        &\W\left( \mu(dx_k \mid c_{\pa(k)}(x_{\pa(k)})) ,   \mu(dx_k \mid c_{\pa(k)}(\tilde x_{\pa(k)})) \right) \\
  &\quad \leq 2 L \|x_{\pa(k)} -  \tilde{x}_{\pa(k)}\|.
  \end{split}
    \end{align}
    To that end, first note that for every distinct $x_{\pa(k)},\tilde x_{\pa(k)}\in\mathcal{M}_{\pa(k)}$ and $x_{\pa(k)}'\in c_{\pa(k)}(x_{\pa(k)})$ and $\tilde x_{\pa(k)}'\in c_{\pa(k)}(\tilde x_{\pa(k)})$,
    \[
    \|x_{\pa(k)}' -  \tilde{x}_{\pa(k)}'\| \leq 2 \|x_{\pa(k)} -  \tilde{x}_{\pa(k)}\|.
    \]
    Moreover, since $\mu(dx_k \mid c_{\pa(k)}(x_{\pa(k)}))$ is the average of $\mu(dx_k \mid x'_{\pa(k)})$ over $x_{\pa(k)}'\in c_{\pa(k)}(x_{\pa(k)})$ and similarly for $\mu(dx_k \mid \tilde x_{\pa(k)}')$, a twofold application of convexity of $\W$ and using Assumption \ref{ass:TV.Lip} shows that
    \begin{align*}
        &\W\left( \mu(dx_k \mid c_{\pa(k)}(x_{\pa(k)})) ,   \mu(dx_k \mid c_{\pa(k)}(\tilde x_{\pa(k)})\right) \\
        & \leq \iint \W\left( \mu(dx_k \mid x_{\pa(k)}')) ,   \mu(dx_k \mid \tilde x_{\pa(k)}')\right)  \, \ \mu|_{ c_{\pa(k)}(x_{\pa(k)}}(dx_{\pa(k)}')\mu|_{ c_{\pa(k)}(\tilde x_{\pa(k)}}(d\tilde x_{\pa(k)}')\\ 
  &\quad \leq L \max_{x_{\pa(k)}' \in c(x_{\pa(k)}) , \tilde{x}_{\pa(k)}' \in c(\tilde{x}_{\pa(k)})} \| x_{\pa(k)}' -  \tilde{x}_{\pa(k)}'\| 
   \leq 2 L \|x_{\pa(k)} -  \tilde{x}_{\pa(k)}\|
    \end{align*}
    for all $x_{\pa(k)},\tilde x_{\pa(k)}\in\mathcal{M}_{\pa(k)}$.
    This shows \eqref{eq:W.A.lip}.

\vspace{0.5em}
    In the next step, denote by $\psi\colon\mathcal{X}_k\to\mathcal{M}_k$ the map projecting cells to their centres, so that by \eqref{eq:mu.ba.corollary.from.lemma},
    \begin{align*}
        \mu^\mathcal{M}( dx_k \mid x_{\pa(k)} ) = \psi\left( \mu( dx_k \mid c_{\pa(k)}(x_{\pa(k)})) \right).
    \end{align*}
    Since $\|\psi(x_k) - x_k\|\leq \delta_\mathcal{A}/2$, we have
    \[\|\psi(x_k) - \psi(\tilde x_k)\|
    \leq \| x_k -\tilde x_k \| + \delta_\mathcal{A} \]
     for all $x_k, \tilde x_k \in \mathcal{X}_k$, and thus a basic coupling argument 
    shows that 
   \begin{align*}
        &\W\left( \mu^\mathcal{M}(dx_k \mid x_{\pa(k)}) ,   \mu^\mathcal{M}(dx_k \mid \tilde x_{\pa(k)}) \right) \\
        &\quad \leq \W\left( \mu(dx_k \mid c_{\pa(k)}(x_{\pa(k)})) ,   \mu(dx_k \mid c_{\pa(k)}(\tilde x_{\pa(k)})) \right) + \delta_\mathcal{A} \\
        &\quad \leq   2L \|x_{\pa(k)} -  \tilde{x}_{\pa(k)}\| + \delta_\mathcal{A}
    \end{align*}
    where we have used \eqref{eq:W.A.lip} in the last inequality.
    Since, for every distinct $x_{\pa(k)},\tilde{x}_{\pa(k)}\in\mathcal{M}_{\pa(k)}$ we have that $ \|x_{\pa(k)} -  \tilde{x}_{\pa(k)}\| \geq \delta_\mathcal{A}$, the claim follows.
\end{proof}

\begin{lemma}\label{lem:discreteconcrete}
    Let $1\leq k \leq K$, set  $\nu \in \mathcal{P}(\X_k)$ to be a probability measure which is supported on $\mathcal{M}_k$, and denote by $\hat{\nu}$ its empirical measure with $m$ samples.
    Recall that $2^{-\eta}$ is the side-length of the cells.
    Then, we have
	\[
	\mathbb{E}\left[ \W(\nu, \hat{\nu})\right] \leq \begin{cases}
		5 m^{-1/2}, &\text{if } d_k=1, \\
		2 \eta m^{-1/2}, &\text{if } d_k=2, \\
		8 m^{-1/2} 2^{\eta \cdot (\frac{d_k}{2}-1)}, &\text{if } d_k\geq3.
	\end{cases} 
	\]
	\begin{proof}
    		For $\ell=1,\dots, \eta$, denote by  $\mathcal{A}_k(\ell)$ the partition of $\X_k$ into cells of side-lengths $2^{-\ell}$.
        Thus $|\mathcal{A}_k(\ell)|=2^{\ell d_k}$ and $\mathcal{A}_k(\eta)=\mathcal{A}_k$, and $\mathcal{M}_k$ are the mid-points of the cells in $\mathcal{A}_k$.
		The standard refinement-of-partition chaining argument from Dudley \cite{dudley1969speed} yields that 
		\[
		\mathbb{E}\left[\W(\nu, \hat{\nu})\right] 
        \leq \sum_{\ell=1}^\eta 2^{1-\ell} \left(\frac{|\mathcal{A}_{k}(\ell)|}{m}\right)^{1/2} 
        = 2 m^{-1/2} \sum_{\ell=1}^\eta 2^{\ell \cdot (\frac{d_k}{2} - 1)}.
		\]
		The wanted estimate on $\mathbb{E}[\W(\nu, \hat{\nu})] $ in the case $d_k=2$ now follows immediately. 
        The cases $d_k=1$ and $d_k\geq 3$ follow by computing the geometric series.
	\end{proof}
\end{lemma}


 The following result is an analogue to Lemma \ref{lem:conditional.iid}; we defer the proof to Appendix \ref{app:TV}.
\begin{lemma}
\label{lem:M.conditional.iid}
    Let $1\leq k \leq K$, fix $x_{\pa(k)}\in\mathcal{M}_{\pa(k)}$ and let $c_{\pa(k)}$ be the unique cell containing $x_{\pa(k)}$.
    Then, conditionally on the event $n\cdot \hat\mu^\mathcal{M}(\{x_{\pa(k)}\})=m$,
    the random probability measure $\hat\mu^\mathcal{M}( dx_k | x_{\pa(k)})$ has the same distribution as the empirical measure of $\mu^\mathcal{M}( dx_k | x_{\pa(k)})$ with sample size $m$.
\end{lemma}

\begin{lemma}
\label{lem:W.mu.M.hat.mu.M}
 Suppose that $\mu$ satisfies Assumption \ref{ass:TV.Lip}.
   There is a constant $C$ depending only on $G$ and $L$ for which
    \[ \mathbb{E}\left[  \W(\mu^\mathcal{M},\hat\mu^\mathcal{M}) \right]
    \leq C\cdot  n^{-1/2} \cdot
    \begin{cases}
        \delta_\mathcal{A}^{1-d_{\rm loc}/2} & \text{if } d_{\rm loc} \text{ is not attained at } d_k=2,\\
       \log(\frac{1}{\delta_{\mathcal{A}}}) \delta_\mathcal{A}^{1-d_{\rm loc}/2} & \text{else}.
    \end{cases}
    \]
\end{lemma}
\begin{proof}
       By Lemma \ref{lem:mu.M.kernels.Lipschitz}, the kernels of $\mu^\mathcal{M}$ are $2L+1$-Lipschitz w.r.t.\ the Wasserstein distance.
       Hence, it follows from 
       Lemma \ref{lem:Wdecomp} (noting that Lipschitz continuity of the kernels of $\mu^{\mathcal{M}}$ only needs to hold $\mu^{\mathcal{M}}$-almost surely) that 
		\begin{align*}
		\W\left(\mu^\mathcal{M}, \hat{\mu}^\mathcal{M}\right) 
		&\leq C \sum_{k=1}^K \int  \W\left(\mu^\mathcal{M}(dx_k \mid x_{\pa(k)}), \hat{\mu}^\mathcal{M}(dx_k \mid x_{\pa(k)}) \right) \, \hat{\mu}^\mathcal{M}(dx) .
		\end{align*}

        To proceed further,  we fix $k$ and  distinguish between the cases $d_k\in\{1,2\}$ and $d_k\geq 3$, starting with the latter.
        An application of Lemma \ref{lem:M.conditional.iid} together with Lemma \ref{lem:discreteconcrete} shows that
        \begin{align*} 
        &\mathbb{E}\left[ \W\left(\mu^\mathcal{M}(dx_k \mid x_{\pa(k)}), \hat{\mu}^\mathcal{M}(dx_k \mid x_{\pa(k)}) \right) \mid   \hat{\mu}^\mathcal{M}(\{x_{\pa(k)}\})\right] 
        \leq  \frac{8 \cdot 2^{\eta (d_k/2 -1 )} }{( n \cdot \hat{\mu}^\mathcal{M}(\{x_{\pa(k)}\}) )^{1/2}} .
        \end{align*}
        Therefore, it follows from Jensen's inequality exactly as in the proof of Theorem \ref{thm:estWLip} that 
        \begin{align*} 
       ({\rm I}) &:=\mathbb{E}\left[\int  \W\left(\mu^\mathcal{M}(dx_k \mid x_{\pa(k)}), \hat{\mu}^\mathcal{M}(dx_k \mid x_{\pa(k)}) \right) \, \hat{\mu}^\mathcal{M}(dx) \right] \\
        &\leq   8 \cdot 2^{\eta (d_k/2 -1 )} \left( \frac{n}{ |\mathcal{M}_{\pa(k)}| } \right)^{-1/2}.
        \end{align*}
        
        Finally, since $|\mathcal{M}_{\pa(k)}|=\delta_\mathcal{A}^{-d_{\pa(k)}}$, it follows that
        \[ ({\rm I}) \leq  8 \delta_\mathcal{A}^{1-d_k/2} \delta_\mathcal{A}^{-d_{\pa(k)}/2} n^{-1/2}
        \leq 8\delta_\mathcal{A}^{1- d_{\rm loc}/2} n^{-1/2},\]
        where the second inequality follows from the definition of $d_{\rm loc}$.

       Next consider the case that $d_k\in\{1,2\}$.
       Here the same set of arguments as used when $d_k\geq 3$ show  that
        \[ ({\rm I}) \leq C
    \begin{cases}
         n^{-1/2}  \delta_\mathcal{A}^{-d_{\pa(k)} /2} = n^{-1/2} \delta_\mathcal{A}^{1-\frac{d_{\pa(k)}+2}{2}} \leq n^{-1/2} \delta_\mathcal{A}^{1-d_{\rm loc}/2} & \text{if } d_k=1,\\
         n^{-1/2} \log(\frac{1}{\delta_{\mathcal{A}}}) \delta_\mathcal{A}^{-d_{\pa(k)} /2} \leq n^{-1/2} \log(\frac{1}{\delta_{\mathcal{A}}}) \delta_{\mathcal{A}}^{1-d_{\rm loc}/2} \ & \text{if } d_k=2.
    \end{cases}\]
        We note that, if the final inequality in the case $d_k = 2$ is strict (that is, if $2 + d_{\pa(k)} < d_{\rm loc}$), then we can omit the $\log(\frac{1}{\delta_{\mathcal{A}}})$ term by possibly increasing the constant $C$ (cf.~the proof of Theorem \ref{thm:estWLip}). 
        Taking the maximum over all nodes $k$ of the derived inequalities yields the claim. 
\end{proof}

\begin{definition}
    For every  $f\in\mathcal{F}$ and $\nu\in\mathcal{P}(\X)$, define $f^\nu : \mathcal{M} \rightarrow \mathbb{R}$ via 
        \[f^{\nu}(x) := \frac{1}{ ( \otimes_{k=1}^K\nu_{k})(c(x)) }\int \eins_{c(x)}(y) f(y) \, (\otimes_{k=1}^K\nu_{k})(dy),\]
       with  the convention \mbox{$f^{\nu}(x) = 0$ if $(\otimes_{k=1}^K\nu_k)(c(x))=0$} (but this case will never be relevant in the analysis below).
\end{definition}

        Since $\sup_{x'\in c(x), y'\in c(y)} \|x'-y'\| \leq 2 \|x-y\|$ for any distinct $x,y\in\mathcal{M}$, it follows that $f^{\mu}$ is $2$-Lipschitz.
        Moreover, $f^\nu$ is constant in every cell $c\in\mathcal{A}$ and therefore can be seen in duality with the definition of $\nu^\mathcal{M}$ (see Lemma \ref{lem:average.on.cells}).
        As such, it will allow us to control $\W(\mu^{b\mathcal{A}},\hat\mu^{b\mathcal{A}})$ in terms of $\W(\mu^\mathcal{M},\hat\mu^\mathcal{M})$ plus an error that only depends on the marginals which is the subject of Lemma \ref{lem:E.sup.fmu.minus.fhatmu}.

\begin{lemma}
\label{lem:average.on.cells}
    For every $f\in\mathcal{F}$ and every $\nu\in\mathcal{P}(\X)$, 
	\[
		\int f \,d\nu^{b\mathcal{A}} = \int f^\nu \,d\nu^\mathcal{M}.
		\]
\end{lemma}

The proof of this lemma in case that $K=2$ follows essentially from the definitions.
Indeed, for each fixed cell we have that: $\nu^\mathcal{M}$ is the projection of $\nu^{b\mathcal{A}}$ to the cell's midpoint,  $\nu^{b\mathcal{A}}$ itself is the  (weighted) product measure between its marginals (see \eqref{eq:mu.as.local.prod}),  and $f^\nu$ is constant and equal to the average on that cell w.r.t.\ said product measure.
We defer a rigorous proof to Appendix \ref{app:TV}.

\begin{lemma}
\label{lem:E.sup.fmu.minus.fhatmu}
There exists a constant $C$ depending only on $G$ and $L$ such that
         \begin{align*}
        \mathbb{E}\left[ \sup_{f\in\mathcal{F}} \int (f^\mu - f^{\hat{\mu}}) \,d\hat{\mu}^\mathcal{M} \right]
        &\leq   C \begin{cases}
        n^{-1/2}\delta_\mathcal{A}^{1/2} &\text{if } \max_{k=1, \dots, K} d_k = 1,\\
        \max\{\log(n),1\} n^{-1/2} &\text{if } \max_{k=1, \dots, K} d_k  = 2,\\
        n^{-1/d_{\rm max}} & \text{else}.
        \end{cases}
        \end{align*}
\end{lemma}
\begin{proof}
        By the definitions of $f^\mu,f^{\hat{\mu}}$ and $\hat{\mu}^\mathcal{M}$, for any $f\in\mathcal{F}$,
		\begin{align*}
		\int (f^\mu - f^{\hat{\mu}}) \,d\hat{\mu}^\mathcal{M} 
        &= \sum_{ c\in \mathcal{A}} \hat{\mu}^{\mathcal{M}}(c) \int f \,d\big( (\otimes_{k=1}^K\mu_{k|c_k}) - (\otimes_{k=1}^K \hat\mu_{k|c_k}) \big) \\
        &\leq \sum_{ c\in \mathcal{A}} \hat{\mu}^{\mathcal{M}}(c)  \W\left( \otimes_{k=1}^K\mu_{k|c_k}, \otimes_{k=1}^K \hat\mu_{k|c_k} \right).
        \end{align*}
        Using a basic coupling argument, one can readily verify that 
        \[\W\left( \otimes_{k=1}^K\mu_{k|c_k}, \otimes_{k=1}^K \hat\mu_{k|c_k} \right)
        \leq \sum_{k=1}^K\W\left( \mu_{k|c_k},  \hat\mu_{k|c_k} \right).\]
        Finally, since $\hat{\mu}^{\mathcal{M}}(c_k)=\hat{\mu}(c_k)$ by definition and Lemma \ref{lem:mubAwelldefined},
        \begin{align*}
        \sup_{f\in\mathcal{F}} \int (f^\mu - f^{\hat{\mu}}) \,d\hat{\mu}^\mathcal{M} 
        &\leq \sum_{k = 1}^K \sum_{c_k \in \mathcal{A}_k} \hat\mu^{\mathcal{M}}(c_k) \W(\mu_{k\mid c_k}, \hat\mu_{k\mid c_k})
        \\
        &=  \sum_{k=1}^K \sum_{c_k\in\mathcal{A}_k} \hat\mu(c_k)\W\left( \mu_{k|c_k},  \hat\mu_{k|c_k} \right).
        \end{align*}

        Fix $1\leq k\leq K$. 
        In order to estimate $\W( \mu_{k|c_k},  \hat\mu_{k|c_k})$ first note that $\hat{\mu}_k$ is the empirical measure of $\mu_k$.
        Moreover, similarly to the proof of Lemma \ref{lem:M.conditional.iid} (in fact, simpler), one may verify that for every $m=1,\dots, n$, conditionally on the event that $n\hat\mu_{k}(c_k)=m$, the  random measure $\hat\mu_{k|c_k}$ has the same distribution as the empirical measure of $\mu_{k|c_k}$ with $m$ samples.
        Therefore, it follows from \eqref{eq:classical.convergence.wasserstein.in.proof} that
        \begin{align}
            \label{eq:marginals.conditions.W}
             \mathbb{E}\left[ \W( \mu_{k|c_k},  \hat\mu_{k|c_k})  \mid \hat\mu_{k}(c_k) = m \right]
        \leq \delta_\mathcal{A} l_m(d_k) m^{-1/\max\{2,d_k\}}.
        \end{align}
    The $\delta_\mathcal{A}$ factor arises because $\mu_{k \mid c_k}$ is supported on the cell $c_k$, which is a $\delta_\mathcal{A}$-scaled translate of $[0,1]^{d_k}$. Since the Wasserstein distance is homogeneous under rescaling of the domain, this scaling introduces the $\delta_\mathcal{A}$ term.

        Using \eqref{eq:marginals.conditions.W} and repeating the exact same steps as used in the proof of Theorem \ref{thm:estWLip}, it follows that 
        \[  \mathbb{E}\left[  \sum_{c_k\in\mathcal{A}_k} \hat\mu(c_k)\W\left( \mu_{k|c_k},  \hat\mu_{k|c_k} \right) \right]
        \leq \delta_\mathcal{A} l_n(d_k) \left( \frac{ n }{|\mathcal{A}_k| } \right)^{-1/\max\{2,d_k\}}  
        =:({\rm I}).   \]
        Finally, recall that $|\mathcal{A}_k| =\delta_\mathcal{A}^{-d_k}$, thus
        \[
        ({\rm I})  = l_n(d_k)\begin{cases}
         n^{-1/2} \delta_\mathcal{A}^{1/2} &\text{if } d_k=1,\\
         n^{-1/2}  &\text{if } d_k=2,\\
         n^{-1/d_k} &\text{else}.
        \end{cases}\]
        The claim of the lemma readily follows.
\end{proof}

\subsection{Proof of Theorem \ref{thm:estTV}}\label{subsec:proofoftvlipthm}

By the triangle inequality,
		\begin{align*}
		\W(\mu, \hat{\mu}^{b\mathcal{A}}) 
        \leq \W(\mu, \mu^{b\mathcal{A}}) + \W(\mu^{b\mathcal{A}}, \hat{\mu}^{b\mathcal{A}})
		\end{align*}
        and by Lemma \ref{lem:TVboundbias},  $\W(\mu, \mu^{b\mathcal{A}})\leq ((K-1)2L+2C) \delta_\mathcal{A}^2$.
        Next, we claim that 
        \begin{align}
        \label{eq:w.split.another}
            \W(\mu^{b\mathcal{A}}, \hat{\mu}^{b\mathcal{A}}) 
        \leq 2 \cdot \W(\mu^{\mathcal{M}},\hat{\mu}^\mathcal{M}) + \sup_{f\in\mathcal{F}} \int (f^\mu - f^{\hat{\mu}}) \,d\hat{\mu}^\mathcal{M} .
        \end{align}
        Indeed, this follows from the dual representation of the Wasserstein distance (see \eqref{eq:W.dual.rep}),  a twofold application of  Lemma \ref{lem:average.on.cells} which shows that, for any $f\in\mathcal{F}$, 
        \begin{align*}
		\int f \,d(\mu^{b\mathcal{A}} - \hat{\mu}^{b\mathcal{A}})
        &= \int f^\mu \,d(\mu^{\mathcal{M}} - \hat{\mu}^\mathcal{M}) + \int (f^\mu - f^{\hat{\mu}}) \,d\hat{\mu}^\mathcal{M},
		\end{align*}
        and the first term on the right hand side is upper bounded by $2 \cdot \W(\mu^{\mathcal{M}},\hat{\mu}^\mathcal{M})$ because $f^\mu$ is $2$-Lipschitz.
        
    Next recall that by Lemma \ref{lem:W.mu.M.hat.mu.M}, 
       \begin{align}\begin{split}
        \label{eq:final.proof.estimate.2} &\mathbb{E}\left[\W\left( \mu^{\mathcal{M}},\hat{\mu}^\mathcal{M} \right)\right]\\
        &\leq ({\rm I}):=
          C\cdot  n^{-1/2} \cdot
    \begin{cases}
        \delta_\mathcal{A}^{1-d_{\rm loc}/2} & \text{if } d_{\rm loc} \text{ is not attained at } d_k=2,\\
       \log(\frac{1}{\delta_{\mathcal{A}}}) \delta_\mathcal{A}^{1-d_{\rm loc}/2} & \text{else}.
    \end{cases}\end{split}
        \end{align}
        and by Lemma \ref{lem:E.sup.fmu.minus.fhatmu},
        \begin{align}
        \label{eq:final.proof.estimate.3}
        \mathbb{E}\left[ \sup_{f\in\mathcal{F}} \int (f^\mu - f^{\hat{\mu}}) \,d\hat{\mu}^\mathcal{M} \right]
        \leq ({\rm II})
        := C \begin{cases}
        n^{-1/2}\delta_\mathcal{A}^{1/2} &\text{if } \max_{k=1, \dots, K} d_k = 1,\\
      \max\{  \log(n),1\} n^{-1/2} &\text{if } \max_{k=1, \dots, K} d_k  = 2,\\
        n^{-1/d_{\rm max}} & \text{else}.
        \end{cases}
        \end{align}

        Collecting all terms shows that 
        \[ \mathbb{E}[ \W(\mu, \hat{\mu}^{b\mathcal{A}}) ]  
        \leq C \left( \delta_\mathcal{A}^2 + ({\rm I}) +({\rm II})\right) .\]
        Hereby, it is straightforward to see that $({\rm I})$ is bounded by the term $l_n n^{-1/d_{\rm max}}$ in the statement of the Theorem irrespective of $\delta_{\mathcal{A}}$. 

        To treat the terms $\delta_{\mathcal{A}}^2$ and $({\rm I})$, recall that $\eta$ is chosen as the largest integer satisfying $\eta\leq  \frac{\log_2(n)}{2+d_{\rm loc}}$ and $\delta_{\mathcal{A}}=2^{-\eta}$; in particular 
        \[   n^{-1/(2+d_{\rm loc})} \leq \delta_\mathcal{A}  \leq 2 n^{-1/(2+d_{\rm loc})}.\]
        and hence $\delta_{\mathcal{A}}^2 \leq 4 n^{-2/(2+d_{\rm loc})}$ as required. 
        
        Regarding $({\rm I})$, we have
        \[\delta_\mathcal{A}^{1-d_{\rm loc}/2} n^{-1/2} 
        \leq \left(n^{-1/(2+d_{\rm loc})}\right)^{1-d_{\rm loc}/2} n^{-1/2} = n^{-2/(2+d_{\rm loc})}\]
        and
        \[
        \log\left(\frac{1}{\delta_{\mathcal{A}}}\right) \leq \frac{\log(n)}{2+d_{\rm loc}} \leq \log(n),
        \]
        which is thus also as required in the term $l_n n^{-2/(2+d_{\rm loc})}$ as given in the Theorem. 
        
        Finally, noting that the log factor in $l_n$ is only relevant for the dominant term of $({\rm I})$ and $({\rm II})$, the claimed form of $l_n$ follows, completing the proof. \qed

\subsection{Proof of Theorem \ref{thm:main.TV.adaptive}}
\label{sec:tv.adaptive}
Assume without loss of generality that the sample size is $2n$ and split the sample into two independent halves
$
    \mathcal{S} = (X^1,\ldots,X^n) $ and $
    \mathcal{T} = (X^{n+1},\ldots,X^{2n})$.
Let $\hat\mu$ be the empirical measure of $\mathcal{S}$ and let $\hat\mu'$ be the empirical measure of $\mathcal{T}$.

 \vspace{0.5em}
            \noindent
\emph{Step 1:}
For each $G \in \mathcal{G}$, let $\hat\nu_G$ be the estimator from Definition \ref{def:mu.ba}, constructed from $\hat\mu$ with the bandwidth prescribed for the graph $G$, that is, $\eta=\eta(G)     =
     \lfloor
     \frac{\log_2 n}{2+d_{\rm loc}(G)}
     \rfloor$.
     Since $\mathcal{G}$ is a subset of the set of all directed graphs on $K$ vertices,
\begin{align}\label{eq:Nbound}
    \log (|\mathcal{G}|) \;\leq\; K^2 \log 2 .
\end{align}

For every pair $G,H \in \mathcal{G}$, choose a maximizer $f^*_{G,H}\in\mathcal{F}$ in the dual representation \eqref{eq:W.dual.rep} for $\W(\hat\nu_G,\hat\nu_H)$, normalized so that $f^*_{G,H}(0)=0$. If $G=H$, take $f^*_{G,G}=0$. Define
\[
    \hat{\mathcal{F}}
    :=
    \{f^*_{G,H}:\, G,H \in \mathcal{G}\}
    \subset \mathcal{F}.
\]
Then $|\hat{\mathcal{F}}| \leq |\mathcal{G}|^2$, and $\hat{\mathcal{F}}$ is $\mathcal{S}$-measurable.

For probability measures $\alpha,\beta$, define a (random, $\mathcal{S}$-measurable) pseudo-distance
\[
    D(\alpha,\beta)
    :=
    \sup_{f\in\hat{\mathcal{F}}}
    \abs{\int f\,d(\alpha-\beta)} .
\]
Select
\[
     G^\ast
    \in
    \mathop{\rm argmin}_{H\in\mathcal{G}}
    D(\hat\mu',\hat\nu_H),
\]
and set $\hat\mu^{b\mathcal{A},ad}:= \hat\nu_{G^\ast}$ to be the estimator.

 \vspace{0.5em}
            \noindent
\emph{Step 2:}
Fix an arbitrary $H\in\mathcal{G}$. Since $f^*_{ G^\ast,H}\in\hat{\mathcal{F}}$ and
\[
    \W(\hat\mu^{b\mathcal{A},ad},\hat\nu_H)
    =
    \int f^*_{G^\ast,H}\,d(\hat\mu^{b\mathcal{A},ad}-\hat\nu_H),
\]
we have
\begin{align*}
    \W(\hat\mu^{b\mathcal{A},ad},\hat\nu_H)
    &=
    \int f^*_{ G^\ast,H}\,d(\hat\mu^{b\mathcal{A},ad}-\hat\mu')
    +
    \int f^*_{G^\ast,H}\,d(\hat\mu'-\hat\nu_H)  \\
    &\leq
    D(\hat\mu',\hat\mu^{b\mathcal{A},ad})
    +
    D(\hat\mu',\hat\nu_H)  \\
    &\leq
    2D(\hat\mu',\hat\nu_H),
\end{align*}
where the last inequality follows from the definition of $ G^\ast$ and $\hat\mu^{b\mathcal{A},ad}$. Hence, by the triangle inequality,
\begin{align}\label{eq:tournament}
    \W(\hat\mu^{b\mathcal{A},ad},\mu)
    \leq
    2D(\hat\mu',\hat\nu_H)
    +
    \W(\hat\nu_H,\mu).
\end{align}
Moreover, every $f\in\hat{\mathcal{F}}$ is $1$-Lipschitz, so by \eqref{eq:W.dual.rep},
$
   \abs{\int f\,d(\mu-\hat\nu_H)}
    \leq
    \W(\mu,\hat\nu_H)
$
and therefore
\[
    D(\hat\mu',\hat\nu_H)
    \leq
    \Delta
    +
    \W(\mu,\hat\nu_H),
    \qquad
    \Delta
    :=
    \sup_{f\in\hat{\mathcal{F}}}
    \abs{\int f\,d(\hat\mu'-\mu)} .
\]
Plugging this into \eqref{eq:tournament} gives
\begin{align}\label{eq:oracle}
    \W(\hat\mu^{b\mathcal{A},ad},\mu)
    \leq
    3\W(\hat\nu_H,\mu)
    +
    2\Delta .
\end{align}

 \vspace{0.5em}
            \noindent
\emph{Step 3:}
Now fix the true graph $G\in\mathcal{G}$ of $\mu$ and take $H=G$ in \eqref{eq:oracle}.
By construction, 
Theorem \ref{thm:estTV} gives that
\begin{align}\label{eq:bias}
    \E{\W(\hat\nu_G,\mu)}
    \leq
    C_1 \cdot \max\{\log(n),1\}
    \cdot
    \brak{
        n^{-2/(2+d_{\rm loc}(G))}
        +
        n^{-1/d_{\rm max}}
    },
\end{align}
where $C_1$ depends only on $L$, $K$, and $(d_k)_{k=1}^K$. Indeed, any graph-dependent constant can be absorbed into such a constant by taking the maximum over the finite collection $\mathcal{G}$.

Conditional on $\mathcal{S}$, the class $\hat{\mathcal{F}}$ is deterministic and has cardinality at most $|\mathcal{G}|^2$. Moreover, every $f\in\hat{\mathcal{F}}$ is $1$-Lipschitz and satisfies $f(0)=0$, hence $|f|\leq 1$ on $\mathcal{X}=[0,1]^d$.

For fixed $f\in\hat{\mathcal{F}}$, conditionally on $\mathcal{S}$
\[
    \int f\,d(\hat\mu'-\mu)
    =
    \frac1n \sum_{i=n+1}^{2n}
    \brak{f(X^i)-\E{f(X^i) \mid \mathcal{S}}}
\]
is centered and $C n^{-1/2}$-sub-Gaussian by Hoeffding's lemma. Therefore, by the standard maximal inequality for finite sub-Gaussian families, see for example \cite[Section 2.5]{vershynin2018high},
\begin{align}\label{eq:variance}
    \E{\Delta \mid \mathcal{S}}
    \leq
    C_2
    \sqrt{
        \frac{\log(|\hat{\mathcal{F}}|+1)}{n}
    }
    \leq
    C_2
    \sqrt{
        \frac{\log(|\mathcal{G}|^2+1)}{n}
    }
    \leq
    \frac{C_3 K }{\sqrt n},
\end{align}
where $C_2,C_3$ are absolute constants and the last inequality uses \eqref{eq:Nbound}.

Taking expectations in \eqref{eq:oracle} with $H=G$, and combining \eqref{eq:bias} and \eqref{eq:variance}, gives
\[
    \E{\W(\hat\mu^{b\mathcal{A},ad},\mu)}
    \leq
    3C_1 \cdot \max\{\log(n),1\}
    \cdot
    \brak{
        n^{-2/(2+d_{\rm loc}(G))}
        +
        n^{-1/d_{\rm max}}
    }
    +
    2C_3 K n^{-1/2}.
\]
Since $d_{\rm loc}(G)\geq 2$ and $d_{\rm max}\geq 2$, both $n^{-2/(2+d_{\rm loc}(G))}$ and  $n^{-1/d_{\rm max}} $
are at least of order $n^{-1/2}$. Thus the variance term is absorbed into the preceding display after increasing the constant. 
This completes the proof. \qed 

\subsection{Remaining proofs}
\label{sec:remaining.proofs.TV}

\begin{proof}[Proof of Lemma~\ref{lem:char.TV}]
We first prove that \((ii)\) implies $(i)$.  
If \(\mu_{I,J}\ll \mu_I\otimes\mu_J\), then a version of the conditional law of
\(X_I\) given \(X_J=x_J\) is given by
\[
    \mu(dx_I\mid x_J)
    =
    f_{I,J}(x_I,x_J)\,\mu_I(dx_I);
\]
this follows e.g.\ by checking $ \mu_{I,J}(A\times B) = \int_B
        \int_A f_{I,J}(x_I,x_J)\,\mu_I(dx_I)
        \mu_J(dx_J)$ for measurable \(A\subseteq\X_I\) and \(B\subseteq\X_J\).
For \(x_J,\tilde x_J\in \X_J\), we then have
\[
\begin{aligned}
\TV\left(\mu(dx_I\mid x_J),\mu(dx_I\mid \tilde x_J)\right)  
& =
\sup_{\|\varphi\|_{L^\infty}\leq 1/2} \left|\int \left(  f_{I,J}(x_I,x_J)-f_{I,J}(x_I,\tilde x_J) \right) \varphi(x_I) \, \mu_I(dx_I) \right|
\\
&=      
\frac12
\left\|
    f_{I,J}(\cdot,x_J)-f_{I,J}(\cdot,\tilde x_J)
\right\|_{L^1(\mu_I)} \\
&\leq
L\|x_J-\tilde x_J\|.
\end{aligned}
\]
This proves \((i)\).

Conversely, assume \((i)\), i.e.\ that $x_J\mapsto \mu(dx_I\mid x_J)$ has a Lipschitz version.  
First note that for every measurable \(A\subseteq\X_I\) with \(\mu_I(A)=0\),
\[
    0=\mu_I(A)=\int  \mu(A\mid x_J)\,\mu_J(dx_J),
\]
so \( \mu(A\mid x_J)=0\) for \(\mu_J\)-a.e. \(x_J\).  
By continuity and the definition of the support, it follows that \( \mu(A\mid x_J)=0\) for all $x_J$ in the support of $\mu_J$.

By the Radon--Nikodym theorem for dominated kernels, there is a jointly measurable density \(f_{I,J}\) such that
\[
     \mu(dx_I\mid x_J) =f_{I,J}(x_I,x_J)\,\mu_I(dx_I)
\]
for all $x_J$ in the support of $\mu_J$.
Consequently,
\[
    \mu_{I,J}(dx_I,dx_J)
    =
    f_{I,J}(x_I,x_J)\,\mu_I(dx_I)\mu_J(dx_J)
\]
and hence \(\frac{d\mu_{I,J}}{d(\mu_I\otimes\mu_J)} = f_{I,J}\). 
Moreover, for $x_J,\tilde x_J$,
\[
\begin{aligned}
\frac12
\left\|
    f_{I,J}(\cdot,x_J)-f_{I,J}(\cdot,\tilde x_J)
\right\|_{L^1(\mu_I)}
&=
\TV\bigl( \mu(\cdot \mid x_J),\mu(\cdot \mid \tilde x_J)\bigr) \\
&\leq L\|x_J-\tilde x_J\|.
\end{aligned}
\]
Hence the section map is \(2L\)-Lipschitz in \(L^1(\mu_I)\). This proves
\((ii)\).
\end{proof}

\begin{proof}[Proof of Corollary~\ref{cor:leb-density-tvlip}]
Observe that 
\[ \frac{d\mu_{I,J} }{d(\mu_I\otimes\mu_J)} 
= \frac{d\mu_{I,J} }{dx_{I,J}} \frac{dx_I}{d\mu_I}\frac{dx_J}{d\mu_J}
=
\frac{f_{I,J}}{ f_I  \cdot f_J} =:g.\]
To show that $x_J\mapsto g(\cdot ,x_J) \in L^1(\mu_I)$ is Lipschitz, using the triangle inequality, write 
\begin{align*}
&\left| g(x_I,x_J) - g( x_I,\tilde x_J)\right| \\
&\leq 
\frac{|f_{I,J}(x_I,x_J) - f_{I,J}(x_I,\tilde x_J)|}{f_I(x_I)f_J(x_J)}
+ \frac{f_{I,J}( x_I,\tilde x_J)}{f_I(x_I)} \left| \frac{1}{f_J(x_J)}  -\frac{1}{f_J(\tilde x_J)} \right| \\
&=:({\rm I}) + ({\rm II}).
\end{align*}
Since $\mu_I(dx_I)=f_I(x_I)dx_I$ and $f_J\geq a$, by Lipschitz continuity, 
\[\int ({\rm I}) \,d\mu_I
\leq \frac{1}{a} \int |f_{I,J}(x_I,x_J) - f_{I,J}(x_I,\tilde x_J)|  \,dx_I
\leq \frac{M \|x_J-\tilde x_J\|}{a}.\]
Similarly, 
\begin{align*}
\int ({\rm II}) \, d\mu_I 
& =
f_J(\tilde x_J) \left| \frac{1}{f_J(x_J)}  -\frac{1}{f_J(\tilde x_J)} \right| \\
&= \frac{|f_J(\tilde x_J) - f_J(x_J)|}{f_J( x_J)}
\leq \frac{M\|x_J-\tilde x_J\|}{a},
\end{align*}
where the last inequality follows from Jensen's inequality.
Thus $\|g(\cdot,x_J) - g(\cdot, \tilde x_J)\|_{L^1(\mu_I)} \leq 2M/a$ and the proof follows from Lemma~\ref{lem:char.TV}.

If the assumptions hold for all nonempty pairs \((I,J)\) in
\eqref{eq:pairs.i.j}, then applying the first part to these pairs gives the
three total-variation Lipschitz inequalities in Assumption~\ref{ass:TV.Lip},
with constant \(M/a\). Hence Assumption~\ref{ass:TV.Lip} holds with constant
\(M/a\).
\end{proof}

    	\begin{proof}[Proof of Proposition \ref{lem:sharp}]
         As already explained in the introduction, it suffices to prove the lower bound $Cn^{-2/(2+d_{\rm loc})}$.
       By assumption, there exists an index $k$ such that $d_{\rm loc} = d_{\pa(k)} + d_k$.
Define $\mathcal{Y} := \X_{\pa(k)} \times \X_k$, and let $\mathcal{Q} \subset \mathcal{P}(\mathcal{Y})$ denote the set of probability measures $\nu$ such that for every $I\subseteq \pa(k)\times\{k\}$ and $J:= (\pa(k)\times\{k\}) \setminus I$, the kernel $x_{I} \mapsto \nu(dx_J \mid x_I)$ is $L$-Lipschitz with respect to total variation.

Clearly every $\nu \in \mathcal{Q}$ can be extended to $\X$ by appending product measures; in other words, for every $\nu \in \mathcal{Q}$, there exists $\mu \in \mathcal{P}_G(\X)$ satisfying Assumption \ref{ass:TV.Lip} and $\mu|_{\mathcal{Y}} = \nu$. Therefore,
\[
\inf_{E_n} \sup_{\mu \in \mathcal{P}_G(\X) \text{ sat.\ Ass.~\ref{ass:TV.Lip}}} \mathbb{E}[\W(\mu, E_n)] 
\geq \inf_{E_n} \sup_{\nu \in \mathcal{Q}} \mathbb{E}[\W(\nu, E_n)].
\]
We claim that the right-hand side admits a  lower bound of order $n^{-2/(2+d_{\rm loc})}$ as a result of known lower bounds from smooth density estimation combined  with  Corollary~\ref{cor:leb-density-tvlip}. 

To that end,  let $\mathcal{R} \subset \mathcal{P}(\mathcal{Y})$ be the set of distributions with 1-Lipschitz Lebesgue-densities. 
Then (see, e.g., \cite[Example 4]{singh2018nonparametric}, \cite{liang2017well, niles2022minimax}),
\begin{align}
    \label{eq:denisty.lower.bound}
\inf_{E_n} \sup_{\nu \in \mathcal{R}} \mathbb{E}[\W(\nu, E_n)] \geq C n^{-2/(2 + d_{\rm loc})}
\end{align}
for some absolute constant $C > 0$.

By  Corollary~\ref{cor:leb-density-tvlip}, the set $\mathcal{R}$ is essentially contained in $\mathcal{Q}$, up to requiring a lower bound on the densities. This technical issue can be circumvented as follows.
The proof of \eqref{eq:denisty.lower.bound} is based on a standard non-asymptotic technique, namely constructing a finite subset $\mathcal{R}' \subset \mathcal{R}$ of exponentially many elements with pairwise $\W$-distances uniformly bounded below.
The estimate in \eqref{eq:denisty.lower.bound} is then shown to hold for $\mathcal{R}'$ in place of $\mathcal{R}$, which trivially implies \eqref{eq:denisty.lower.bound}.
(See the proof of Theorem \ref{lem:particulargraphlower} for the details of this method in the present setting.)
Since the bound only requires considering $\nu \in \mathcal{R}'$, the standard construction can be adapted by replacing each $\nu$ with $\gamma := \frac{1}{2}(\nu + \mathcal{U})$, where $\mathcal{U}$ denotes the uniform distribution on $\mathcal{Y}$. Note that
$
\W(\gamma, \tilde{\gamma}) = \frac{1}{2} \W(\nu, \tilde{\nu})$ for all $\nu, \tilde{\nu}$, 
so the minimax risk over these smoothed distributions is still bounded below:
\[
\inf_{E_n} \sup_{\gamma = \frac{1}{2}(\nu + \mathcal{U}),\; \nu \in \mathcal{R}'} \mathbb{E}[\W(\gamma, E_n)] \geq \frac{C}{2} n^{-2/(2 + d_{\rm loc})}.
\]

Finally, each such $\gamma$ has a $1$-Lipschitz density that is bounded from below by $\frac{1}{2}$.
Since $L\geq 2$, it follows from Corollary~\ref{cor:leb-density-tvlip} that  $\gamma \in\mathcal{Q}$, which completes the proof.
		\end{proof}

 \section{Lower bounds without continuity}
    \label{sec:lowerbounds}
    
    In this section, we establish that even under a known graph structure, the minimax learning rate remains of order $n^{-1/d}$ unless quantitative continuity assumptions are imposed. 
    Note that this rate matches the one obtained in the fully agnostic setting. 
    The results are based on technical adaptations of existing methods to our setting.
    
	We first recall the following result, which essentially follows from  \cite{chewi2024statistical}.
    
	\begin{theorem}\label{thm:basiclower}
        There exists an absolute constant $C>0$  such that for all $d\geq 1$ and $n \geq 8$, 
		\begin{align}
		    \label{eq:lower.bound.general.measures}
\inf_{E_n} \sup_{\mu \in \mathcal{P}([0, 1]^d)} \int \W(E_n, \mu) \,\mu^{\otimes n}
        \geq C \, n^{-1/\max\{d,2\}},
		\end{align}
		where the infimum ranges over all measurable maps $E_n : ([0, 1]^d)^n \rightarrow \mathcal{P}([0, 1]^d)$.
	\end{theorem}
    \begin{proof}
        The case $d\geq 3$ follows directly from {\cite[Theorem 2.15]{chewi2024statistical}}.
        For the case $d=1,2$, denoting by ${\rm m}(\nu)$ the mean of a probability measure, we have that $\W(E_n, \mu) \geq \|{\rm m}(E_n)-{\rm m}(\mu)\|$.
        Hence, \eqref{eq:lower.bound.general.measures} is an immediate consequence of the classical result that the minimax rate for mean estimation is $Cn^{-1/2}$, see, e.g., Section 2 in \cite{tsybakov}.     
    \end{proof}

	The same lower bound as in the theorem happens to be true even if one restricts in \eqref{eq:lower.bound.general.measures}  to seemingly `small' subsets of $\mathcal{P}([0, 1]^d)$.
    Relevant to our setting, one particular instance of such a small subset is the set of all measures with \emph{conditional} independences, denoted by $\mathcal{P}_{\rm CI}$. 
    Formally, $\mu\in \mathcal{P}_{\rm CI}$ if and only if for $X\sim \mu$ and any pairwise disjoint and non-empty sets  $I,J,J'\subset\{1,\dots,K\}$, conditionally on $X_I$, the random vectors $X_J$ and $X_{J'}$ are independent.

    Note that $\mathcal{P}_{\rm CI}$ is indeed relatively small from the perspective of this article---for instance, $\mathcal{P}_{\rm CI}\subset \mathcal{P}_G$ for the Markov graph $G$ (or, more generally, any graph $G$ that has a single root node such as many tree-like graphs).

	\begin{theorem}\label{thm:lowercond}
    The set $\mathcal{P}_{\rm CI}$ is dense in $\mathcal{P}([0,1]^d)$ with respect to the weak convergence of probability measures.
    Moreover, there is an absolute constant $C>0$ for which, if $d\geq 3$ and $n\geq 8$,
    \begin{align}
		    \label{eq:lower.bound.P.CI}
            \inf_{E_n} \sup_{\mu \in \mathcal{P}_{\rm CI}} \int \W(E_n, \mu) \,\mu^{\otimes n}
        \geq C \, n^{-1/d},
		\end{align}
    \end{theorem}

   \begin{corollary}\label{cor:graphcondind}
		For any graph $G$ that has a single root node (e.g.\ Markov graphs), \eqref{eq:lower.bound.P.CI} holds true with $\mathcal{P}_{G}$ instead of $\mathcal{P}_{\rm CI}$.
	\end{corollary}

		\begin{proof}[{Proof of Theorem \ref{thm:lowercond}}]
        We start by proving the first claim in the theorem.
        To that end, recall that the set of discrete measures $\mu=\frac{1}{|\mathcal{X}'|} \sum_{x\in\mathcal{X}'} \delta_{x}$ with finite $\mathcal{X}'\subset[0,1]^d$ is weakly dense in $\mathcal{P}([0, 1]^d)$. 
        Thus it suffices to show that $\mathcal{P}_{\rm CI}$ is dense in the set of those discrete measures.
        
         Fix some $\mu$ as above, defined using $\mathcal{X}'\subset [0,1]^d$.
        Next, for  small $\varepsilon>0$ consider the set $\mathcal{X}^\varepsilon$ which is obtained from $\mathcal{X}'$ by changing  each $x\in\mathcal{X}'$ at most by a distance $\varepsilon$ in a way such that the following holds: 
        \begin{align*}
            &\text{for all distinct $x,y\in\mathcal{X}^\mathcal{\varepsilon}$ and all $k=1,\dots, K$ : $x_k\neq y_k$}.
        \end{align*}

        For the resulting measure $\mu ^\varepsilon$, if  $X \sim \mu^\varepsilon$ and $I,J,J    '\subset \{1,\dots,K\}$ are pairwise disjoint (and $I$ is non-empty), we have that conditionally on $X_I$, the random vectors $X_J$ and $X_{J'}$ are deterministic (as knowing any entry $x_k$ for $k \in I$ completely determines the whole vector $x \in \mathcal{X}^{\mathcal{\varepsilon}}$)---thus independent.
    Hence, $\mu^\varepsilon \in \mathcal{P}_{\rm CI}$ and it is clear that $\mu^\varepsilon\to\mu$ as $\varepsilon\to 0$.
    This completes the proof of the first statement.

    \vspace{0.5em}
    We proceed with the proof of \eqref{eq:lower.bound.P.CI}.
    First observe that \eqref{eq:lower.bound.P.CI} does not directly follow from the denseness of $\mathcal{P}_{\rm CI}$, as $E_n$ need not be continuous. 
    However, the proof for the lower bound of Theorem \ref{thm:basiclower} presented in \cite{chewi2024statistical} (see Theorem 2.15 and its proof therein) shows that it suffices to restrict the supremum to measures supported on the mid points of a grid of side length approximately $n^{-1/d}$. 
    Perturbing the mid points slightly as in the first part of this proof yields that all probability measures supported on the grid correspond to completely deterministic relations across dimensions, and thus are contained in $\mathcal{P}_{\rm CI}$. 
    As the structure of the support (aside from the distance between points) plays no role in the proof given in \cite{chewi2024statistical}, the arguments trivially extend to the present setting.
		\end{proof}

	The obvious next question is whether graphs which have several root nodes (i.e., imply several unconditional independencies) still lead to the same lower bound. While the general answer is open and will not be provided in this paper, we instead focus on one extreme case of this sort, where indeed the same lower bound is true: 
	
    \begin{proposition}\label{lem:particulargraphlower}
		There exists an absolute constant $C$ such that the following holds.
        Let $K\geq 3$ and $G$ be the graph with nodes $1, \dots, K$ only including the edges $k \rightarrow K$ for $k=1, \dots, K-1$. 
        Then, for every $n \geq 1$, 
		\[
		\inf_{E_n} \sup_{\mu \in \mathcal{P}_G(\X)} \int \W(E_n, \mu) \,d\mu^{\otimes n} \geq C \, n^{-1/d},
		\]
		where the infimum ranges over all measurable maps $E_n : \X^n \rightarrow \mathcal{P}(\X)$.
\begin{proof}	
			We follow the standard minimax approach from \cite[Section 2.2]{tsybakov} using a lower bound via decision rules: if $s>0$ and  $\mathcal{Q}\subset \mathcal{P}_G(\X)$ is a finite family satisfying that $\W(\mu,\nu)\geq 2s$ for any distinct $\mu,\nu\in\mathcal{Q}$, then 
			\begin{align}
            \label{eq:tsybakov.lower}
			\inf_{E_n} \sup_{\mu \in \mathcal{P}_G(\X)} \int \W(E_n, \mu) \,d\mu^{\otimes n} \geq s \, \inf_{\psi_n} \max_{\mu\in\mathcal{Q}} \mu^{\otimes n}(\psi_n \neq \mu),
			\end{align}
            where the infimum is taken over all so-called decision rules  $\psi_n : ([0,1]^d)^n \rightarrow \mathcal{Q}$.

            In order to apply this result, let $\beta\in \mathbb{N}$ to be specified in what follows (in Step 4 below) and partition $[0, 1]^d=\X_{1:K-1} \times \X_K$ into $\beta^d$ many cubes.
            For a finite set $A\subset \X_{1:K-1}$ or $A\subset \X_K$, we denote by $\mathcal{U}_A$ the (discrete) uniform distribution on $A$.

            \vspace{0.5em}
            \noindent
            \emph{Step 1:} We start by constructing a family of measures on $\X_K$.
            
            Denote the centres of the cubes in $\X_{K}$ by $\mathcal{M}_K$; thus $|\mathcal{M}_K|=\beta^{d_K}$.
            By a version of the Varshamov-Gilbert bound \cite[Lemma 2.9]{tsybakov}, there is a family $\mathcal{S} \subset 2^{\mathcal{M}_K}$ satisfying 
            \begin{align}
                \label{eq:properties.S}
                |\mathcal{S}| \geq 2^{ |\mathcal{M}_K| /8} \quad\text{and}\quad |S| \geq \frac{|\mathcal{M}_K|}{8} \quad\text{and}\quad |(S \setminus \tilde{S}) \cup (\tilde{S}\setminus S)| \geq \frac{|\mathcal{M}_K|}{8}
            \end{align}
            for all $S, \tilde{S} \in \mathcal{S}$.
            Set $\nu_0= \mathcal{U}_{\mathcal{M}_K}$ and for $S \in \mathcal{S}$, put $\nu_S := \frac{1}{2} ( \mathcal{U}_{\mathcal{M}_K} + \mathcal{U}_{S})$. 
            The following two observations follow from \eqref{eq:properties.S}:
			\begin{itemize}
				\item[(a)] The density satisfies $\frac{d\nu_0}{d\nu_S} \in[\frac{1}{5},2]$.
				\item[(b)] For distinct $\nu,\nu' \in \{\nu_0\} \cup \{\nu_S\mid S \in \mathcal{S}\}$, we have $\TV(\nu, \nu') \geq \frac{1}{32}$.
			\end{itemize}
		

             \vspace{0.5em}\noindent
            \emph{Step 2:} We proceed to construct kernels from $\X_{1:K-1}$ to $\X_K$ and measures on $\X_{1:K}$.
            
           Denote by $\mathcal{M}_{1:K-1}$ the centres of the cubes in $\X_1 \times \dots \times \X_{K-1}$; thus $|\mathcal{M}_{1:K-1}|= \beta^{d_{1:K-1}}$.
             By Lemma \ref{lem:appendixbound} (applied with $\mathcal{M}_{1:K-1}$ and $\{ \nu_S \mid S \in \mathcal{S}\}$), there is  a set $\mathcal{R}$ of kernels $R : \mathcal{M}_{1:K-1}\rightarrow \{ \nu_S \mid S \in \mathcal{S}\}$ satisfying 
             \[   \mathcal{U}_{\mathcal{M}_{1:K-1}}(R \neq R') \geq  \frac{1}{8} \]
             for any distinct $R, R' \in \mathcal{R}$, and
             \[ |\mathcal{R}| 
             \geq   \frac{1}{2} |\{\nu_S : S\in\mathcal{S}\}|^{C |\mathcal{M}_{1:K-1}|}
             \geq   \frac{1}{2} \left(2^{|\mathcal{M}_K|/8} \right)^{C |\mathcal{M}_{1:K-1}|}
             =  \frac{1}{2} 2^{C \beta^d/8}, \]
             where $C>0$ is  an absolute constant.
            Define 
            \[ \mathcal{Q} := \left\{ \mu_R := \mathcal{U}_{\mathcal{M}_{1:K-1}} \otimes R : R \in \mathcal{R} \right\} \cup \left\{ \mu_0 := \mathcal{U}_{\mathcal{M}_{1:K-1}} \otimes \mathcal{U}_{\mathcal{M}_{K}} \right\}.\]

             \vspace{0.5em}
            \emph{Step 3:}
            Observe that clearly $\mathcal{Q}\subset\mathcal{P}_G$ by the definition of the graph.
            Next, since
            \[\frac{d\mathcal{U}_{\mathcal{M}_{1:K-1}} \otimes R}{d\mathcal{U}_{\mathcal{M}_{1:K-1}} \otimes R'}(x) = \frac{dR(x_{1:K-1})}{dR'(x_{1:K-1})}(x_K),\]
           for any $\mu_R,\mu_{R'}\in\mathcal{Q}$, we have that 
			 \[
			 \TV(\mu_R, \mu_{R'}) = \int \TV(R, R') \,d\mathcal{U}_{\mathcal{M}_{1:K-1}}\geq \frac{1}{8\cdot32}.
			 \]
             Moreover, since $\mu_R,\mu_{R'}$ are supported on the same grid of size $\frac{1}\beta$, it follows that $\W(\mu_R, \mu_{R'}) \geq \frac{1}{\beta} \TV(\mu_R, \mu_{R'}) \geq s$ for $s := \frac{1}{512 \beta}$.

            \vspace{0.5em}\noindent
             \emph{Step 4:} Computation of the lower bound.
             
             For any decision rule $\psi_n\colon ([0,1]^d)^n \to \mathcal{Q}$, using that $\frac{d\mu_0}{d\mu_R}\geq \frac{1}{5}$, it follows that
			  \begin{align*}
			 \mu_0^{\otimes n}(\psi_n \neq \mu_0 ) 
             &= \sum_{R \in \mathcal{R}} \mu_0^{\otimes n}(\psi_n = \mu_R)  \\
             &\geq \frac{|\mathcal{R}|}{5^n} \frac{1}{|\mathcal{R}|} \sum_{R \in \mathcal{R}} \mu_R^{\otimes n}(\psi_n = \mu_R) =: \frac{|\mathcal{R}|}{5^n} p_{\mathcal{R}}
			 \end{align*}
			 and thus
			 \begin{align*}
				\max\left\{\mu_0^{\otimes n}(\psi_n \neq \mu_0) \,,\, \max_{R \in \mathcal{R}} \mu_R^{\otimes n}(\psi_n \neq \mu_R)\right\} 
                &\geq \max\left\{\frac{|\mathcal{R}|}{5^n} p_{\mathcal{R}} \, , \, \frac{1}{|\mathcal{R}|} \sum_{R \in \mathcal{R}} \mu_R^{\otimes n}(\psi_n \neq \mu_R)\right\} \\
				&\geq \max\left\{\frac{|\mathcal{R}|}{5^n} p_{\mathcal{R}} \, ,\, 1-p_{\mathcal{R}}\right\}=:({\rm I}).
			 \end{align*}
             
             Recall that $|\mathcal{R}|\geq \frac{1}{2}2^{C\beta^d}$ and let $\beta$ be the smallest integer for which $\frac{1}{2} 2^{C\beta^d}\geq 5^n$; thus $\beta \leq C' n^{1/d}$ for some absolute constant $C'$.
            With this choice of $\beta$ clearly $({\rm I})\geq \frac{1}{2}$ and therefore, by \eqref{eq:tsybakov.lower},
            \[  \inf_{E_n} \sup_{\mu \in \mathcal{P}_G(\X)} \int \W(E_n, \mu) \,d\mu^{\otimes n} 
            \geq  \frac{s}{2}
            =  \frac{1}{1024 \beta}
            \geq  \frac{1}{C'1024} n^{-1/d},\]
            completing the  proof.
		\end{proof}
	\end{proposition}
    
	\begin{remark}
		There are quite a few graphs for which we can easily derive the $n^{-1/d}$ lower bound by combining Theorem \ref{thm:lowercond} and Theorem \ref{lem:particulargraphlower}, such as for instance graphs like $1 \rightarrow 2 \rightarrow 3 \leftarrow 4$.
		
		A critical case for a graph which is not covered by the above results is \[1\rightarrow 2\leftarrow 3 \rightarrow 4 \leftarrow 5\rightarrow 6\leftarrow 7.\] It is open to us at this point what the right lower bound for $\mathcal{P}_G$ for this graph should be without continuity assumptions.
	\end{remark}    

\vspace{1em}
\noindent
{\bf Acknowledgements:}
The authors are grateful to the associate editor and the reviewers for thoughtful comments and suggestions which significantly improved the paper.
Daniel Bartl is grateful for financial support through the Austrian Science Fund [doi: 10.55776/P34743 and 10.55776/ESP31], the Austrian National Bank [Jubil\"aumsfond, project 18983], and a Presidential-Young-Professorship grant [`Robust Statistical Learning from Complex Data'].
Stephan Eckstein is grateful for support by the German Research Foundation through Project 553088969 as well as the Cluster of Excellence “Machine Learning --- New Perspectives for Science” (EXC 2064/1 number 390727645).

\bibliography{smoothbib}
\bibliographystyle{abbrv}

\appendix

\section{Proofs for Section \ref{sec:wlip}}
\label{sec:proofs.W}

\begin{proof}[Proof of Lemma \ref{lem:Wdecomp}]
        For every $j=1,\dots,K$ and $m^j=(m^j_k)_{k=1}^j\in (\R_+)^j$, let
        \[  d_{m^j}(x_{1:j},y_{1:j}) := \sum_{k=1}^j m^j_k \|x_k - y_k\|\]
        and set $\W_{m^j}$ to be first order Wasserstein distance on $\X_{1:j}$ with respect to $d_{m^j}$.
        In particular, $\W\leq \W_{m^K}$ for $m^K=(1,\dots,1)$.

        \vspace{0.5em}\noindent
        \emph{Step 1:} We claim that for every $j\geq 2$, $m^j\in\R_+^j$ and $\mu,\nu\in \mathcal{P}_G(\X)$,
        \begin{align} 
        \label{eq:wasserstein.basic.recursion}
        \begin{split}
        \W_{m^j}(\mu_{1:j},\nu_{1:j})
       & \leq \W_{m^{j-1}}(\mu_{1:j-1},\nu_{1:j-1})  \\
       &\quad + m^j_j  \int \W(\mu(\cdot \mid y_{\pa(j)}), \nu(\cdot \mid y_{\pa(j)})) \, \nu(dy),
       \end{split}
       \end{align}
        where  $m^{j-1}\in \R_+^{j-1}$ is defined by
        \begin{align}
            \label{eq:recursive.def.m}
            m^{j-1}_k := m^j_k + L m^j_j \eins_{\pa(j)}(k), \quad k=1,\dots,j-1.
        \end{align}

        To prove \eqref{eq:wasserstein.basic.recursion}, first observe that
        \begin{align}
        \label{eq:W.leq.AW}
        \W_{m^j}(\mu_{1:j}, \nu_{1:j})
        &\leq \inf_{\pi \in \Pi(\mu_{1:j-1}, \nu_{1:j-1})} \int d_{(m^j_1,\dots,m^j_{j-1})}(x_{1:j-1}, y_{1:j-1})  \\
        &\qquad + m^j_j \W(\mu(\cdot \mid x_{\pa(j)}), \nu(\cdot \mid y_{\pa(j)})) \,\pi(dx_{1:j-1}, dy_{1:j-1}).
        \nonumber
        \end{align}
        Indeed, by a standard measurable selection argument (see, e.g., \cite[Proposition 7.50(b)]{Bertsekas1978}), there exists a universally measurable map assigning to each pair $(x_{\pa(j)}, y_{\pa(j)})$ an optimal coupling $\gamma^{(x_{\pa(j)}, y_{\pa(j)})}$ for the Wasserstein distance between the conditional measures $\mu(\cdot \mid x_{\pa(j)})$ and $\nu(\cdot \mid y_{\pa(j)})$
        (recall that the existence of such optimal couplings is ensured, for instance, by \cite[Theorem 4.1]{villani2008optimal}).
        In particular, for every $\pi \in \Pi(\mu_{1:j-1}, \nu_{1:j-1})$, the measure 
        \[\Gamma(dx_{1:j},dy_{1:j}) := \pi(dx_{1:j-1},dy_{1:j-1})\otimes \gamma^{(x_{\pa(j)},y_{\pa(j)})}(dx_j,dy_j)  \]
        is well-defined.
        Moreover, one readily checks that $\Gamma$ is a coupling between $\mu_{1:j}$ and $\nu_{1:j}$, from which \eqref{eq:W.leq.AW} follows.

        Next observe that,  by Assumption \ref{ass:W.Lip},
         \begin{align*}
          \W(\mu(\cdot \mid x_{\pa(j)}), \nu(\cdot \mid y_{\pa(j)})) 
        &  \leq  L \|x_{\pa(j)}-y_{\pa(j)} \| +  \W(\mu(\cdot \mid  y_{\pa(j)}), \nu(\cdot \mid y_{\pa(j)})) .
        \end{align*}
        Finally, by the definitions of the metrics $d_{m^j}$ and $d_{m^{j-1}}$ and the definition of $m^{j-1}$,
        \begin{align*}
           & d_{(m^j_1,\dots,m^j_{j-1})}(x_{1:j-1},y_{1:j-1}) + L m^j_j \|x_{\pa(j)}-y_{\pa(j)} \| \\
        &\leq d_{(m^{j-1}_1,\dots,m^{j-1}_{j-1})}(x_{1:j-1},y_{1:j-1}).
        \end{align*}
        Concatenating \eqref{eq:W.leq.AW} with the two inequalities above completes the proof of \eqref{eq:wasserstein.basic.recursion}.

        \vspace{0.5em}
        \emph{Step 2:} 
        It follows from Step 1 and a simple induction that 
         \[ \W(\mu,\nu)
        \leq \sum_{j=1}^K m^j_j  \int \W(\mu(\cdot \mid y_{\pa(j)}), \nu(\cdot \mid y_{\pa(j)})) \nu(dy) \]
        where $m^j$ is given recursively by \eqref{eq:recursive.def.m} starting with $m^K=(1,\dots,1)$.
        Thus, to complete the proof, we are left to show that $m^j_j = M_{L,j}$, the latter being defined in the assertion of the lemma. To see this, one verifies that $m^j_i$ arise from a standard dynamic programming approach to calculate $M_{L, j}$ and we leave the details to the reader;\footnote{Cf.~\cite{cormen2009introduction} for a standard reference; notably, the dynamic programming procedure in this proof is very similar to \cite[Exercise 24.2-4]{cormen2009introduction}, where the total number of paths in a DAG is counted.} an exemplary case is shown in Figure \ref{fig:constants}.
		\end{proof}

    \begin{figure}
        \centering
        \includegraphics[width=0.9\linewidth]{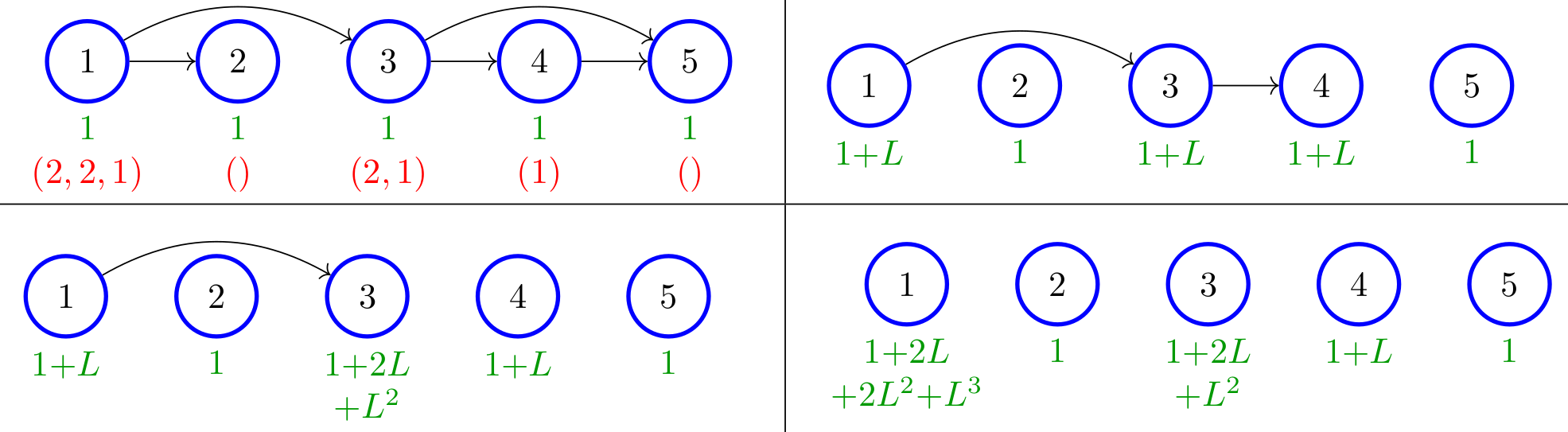}
        \caption{Exemplification of the constants occurring in Lemma \ref{lem:Wdecomp}. The red numbers indicate the number of outgoing paths of different lengths (e.g., $(2, 1)$ below node 3 indicates that there are 2 paths of length 1, and 1 path of length 2 outgoing). The green numbers indicate how the constants for the cost change in the backward induction of the proof of Lemma \ref{lem:Wdecomp}. At the end of the backward induction (bottom right), the red numbers indicate the constants for each node, e.g., $2L+2L^2+L^3$ corresponds to $(2, 2, 1)$ for node 1.}
        \label{fig:constants}
    \end{figure}

\begin{proof}[Proof of Lemma \ref{lem:conditional.iid}]
        Let $X\sim \mu$ and let $(X^i)_{i=1}^n$ be an i.i.d.\ sample of $X$.
        For shorthand notation, set $(Y,Z):=(X_{\pa(k)},X_k)$, similarly for $Y^i$ and $Z^i$.
        Moreover, write $c=c_{\pa(k)}$.
        Thus, by Lemma \ref{lem:muAwelldefined}, $\mu^{\mathcal{A}}(c)=\mu(c)=\mathbb{P}(Y\in c)$ and by \eqref{eq:nu.A.kernels.identity} for every measurable set $B\subset \X_k$, $\mu^{\mathcal{A}}(B\mid c)  = \mathbb{P}(Z\in B \mid Y\in c)$.

        Denote by $I$ the (random) set of indices $i\leq n$ for which $Y^i\in  c$ so that $\hat{\mu}^\mathcal{A}(\cdot \mid c) = \frac{1}{|I|}\sum_{i\in I} \delta_{Z^i}$  by the definition of  $\hat{\mu}^\mathcal{A}$.
        Thus, it suffices to show that, conditionally on $|I|=m$, the random vector $(Z^i)_{i\in I}$ has the same distribution as an i.i.d.\ sample of size $m$ from $\mu^\mathcal{A}(d x_k \mid c)$; that is, $(Z^i)_{i\in I} \sim (\mu^\mathcal{A}( dx_k \mid c))^{\otimes m}$. Equivalently, for a measurable set $B\subseteq (\X_k)^m$ we need to show $\mathbb{P}((Z^i)_{i \in I} \in B , |I|=m) = (\mu^{\mathcal{A}}(\cdot\mid c))^{\otimes m}(B) \,\cdot\,\mathbb{P}(|I|=m)$.
        
        To that end,  note that
        \begin{align}
        \nonumber
           & \mathbb{P}\left( (Z^i)_{i\in I} \in B ,\,  |I|=m\right)\\
            \nonumber
            & =  \sum_{J\subset [n] \, : \, |J|=m}  \mathbb{P}\left( (Z^i)_{i\in J} \in B, \, \forall i\in J:  Y^i\in c,  \, \forall i\notin J: Y^i\notin c\right) \\
            & = \sum_{J\subset [n] \, : \, |J|=m}  \mathbb{P}\left( (Z^i)_{i\in J} \in B, \, \forall i\in J:  Y^i\in c \right) \mathbb{P}\left(  Y\notin c\right)^{n-m} 
            \label{eq:sum.subset.J}
        \end{align}
        where we have used independence of the sample in the last equality.
        Since 
        \[\mathbb{P}\left( (Z^i)_{i\in J} \in B, \, \forall i\in J:  Y^i\in  c  \right)
        = (\mu( dx_k \mid c))^{\otimes m}(B) \mu(c)^m\]
        and $\mu(dx_k \mid c ) = \mu^\mathcal{A}(dx_k\mid c)$, the claim readily follows noting that there are $\binom{n}{m}$-many subsets $J$ in \eqref{eq:sum.subset.J} and that $\binom{n}{m}  \mu(c)^m (1-\mu(c))^{n-m} = \mathbb{P}\left(|I|=m\right)$.
    \end{proof}
    
\begin{proof}[Proof of Theorem \ref{thm:estWLip}]
\emph{Step 1:} It follows from Lemma \ref{lem:Wdecomp} that
		\begin{align*}
			\W(\mu, \hat{\mu}^{\mathcal{A}}) 
            &\leq \sum_{k=1}^K M_{L, k} \int \W(\mu(\cdot \mid y_{\pa(k)}), \hat{\mu}^{\mathcal{A}}(\cdot \mid y_{\pa(k)})) \, \hat{\mu}^{\mathcal{A}}(dy_{\pa(k)}) .
		\end{align*}
        Moreover, for every $k\leq K$, since  $\mu^{\mathcal{A}}(\cdot \mid y_{\pa(k)})$ is an average of measures of the form $\mu(\cdot \mid x_{\pa(k)})$ over $x_{\pa(k)}$ that satisfy $\|x_{\pa(k)} - y_{\pa(k)} \| \leq \delta_{\mathcal{A}}$, the triangle inequality together with Assumption \ref{ass:W.Lip} implies that
        \begin{align*}
            &\W(\mu(\cdot \mid y_{\pa(k)}), \hat{\mu}^{\mathcal{A}}(\cdot \mid y_{\pa(k)})) 
            \leq  L  \delta_{\mathcal{A}} +  \W(\mu^{\mathcal{A}}(\cdot \mid y_{\pa(k)}), \hat{\mu}^{\mathcal{A}}(\cdot \mid y_{\pa(k)}))  .
        \end{align*}
        Finally, since $\mu^{\mathcal{A}}(\cdot \mid y_{\pa(k)})$ and $\hat{\mu}^{\mathcal{A}}(\cdot \mid y_{\pa(k)}) $ are constant in $y_{\pa(k)}$ as long as $y_{\pa(k)}$ belongs to a fixed cell $c\in\mathcal{A}_{\pa(k)}$, since $\mu^{\mathcal{A}}(\cdot \mid c) = \mu(\cdot \mid c)$, we get
        \begin{align}
        \label{eq:wasserstein.kernels.split.to.cells}
			\W(\mu, \hat{\mu}^{\mathcal{A}}) 
            &\leq \sum_{k=1}^K M_{L, k} \left( L\delta_{\mathcal{A}} + \sum_{c\in\mathcal{A}_{\pa(k)}} \hat{\mu}^{\mathcal{A}}(c) \W(\mu^{\mathcal{A}}(\cdot \mid c), \hat{\mu}^{\mathcal{A}}(\cdot \mid c) ) \right).
		\end{align}

\vspace{0.5em}
\noindent
    \emph{Step 2:}
    Fix $k\leq K$, and let $c\in \mathcal{A}_{\pa(k)}$.
    Observe that if $\mu^{\mathcal{A}}(c)=0$, then $\hat{\mu}^{\mathcal{A}}(c)=0$ almost surely.
    Otherwise, for $m\leq n$, by Lemma \ref{lem:conditional.iid}, conditionally on the event $n\hat{\mu}^{\mathcal{A}}(c)=m$, $\hat{\mu}^{\mathcal{A}}(\cdot \mid c)$ has the same distribution as the empirical measure of $\mu^{\mathcal{A}}(\cdot \mid c)$ with sample size $m$.
    Thus, setting $\bar{d}_k:=\max\{2,d_k\}$, it follows from \eqref{eq:classical.convergence.wasserstein.in.proof} that
    \begin{align*} 
    \mathbb{E}\left[ \W(\mu^{\mathcal{A}}(\cdot \mid c)), \hat{\mu}^{\mathcal{A}}(\cdot \mid c)) \mid n\hat{\mu}^{\mathcal{A}}(c)=m \right] 
    &\leq  8 l_m(d_k) m^{-1/\bar{d}_k}\\
    &\leq  8 l_n(d_k) \left( n\hat{\mu}^{\mathcal{A}}(c) \right) ^{-1/\bar{d}_k}.
    \end{align*}
    Therefore,  by the tower property,
    \begin{align*}
        & \mathbb{E}\left[\sum_{c\in\mathcal{A}_{\pa(k)}}  \hat{\mu}^{\mathcal{A}}(c) \W(\mu^{\mathcal{A}}(\cdot \mid c), \hat{\mu}^{\mathcal{A}}(\cdot \mid c))   \right] \\
        &=\sum_{c\in\mathcal{A}_{\pa(k)}} \mathbb{E}\left[ \hat{\mu}^{\mathcal{A}}(c)  \mathbb{E}\left[ \W(\mu^{\mathcal{A}}(\cdot \mid c)), \hat{\mu}^{\mathcal{A}}(\cdot \mid c)) \mid n\hat{\mu}^{\mathcal{A}}(c) \right] \right] \\
        &\leq \sum_{c\in\mathcal{A}_{\pa(k)}}  \mathbb{E}\left[ \hat{\mu}^{\mathcal{A}}(c)  8l_n(d_k) \left( n\hat{\mu}^{\mathcal{A}}(c) \right)^{-1/\bar{d}_k} \right] =: 8\cdot l_n(d_k) \cdot ({\rm I}).
    \end{align*}
    Moreover, by an application of Jensen's inequality,
   \begin{align*}
    ({\rm I}) 
     &=   \frac{|\mathcal{A}_{\pa(k)}|}{n} \mathbb{E}\left[ \frac{1}{|\mathcal{A}_{\pa(k)}|}\sum_{c\in\mathcal{A}_{\pa(k)}}   \left( n\hat{\mu}^{\mathcal{A}}(c) \right)^{1-1/\bar{d}_k} \right] \\
     &\leq \frac{|\mathcal{A}_{\pa(k)}|}{n} \mathbb{E}\left[\left( \frac{1}{|\mathcal{A}_{\pa(k)}|}\sum_{c\in\mathcal{A}_{\pa(k)}}    n\hat{\mu}^{\mathcal{A}}(c) \right)^{1-1/\bar{d}_k} \right] \\
    & =\frac{|\mathcal{A}_{\pa(k)}|}{n} \left( \frac{n}{|\mathcal{A}_{\pa(k)}|} \right)^{1-1/\bar{d}_k} 
    = \left( \frac{n}{|\mathcal{A}_{\pa(k)}|} \right)^{-1/\bar{d}_k}.
    \end{align*}

\vspace{0.5em}
\noindent
    \emph{Step 3:}
    By combining    Step 1 and Step 2, 
    \begin{align}
    \label{eq:recursive.split.some.proof}
    \mathbb{E}\left[ \W(\mu, \hat{\mu}^{\mathcal{A}}) \right]
    &\leq \sum_{k=1}^K M_{L, k} \left( L\delta_{\mathcal{A}} + 8 l_n(d_k) \left( \frac{n}{|\mathcal{A}_{\pa(k)}|} \right)^{-1/\bar{d}_k} \right),
    \end{align}
    and it remains to estimate the last expression.
    By the choice of $\eta$ in the theorem, namely $\eta=\lfloor \frac{1}{d_{\rm loc}} \log_2(n) \rfloor$, we have that 
    \begin{align*}
        \delta_\mathcal{A}
        =  2^{-\eta} 
        \leq 2^{- \frac{1}{d_{\rm loc}} \log_2(n)  +1} =2  n^{-1/d_{\rm loc}}, 
    \end{align*}
    and
    \begin{align*}    
        |\mathcal{A}_{\pa(k)}|
        &= 2^{\eta d_{\pa(k)}} \leq n^{d_{\pa(k)}/d_{\rm loc}}.
    \end{align*}
    Therefore,
    \begin{align*}
    L\delta_{\mathcal{A}} + 8l_n(d_k) \left( \frac{n}{|\mathcal{A}_{\pa(k)}|} \right)^{-1/\bar{d}_k} 
    &\leq 2L n^{-1/d_{\rm loc}} + 8l_n(d_k)\left( n^{1-d_{\pa(k)}/d_{\rm loc}} \right)^{-1/\bar{d}_k} \\
    &=:2L n^{-1/d_{\rm loc}} + ({\rm II}).
    \end{align*}

    Finally, it remains to estimate the term $({\rm II})$.
    If $d_k\neq 2$ then $l_n(d_k)=1$ and, since $d_{\rm loc}\geq d_{\pa(k)}+\bar{d}_k$ by definition and thus $1-\frac{d_{\pa(k)}}{d_{\rm loc}} \geq \frac{\bar{d}_k}{d_{\rm loc}}$, we have $({\rm II}) \leq 8 n^{-1/d_{\rm loc}}$.
    If $d_k=2$ and $d_{\rm loc}$ is not attained for this node, then $d_{\rm loc}\geq d_{\pa(k)}+\bar{d}_k+1$. 
    Using that $\log(n)\leq r n^{1/r}$ for all $r,n\geq 1$, a straightforward calculation shows that $({\rm II}) \leq 16 d_{\rm loc} n^{-1/d_{\rm loc}}$.
    Ultimately, if $d_k=2$ and $d_{\rm loc}$ is attained for this node, then clearly $({\rm II})\leq 8 l_n(d_k) n^{-1/d_{\rm loc}}$.  
Hence the proof follows from \eqref{eq:recursive.split.some.proof}.  
\end{proof}

\begin{proof}[Proof of Theorem \ref{thm:estWLip.adaptoive}]
    The proof follows from the same arguments as the proof presented for Theorem \ref{thm:main.TV.adaptive}.
    The only difference is that we set $\hat\nu_G$ to be the estimator from Definition \ref{def:mu.A} with the corresponding bandwidth for $\eta$ given in Theorem \ref{thm:estWLip}.
    The rest of the proof follows essentially verbatim, with Theorem \ref{thm:estTV} and \eqref{eq:bias} replaced by Theorem \ref{thm:estWLip} and 
    \[ \mathbb{E}\left[ \W(\mu, \hat{\nu}_G) \right] 
        \leq C \cdot \max\{1,\log(n)\} \cdot n^{-1/d_{\rm loc}},\]
    respectively. 
    We leave the details to the reader.
\end{proof}

\section{Proofs for Section \ref{sec:tvlip}}
\label{app:TV}

\begin{proof}[Proof of Lemma \ref{lem:M.conditional.iid}]
    Let $X\sim \mu$ and let $(X^i)_{i=1}^n$ be the i.i.d.\ sample selected according to $\mu$.
    Set $I$ to be the random set of indices $i \leq n$ for which $X^i_{\pa(k)}\in c_{\pa(k)}$; thus $ n\cdot \hat\mu^\mathcal{M}(c_{\pa(k)}) = |I|$ (by Lemma \ref{lem:mubAwelldefined}).

    First note that, conditionally on $|I|=m$, the vector $(X^i_k)_{i\in I}$ has the same distribution as that of an i.i.d.\ sample of the distribution $\mu(dx_k \mid c_{\pa(k)})$ (this follows exactly as in the proof of Lemma \ref{lem:conditional.iid}).
    In particular,  conditionally on $|I|=m$, $\alpha:=\frac{1}{|I|}\sum_{i\in I} \delta_{X^i_k}$ has the same distribution as the empirical measure of $\mu(dx_k\mid c_{\pa(k)})$ with sample size $m$.

    Next, denote by $\psi\colon \X_k\to\mathcal{M}_k$ the projection of each cell to its center.
    Thus $\mu^{\mathcal{M}}(d x_k \mid x_{\pa(k)}) $ is the push-forward of $\mu(dx_k\mid c_{\pa(k)})$ under $\psi$ and, similarly, $\hat\mu^\mathcal{M}(dx_k\mid c_{\pa(k)})$ is the push-forward of $\alpha$ under $\psi$  (again using Lemma \ref{lem:mubAwelldefined}).
    Since the push-forward of the empirical measure is the empirical measure of the push-forward of the measure, 
    the claim follows.
\end{proof}

\begin{proof}[Proof of Lemma \ref{lem:average.on.cells}]
    For every $1\leq k \leq K$, define the kernel $R\colon \X_{k:K} \to\mathcal{P}(\X_{k:K})$ via
    \[ R_k(x_{k:K},d\tilde{x}_{k:K}) := \nu_{k|c_{k}(x_{k})}(d\tilde x_{k}) \cdots \nu_{K|c_{K}(x_{K})}(d\tilde x_K). \]
    Thus $f^\nu(x) = \int f(\tilde x )\, R_1(x,d \tilde x)$.
    We claim that, for every $k=1,\dots, K$ and $x_{1:k}\in \X_{1:k}$,\footnote{For the following, we note that while $\nu^{\mathcal{M}}$ is supported on the midpoints of the cells, we naturally extend its kernels in a constant fashion to $\mathcal{X}$. That is, $\nu^{\mathcal{M}}(dx_{k:K}\mid x_{1:k-1}) = \nu^{\mathcal{M}}(dx_{k:K}\mid c_{1:k-1}(x_{1:k-1}))$.}
    \begin{align*}
    &\iiint f(x_{1:k-1}, \tilde x_{k:K}) \, R_{k}(x_{k:K},d \tilde x_{k:K}) \, \nu^\mathcal{M}(dx_{k:K} \mid x_{1:k-1})\,\nu^{b\mathcal{A}}(d x_{1:k-1}) \\
    &=  \iiint  f(x_{1:k}, \tilde x_{k+1:K}) \, R_{k+1}(x_{k+1:K},d \tilde x_{k+1:K}) \, \nu^\mathcal{M}(dx_{k+1:K} \mid x_{1:k})\,\nu^{b\mathcal{A}}(d x_{1:k}) . 
    \end{align*}
    If that claim is true, the proof of the lemma follows from an iterative application, noting that left hand side  is equal to $\int f^\nu \,d\nu^\mathcal{M}$ for $k=1$ and the right hand side is equal to $\int f \,d\nu^{b\mathcal{A}}$ for $k=K$.  

    To prove the claim, fix $1\leq k< K$ and $x_{1:k-1}\in \X_{1:k-1}$, and note that
    \begin{align*}
    ({\rm I})&:=R_{k}(x_{k:K},d \tilde x_{k:K}) \, \nu^\mathcal{M}(dx_{k:K} \mid x_{1:k-1}) \\
    &= R_{k+1}(x_{k+1:K},d \tilde x_{k+1:K}) \, \nu_{k|c_k(x_k)}(d\tilde x_k) \, \nu^\mathcal{M}(dx_{k+1:K} \mid x_{1:k}) \, \nu^\mathcal{M}(dx_k \mid x_{1:k-1}) \\
    &= R_{k+1}(x_{k+1:K},d \tilde x_{k+1:K})  \, \nu^\mathcal{M}(dx_{k+1:K} \mid x_{1:k}) \, \nu_{k|c_k(x_k)}(d\tilde x_k) \, \nu^\mathcal{M}(dx_k \mid x_{1:k-1})
    \end{align*}
    Next, recall that $ \nu^\mathcal{M}(dx_k \mid x_{1:k-1})$ is the projection of $\nu^{b\mathcal{A}}(dx_k \mid x_{1:k-1})$ to the centres of the cells (i.e., $\nu^{b\mathcal{A}}(c_k(x_k) \mid x_{1:k-1}) = \nu^{\mathcal{M}}(\{x_k\} \mid x_{1:k-1})$) and
   \begin{align*}
   \nu^{b\mathcal{A}}(d \tilde x_k |x_{1:k-1})
   &=  \sum_{ c_k\in\mathcal{A}_k}  \nu^{b\mathcal{A}}(c_k \mid x_{1:k-1}) \nu_{k|c_k}(d \tilde{x}_k ) \\
    &=  \nu_{k|c_k(x_k)}(d\tilde x_k) \, \nu^{\mathcal{M}}(d x_k |x_{1:k-1}) .
    \end{align*}
    Finally, since $\nu_{k|c_k(x_k)}(d\tilde x_k) $ is supported on the same cell that $x_k$ belongs to (by definition) and  $\nu^\mathcal{M}(dx_{k+1:K} \mid x_{1:k})$ is constant in $x_{1:k}$ on each fixed cell (which implies $\nu^\mathcal{M}(dx_{k+1:K} \mid x_{1:k}) = \nu^\mathcal{M}(dx_{k+1:K} \mid x_{1:{k-1}}, \tilde{x}_k)$ for each $\tilde{x}_k \in c_k(x_k)$), it follows that
    \begin{align*}
    ({\rm I}) &= R_{k+1}(x_{k+1:K},d \tilde x_{k+1:K})  \, \nu^\mathcal{M}(dx_{k+1:K} \mid x_{1:k}) \, \nu^{b\mathcal{A}}(d x_k |x_{1:k-1}).
    \end{align*}  
    This shows our claim, and thus also completes the proof of the lemma.
\end{proof}

\section{Supplementary facts}
\label{app:supp}
The lemma below is a natural corollary of a version of the Gilbert-Varshamov bound, which we state here for reference:
\begin{lemma}\label{lem:appendixbound}
	Let $A$ and $B$ be two finite sets with $n := |A|$ and $m := |B| \geq 3$. Then, there exists an absolute constant $C > 0$ and a set $\mathcal{F} \subseteq \{f : A \rightarrow B\}$ such that $|\mathcal{F}| \geq \frac{1}{2} m^{C n}$ and for all $f, g \in \mathcal{F}$, we have
	\begin{align}\label{eq:desiredGV}
		\left| \{a \in A \mid f(a) \neq g(a)\}\right| \geq n/8.
	\end{align}
    \begin{proof}
		A version of the Gilbert-Varshamov bound (also called sphere-covering bound, see \cite[Theorem 5.2.4]{ling2004coding}) with minimal distance $z := \lceil n/8 \rceil$ yields the existence of a set $\mathcal{F}$ satisfying \eqref{eq:desiredGV} and
		\[
		|\mathcal{F}| \geq \frac{m^n}{\sum_{j=0}^{z-1} \binom{n}{j}(m-1)^j} =: \frac{1}{M}.
		\]
		Defining $X^1, \dots, X^n$ as i.i.d.~Bernoulli variables with $\mathbb{P}(X^1=1) =p:=\frac{m-1}{m}$, we get
		\begin{align*}
			M&= \sum_{j=0}^{z-1} \binom{n}{j}(m-1)^j (1/m)^n
			=\sum_{j=0}^{z-1} \binom{n}{j}  p^j (1-p)^{n-j} 
			=\mathbb{P}\left( \sum_{i=1}^n X^i \leq z-1\right) .
        \end{align*}
        Since $z-1\leq n/8$, 
        \begin{align*}
        M	 
			&\leq \mathbb{P}\left( \Big|\sum_{i=1}^n X^i - \mathbb{E}[X^i] \Big| \geq np- \frac{n}{8}\right)
            \leq 2\exp\left( \frac{- C (np- \frac{n}{8})^2}{n \, (1/\sqrt{\log(m)})^2} \right),
		\end{align*}
        where the second inequality follows from Hoeffding’s inequality for sub-Gaussian random variables (see, e.g., \cite[Theorem 2.6.2]{vershynin2018high}), observing that the sub-Gaussian norm of each $X^i-p$  is at most $1/\sqrt{\log m}$, and $C$ denotes an absolute constant.

        Finally, since $np- \frac{n}{8}\geq  \frac{n}{4}$, it follows that 
        $M \leq 2\exp( \frac{- Cn^2}{16n \, 1/\log(m)})  = 2m^{-\frac{C}{16}n}$, which completes the proof.
	\end{proof}
\end{lemma}

\section{Numerical experiments}
\subsection{Numerical lower bound for $\hat\mu^{\mathcal{A}}$ rate}
 \label{app:lower.muA}
\begin{figure}[t]
    \centering
    \includegraphics[width=0.5\linewidth]{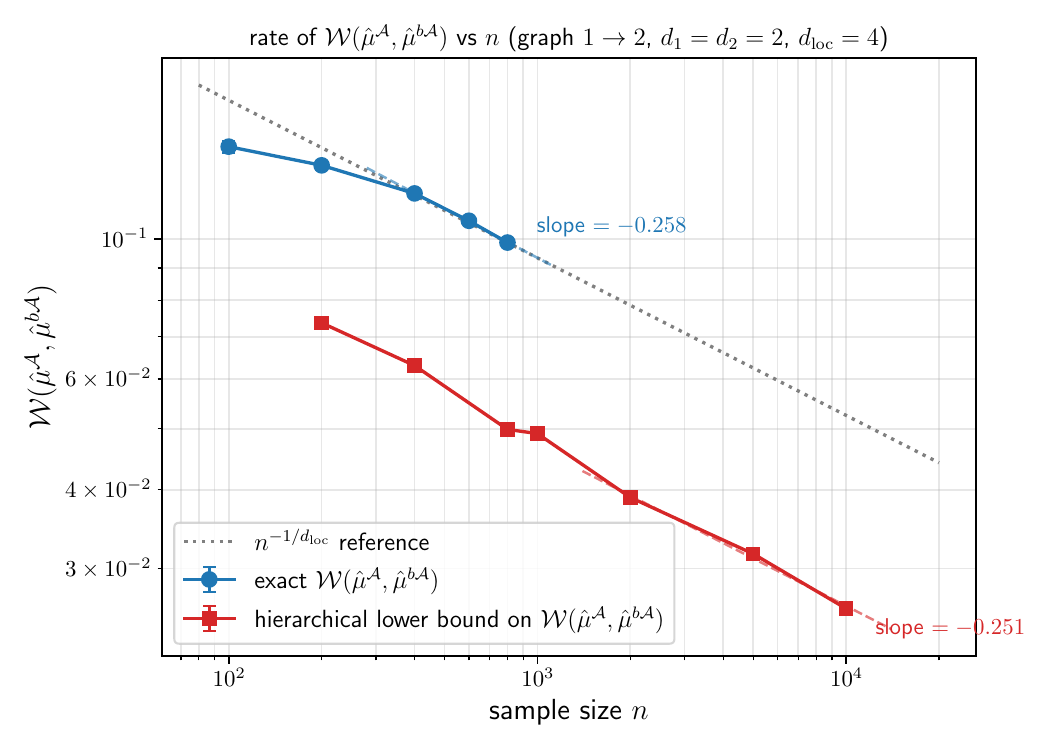}
    \caption{Illustration of a numerical example indicating that the upper rate established for the estimator $\hat\mu^{\mathcal{A}}$ may be sharp.}
    \label{fig:lowerW}
\end{figure}
While a full treatment for lower bounds with Wasserstein-Lipschitz kernels requires new ideas and goes beyond the scope of the current work, we briefly illustrate that it is possible that the rate $n^{-1/d_{\rm loc}}$ of Theorem \ref{thm:estWLip} may indeed be sharp. More precisely, we provide a numerical example indicating that at least the estimator $\hat\mu^{\mathcal{A}}$ seems to yield nothing better than the rate $n^{-1/d_{\rm loc}}$. This also underlines the importance of the new estimator in Section \ref{sec:tvlip}, cf.~Remark \ref{rem:sec23}.

For notational simplicity, the following example is stated on $[-1.25, 1.25]^4$ instead of $[0, 1]^4$, which obviously only matters up to constants in the observed rates.
The example is for the graph with two nodes, $1\rightarrow 2$ and $d_1=d_2=2$ (we ignore log factors in the following). 
Let $X_1 =(X_{1,1},X_{1,2})$ and $\varepsilon=(\varepsilon_1,\varepsilon_2)$ be two independent random vectors uniformly distributed in $[-1,1]^2$, and set
\[X_2 := X_1 + \left(\begin{array}{l}
    \sigma(X_{1 ,1}) \cdot \varepsilon_1\\
    \sigma(X_{1, 2}) \cdot \varepsilon_2
\end{array}\right),\quad\text{where }\sigma(x): = \min\{0.5, 0.25 + \min\{|x-1|, |x+1|\}\}.
\]
Denote by $\mu$ the law of $X=(X_1,X_2)$.
Because $\sigma$ is Lipschitz and bounded below, one can check that the forward and backward kernels are Lipschitz w.r.t.~$\TV$; thus $\mu$ satisfies Assumption \ref{ass:TV.Lip} and we know that $\W(\hat\mu^{b\mathcal{A}}, \mu) \lesssim n^{-2/(2+4)} = n^{-1/3}$ in this case. 

 Therefore, when evaluating $\W(\hat\mu^{b\mathcal{A}}, \hat\mu^{\mathcal{A}})$, the slow rate given by $\W(\hat\mu^{\mathcal{A}}, \mu)$ should dominate the distance. That is, if we observe $\W(\hat\mu^{b\mathcal{A}}, \hat\mu^{\mathcal{A}}) \approx n^{-1/4}$, it indicates that $\W(\hat\mu^{\mathcal{A}}, \mu) \approx n^{-1/4}$ as well. While it is numerically challenging to evaluate the Wasserstein distance precisely (and using regularization, as with Sinkhorn's algorithm, would distort the rate), we showcase both the exact Wasserstein distance up to $n=1000$ and a lower bound based on approximating the dual $\W$ with simpler Lipschitz functions, loosely inspired by the grid structure in \cite{dudley1969speed}.\footnote{More precisely, to lower bound $\W(\mu, \nu)$, we use different partitions of the support and define Lipschitz functions to take values $\pm 1$ on the centers of each cell based on whether $\mu-\nu$ has positive or negative mass on this cell.} The results are shown in Figure \ref{fig:lowerW}. The observed trends therein are indeed compatible with the case that the established rates for $\hat\mu^{\mathcal{A}}$ are sharp.

\subsection{Experiments regarding circumventing Assumption \ref{ass:graph_struc}} \label{app:graphass}
We briefly explore how costly it is to circumvent Assumption \ref{ass:graph_struc} by adding edges to the graph in terms of the average increase in $d_{\rm loc}$. While it is algorithmically hard to find the best way to do so (as adding edges to get rid of one collider may introduce new colliders, etc.), up to a reasonable number of edges (like $K \leq 20$), a straightforward implementation via backward induction  to iteratively add edges (with pruning to avoid exploring redundant graphs) still works well. We outline numerical results in this regard for randomly generated graphs (roughly speaking, directed acyclic Erd\"{o}s--Rényi graphs): For a number of nodes $K$, we first sample each $d_k$ uniformly in $\{1, \ldots, 5\}$, then $p$ uniformly in $[0, 1/4]$ and then we add each edge $i\rightarrow j$ with probability $p$. The resulting initial $d_{\rm loc}$ is denoted by $d_{\rm loc}^{\rm init}$ and the best $d_{\rm loc}$ after adding edges to the graph to make it collider free is denoted by $d_{\rm loc}^*$. The average values over 1000 simulations for each $K \in \{4, 7, 10, 13, 16, 19\}$ are reported in Table \ref{tab:er-results}. We find that the difference is modest, yet becomes significant for a larger number of nodes $K$.

\begin{table}[t]
\centering
\caption{Average cost of circumventing the Assumption \ref{ass:graph_struc} for directed acyclic Erdös--R\'enyi graphs with varying number of nodes $K$ and 1000 simulations per $K$. The time column shows how long (for a single graph) it takes to find the best edges to add using backward induction. For the given graphs, the initial $d_{\rm loc}^{\rm init}$ is significantly lower than the global $d$. On average, the local dimension only modestly increases (to $d_{\rm loc}^*$) when the graph is made compatible with Assumption \ref{ass:graph_struc}, though the difference becomes larger with increasing $K$.}
\begin{tabular}{rlllll}
\toprule
$K$ & time (s) & $d$ & $d_{\rm loc}^{\rm init}$ & $d_{\rm loc}^{*}$ & difference \\
\midrule
 4 & 0.000 & 11.973 &  5.662 &  5.663 & 0.001 \\
 7 & 0.000 & 20.940 &  7.845 &  7.865 & 0.020 \\
10 & 0.001 & 30.030 &  9.901 & 10.030 & 0.129 \\
13 & 0.006 & 39.092 & 11.756 & 12.139 & 0.383 \\
16 & 0.039 & 47.970 & 13.800 & 14.761 & 0.961 \\
19 & 0.274 & 56.793 & 15.468 & 17.446 & 1.978 \\
\bottomrule
\end{tabular}
\label{tab:er-results}
\end{table}

\end{document}